%% file: SasakiInstantons.tex
\documentclass{article}
\input{pacotes.tex}     
\input{init.tex}

\begin{document}
\title{Instantons on Sasakian  \texorpdfstring{$7$}{Lg}-manifolds}
\author{Luis E. Portilla \quad\&\quad Henrique N. Sá  Earp}
\affil{University of Campinas (Unicamp)}
\date{\today}
\newgeometry{top=2cm,bottom=2cm,right=2.54cm,left=2.54cm}
\maketitle
\begin{abstract}    
  We study a natural \emph{contact instanton (CI)} equation on gauge fields over  $7$-dimensional Sasakian manifolds, which is closely related both to the transverse Hermitian Yang-Mills (HYM) condition and the   $\rG_2$-instanton equation. We obtain, by Fredholm theory, a finite-dimensional local model for the moduli space of irreducible solutions. Following the approach by Baraglia and Hekmati in $5$ dimensions \cites{Baraglia2016}, we derive cohomological conditions for smoothness, and we express its dimension in terms of the index of a transverse elliptic operator. Finally we show that the moduli space of selfdual (SD) contact instantons is Kähler, in the Sasakian case. 
  
  As an instance of concrete interest, we specialise to transversely holomorphic Sasakian bundles over contact Calabi-Yau $7$-manifolds, as studied by Calvo-Andrade,  Rodríguez and Sá Earp \cites{Calvo-Andrade2016}, and we show that in this context the notions of contact instanton,  integrable $\rG_2$-instanton and HYM connection coincide.     
\end{abstract}
\setcounter{tocdepth}{2}
\begin{adjustwidth}{0.85cm}{0.85cm}
\tableofcontents
\end{adjustwidth}
\newpage

\newgeometry{top=2.54cm,bottom=2.54cm,right=2.54cm,left=2.54cm}
\section{Introduction}
\label{sec:1} 
We describe the  moduli space of solutions to a natural gauge-theoretic equation, on a suitable class of vector bundles over a Sasakian manifold $M$.  We recall that on a $d$-dimensional manifolds M $d\geq 4$, if  $F_A$ is the curvature of a connection $A$ the classical  instanton  equation $ \pm F_A=\ast F_A$  \cites{Donaldson1990},   can be generalised  relatively to an appropriate $(d-4)$-form $\sigma$   by \cites{Donaldson1998,Tian2000}
\begin{equation}
\label{eq:introInstEquation}
    \lambda F_A=\ast(\sigma\wedge F_A), 
\end{equation}
for eigenvalues $\lambda$ of the operator $\ast(\sigma\wedge\cdot):\Omega^2(M)\to\Omega^2(M)$. On a  Sasakian $7$-manifolds,  the contact structure $\eta$ induces a natural $3$-form $\sigma=\eta\wedge d\eta$ in which $\ast(\sigma\wedge\cdot)$ provides  $\lambda\in\{\pm1, -2\}$ (see \S \ref{sec: eigenspaces of Lphi}). Indeed, in that context we argue that the meaningful condition is the selfdual contact instanton (SDCI) equation. We adopt  the approach of Baraglia and Hekmati  \cites{Baraglia2016}, who described  the moduli space of contact instantons on contact metric $5$-manifolds, studying obstructions to smoothness and determining its expected dimension as the index of an elliptic operator transverse to the Reeb foliation. We will see that these results admit precise analogues in the appropriate $7$-dimensional setup, while some new distinct gauge-theoretic phenomena also occur.
\subsection{Instantons in \texorpdfstring{$7$}{Lg}-dimensions: contact, Hermitian Yang-Mills and \texorpdfstring{$\rG_2$}{Lg}}
Let $(M^{2n+1},\eta,\xi)$  denote a contact   manifold, with contact form  $\eta$ and  Reeb vector field $\xi$ \cites{Boyer2008,blair2010contact}. Then the natural $(2n-3)$-form $\sigma=\eta\wedge(d\eta)^{n-2}$ provides an instance of \eqref{eq:introInstEquation}: 
\begin{equation}
\label{eq:ContactInstEquation}
       \pm F_A=\ast(\eta\wedge(d\eta)^{n-2}\wedge F_A).
\end{equation}
Solutions of \eqref{eq:ContactInstEquation} are said to be \emph{self dual contact instantons (SDCI)} for $\lambda=1$ (respectively \emph{anti-selfdual contact instanton (ASDCI)}  for $\lambda=-1$).  

When the contact manifold  is endowed in addition with a Sasakian structure, namely  an integrable transverse complex structure  $\Phi$ and a compatible metric $g$, Biswas    \cites{Biswas2010} proposes a natural  notion of Sasakian holomorphic structure for complex vector bundles $E\to M$ (see Appendix \ref{apendixA}). We recall that a connection $A$ on a
complex vector bundle over a Kähler manifold is said to be \emph{Hermitian Yang-Mills (HYM)} if
\begin{equation}
\label{eq: tHYM}
    \hat{F}_A := (F_A,\omega)= 0
    \qandq
    F^{ 0,2}_A = 0.
\end{equation}
This notion indeed extends  to Sasakian bundles, by taking $\omega := d\eta\in\Omega^{1,1}(M)$ as a `transverse Kähler form', and defining HYM connections to be the solutions of \eqref{eq: tHYM} in that sense. The well-known concept of \emph{Chern  connection} also extends,  namely as a connection mutually compatible with the holomorphic structure (\emph{integrable}) and a given Hermitian bundle metric (\emph{unitary}), see \cite{Biswas2010}*{\S~3}. 

An important class of Sasakian manifolds are those endowed with a \emph{contact Calabi-Yau (cCY)} structure [Definition \ref{def:cCY}], the Riemannian metric  of which have transverse holonomy $\SU(n)$, in the sense of foliations, corresponding to the existence of a global  transverse holomorphic volume form $\epsilon\in\Omega^{n,0}(M)$  \cites{habib2015some}. Furthermore, when $n=3$, such cCY   $7$-manifolds are naturally endowed with a $\rG_2$-structure defined by the $3$-form
\begin{equation}
\label{eq:G2equation}
    \varphi :=\eta\wedge d\eta+\Im(\varepsilon),
\end{equation} 
which is \emph{cocalibrated}, in the sense that its Hodge dual $\psi:=\ast_g\varphi $ is closed under the de Rham differential. When a $3$-form $\varphi $ on a $7$-manifold defines a $\rG_2$-structure, the instanton condition \eqref{eq:introInstEquation} for $\sigma=\varphi $ and $\lambda=1$ is referred to as the  \emph{$\rG_2$-instanton equation}. On holomorphic Sasakian bundles over closed cCY  $7$-manifolds, it has the distinctive feature that integrable solutions are indeed Yang-Mills critical points, even though the $\rG_2$-structure has torsion  \cites{Calvo-Andrade2016}.  

As a first point of interest,   we will describe exactly how the three above concepts of instanton interrelate on suitable Sasakian $7$-manifolds. The contact instanton equation \eqref{eq:ContactInstEquation} in this case is determined by the natural $3$-form 
\begin{equation}
\label{eq:varphi}
    \sigma :=\eta\wedge d\eta.
\end{equation}
The operator $\ast(\sigma\wedge\cdot)$ splits the space of $2$-forms into  $\{\pm 1,-2\}$-eigenspaces  [\S \ref{sec: eigenspaces of Lphi}]. However, the $(-2)$-eigenspace  is $1$-dimensional, spanned by $d\eta$ and rather uninteresting, so we will focus on the $(\pm 1)$-eigenspaces,   which in some sense still signify the instanton equation \eqref{eq:introInstEquation} as an (anti-) selfduality condition.  The SDCI and HYM conditions are related as follows.

\begin{proposition}
\label{prop:instantonRelation}
    Let $\bE\to M$ be a complex Sasakian  bundle over $(M^{7},\eta,\xi,g,\Phi)$. If  a connection $A\in\cA(\bE)$ is SDCI, then $A$ is a transverse  HYM connection. Conversely, if moreover $\bE$ underlies a holomorphic Sasakian bundle $\cE$ [Definition \ref{def:SasakianoHol}] and $A\in\cA(\cE)$ is an integrable HYM connection, then $A$ is a SDCI [\S \ref{sec:prop instantonRelation}].
\end{proposition}

In particular, on a contact Calabi-Yau $7$-manifold, the three notions of instanton are related in the following ways:
\begin{theorem}
\label{thm:instantonRelation}
  Let $\cE$ be a holomorphic Sasakian bundle over a contact Calabi-Yau  manifold $(M^7,\eta,\xi,g,\Phi)$ endowed with its natural $\rG_2$-structure \eqref{eq:G2equation}; then the following hold:
    \begin{enumerate}[(i)]
        \item 
        Every solution of the contact instanton equation  $\pm F_A=\ast(\sigma\wedge F_A)$ is also a solution of $\ast(\varphi \wedge F_A)=\pm F_A$, i.e., every contact instanton is a $\rG_2$-instanton  [Proposition \ref{prop:ContactthenG2}].

        \item  A Chern connection is a $\rG_2$-instanton if, and only if, it is a contact instanton  [Proposition  \ref{prop:HYM-ContactCorrespondence}].

        \item A Chern connection is HYM if, and only if, it is a $\rG_2$-instanton \cite{Calvo-Andrade2016}*{Lemma ~21}.
    \end{enumerate}
  In particular, among Chern connections, the three notions are equivalent.
\end{theorem}
\subsection{Local model of the moduli space}
Our first main result is a complete description of  the local deformation theory of   $7$-dimensional Sasakian contact instantons. Let $\bE\to M$ be a Sasakian  vector bundle with compact, connected, semi-simple structure group $\rG$,  and denote by  $\fg_E$ its adjoint bundle and by  $\Omega^k(\fg_E)$ the $\fg_E$-valued $k$-forms on $M$. The  operator  
\begin{equation}
\label{eq:Phioperator}
    L_{\sigma}:=\ast(\sigma\wedge\cdot)
    \colon \Omega^2(\fg_E)
    \to\Omega^2(\fg_E) 
\end{equation}
induces the irreducible splitting [cf. \eqref{eq:omega2}] 
$$
\Omega^2( \fg_E)
=(\underbrace{\Omega_6^2
\oplus\Omega_8^2
\oplus\Omega_1^2}_{\Omega^2_H}\oplus\;\Omega_V^2)(\fg_E)
$$
where $\Omega_6^2( \fg_E)$, $\Omega_8^2( \fg_E)$ and $\Omega_1^2( \fg_E)$ are the eigenspaces associated to $-1$, $1$ and $-2$, respectively. In Proposition \ref{prop:flatness}, we show that Chern connections with curvature in $\Omega^2_6(\fg_E)$, i.e. anti-selfdual contact instantons,  are necessarily flat. Hence the meaningful notion of Sasakian instanton in this case is that of SDCI, with curvature in $\Omega^2_8(\fg_E)$, which is the kernel of the projection map
\begin{equation}
\label{eq:map P}
    p\co \Omega^2_H(\fg_E)
    \to 
    \Omega^2_{6\oplus 1}(\fg_E)
    :=
    (\Omega_6^2\oplus\Omega_1^2)(\fg_E).    
\end{equation}

The Fréchet  Lie group $\cG$    of smooth gauge transformations acts smoothly on the space  $\cA(\bE)$ of connections on $\bE$,   and the topological quotient  $\cB :=\cA /\cG $ is a  Hausdorff  space.  We denote by  $\cB^\ast \subset\cB$  the open subspace  of irreducible connections, by $\cM\subset\cB $  the set of gauge  equivalence classes of solutions to the SDCI equation:
\begin{equation}
\label{eq:modulispace}
    \cM:=
    \{[A]\in\cB\;\vert\; p(F_A)=0\},
\end{equation}
and, accordingly, by $\cM^\ast \subset \cM$ its irreducible stratum. 
Linearising the SDCI condition, in terms of the  projection \eqref{eq:map P}, we introduce: 
\begin{equation}
    \label{eq:d7}
    d_7:=p\circ d_A\colon\Omega^1_H(\fg_E)\to \Omega^2_{6\oplus 1}(\fg_E).
\end{equation} 
In Proposition \ref{prop:deformationSpace}, we will see that a local model for the moduli space  $\cM^\ast$  of  SDCI is given by  the cohomology group  $H^1(\rC):=\frac{\Ker(d_7)}{\im(d_A)}$  of the  deformation  complex 
\begin{equation}
\label{eq:Complexd7}
    \begin{tikzcd}
        \rC^\bullet\colon 0\arrow[r] 
        &\Omega^0(\fg_E) \arrow[r,"d_A"]
        & \Omega^1(\fg_E) \arrow[r, "d_7"]
        & \Omega^2_{6\oplus 1}(\fg_E) \arrow[r]
        &  0
    \end{tikzcd}. 
\end{equation}
This complex, however, is not elliptic, and in order to compute the dimension of $H^1(\rC)$ we resort to an auxiliary construction, studied in \S \ref{sec:moduli space CI}. We introduce the quotient spaces of $k$-forms modulo the Lie algebra ideal $I$ generated by   
$\Omega^2_8(\fg_E)$:
$$
\rL^k:=\frac{\Omega^k(\fg_E)}{I},
\quad
k=0,1,2,3,
\qwithq
I:=\left\langle\Omega^2_8(\fg_E)\right\rangle
\subset
\left(\Omega^\bullet(\fg_E),\wedge\right).
$$ 
For a natural choice of differentials $D_k$, the $\rL^k$ spaces fit in a complex [Proposition \ref{prop:Lidentification}]: 
\begin{equation}
\label{eq:complexDeformation} 
    \begin{tikzcd}
        \rL^\bullet \colon 0\arrow[r] 
        &\rL^0 \arrow[r,"D_0"]
        & \rL^1 \arrow[r, "D_1"]
        & \rL^2 \arrow[r, "D_2"]
        &\rL^3 \arrow[r]
        &  0
    \end{tikzcd}. 
\end{equation}
We denote by $H^\bullet:=H^\bullet(\rL)$ the cohomology of \eqref{eq:complexDeformation}, and indeed the first cohomology group $H^1$ is isomorphic to the infinitesimal deformations of $[A]\in\cM$ as a contact instanton. 

Relatively to the  Reeb orbits, a $S^1$-invariant differential form    $\alpha\in\Omega^\bullet(M)$ is called \emph{basic}. The graded ring  $\Omega^\bullet_B(M)$  of basic forms inherits a natural \emph{basic de Rham differential}    
$$
d_B:=d\vert_{\Omega^k_B(M)}
\co \Omega^k_B(M)\to \Omega^{k+1}_B(M),
$$
and the cohomology of $d_B$ is referred to as the \emph{basic de Rham cohomology}. Restricting the  differentials $D_k$ in \eqref{eq:complexDeformation} to basic forms in $\rL^k=\Omega^k(M)/\langle \Omega^2_8\rangle$, we obtain a basic complex [Proposition \ref{prop:Lidentification}]:
\begin{equation}
\label{eq:BasicComplexDeformation2}
    \begin{tikzcd}
        \rL^\bullet_B\colon 0\arrow[r] 
        &\Omega^0_B( \fg_E) \arrow[r,"D_B"]
        & \Omega^1_B( \fg_E) \arrow[r, "D_B"]
        & (\Omega^2_{6\oplus1})_B (\fg_E) \arrow[r] 
        &  0
    \end{tikzcd}. 
\end{equation}
We denote by $H^\bullet_B:=H^\bullet(\rL_B)$ the corresponding \emph{basic cohomology}. 
 If $A$ is a contact instanton, the \emph{transverse index} of $A$ is defined as the index of the basic complex   \eqref{eq:BasicComplexDeformation2}: 
$$
\indT(A)= \text{dim}(H^0_B)-\text{dim}(H^1_B)+\text{dim}(H^2_B).
$$
In particular, when $A$ is irreducible,  $\indT(A)= \text{dim}(H^2_B) -\text{dim}(H^1_B)$. 
In summary, we formulate a local model for the moduli space  $\cM^\ast$ in terms of basic cohomology:
\begin{theorem}
\label{thm:Intro} 
    Let $\bE\to M$ be a Sasakian $G$-bundle over a  closed  connected Sasakian manifold $(M^7, \eta,\xi,g,\Phi)$, with adjoint bundle  $\fg_E$, and denote by $\cM^\ast$  the moduli space \eqref{eq:modulispace} of irreducible selfudal contact instantons, i.e. solutions of \eqref{eq:introInstEquation} for $\lambda=1$:
    $$
    F_A=\ast(\eta\wedge d\eta\wedge F_A).
    $$
    Then, the following hold:
\begin{enumerate}[(i)]
\item 
    The tangent space of  $\cM^\ast$    at    $[A]$, i.e. the space  of infinitesimal deformations of $[A]$ as a contact instanton,  is isomorphic to the finite-dimensional cohomology group $H^1(\rC):=\frac{\Ker(d_7)}{\im(d_A)}$ of the complex \eqref{eq:Complexd7}, and
        $$
        \dim(T_{[A]}\cM^\ast)=\dim(H^1(\rC)).
        $$ 
\item 
    The dimension of  $\cM^\ast$ near $[A]$ can be computed from the  cohomology   of the basic complex \eqref{eq:BasicComplexDeformation2}, which is elliptic transversely to the Reeb foliation, namely there is an isomorphism  $H^1\cong H^1_B$, where  $H^1_B$ is the cohomology of  \eqref{eq:BasicComplexDeformation2}:
        $$
        \dim(T_{[A]}\cM^\ast)
        = \dim(H^2_B) -\indT (A).
        $$
 
        \item 
        The local model of $\cM^*$ is cut out  as the zero set of an obstruction map [Definition \ref{eq:obstructionmap}], which vanishes  precisely when  $H^2_B=0$ [Proposition \ref{prop:vanishingH2}]. Thus,   for an irreducible selfdual contact instanton $A$ such that  $ H^2_B=0$, $\cM^\ast$ is smooth near $A$, with finite dimension  $\dim \cM^\ast=-\indT (A)$ [Corollary \ref{cor:ModuliSuave}]. 
    \end{enumerate}
\end{theorem} 

\begin{remark}
\label{rem:theorem intro}
    Parts $(i)$ and   $(ii)$  in Theorem \ref{thm:Intro} establish somewhat independently that the tangent space near an irreducible contact instanton is finite-dimensional, since it occurs as the first cohomology group in both complexes \eqref{eq:complexDeformation} and \eqref{eq:BasicComplexDeformation2}. However, in terms of the obstruction theory, we learn something finer from $(ii)$ and $(iii)$. In the context of $(i)$, the moduli space near an \emph{acyclic} point, i.e. $h^0(\rC)=h^2(\rC)=0$ in \eqref{eq:Complexd7}, would be necessarily $0$-dimensional, whereas the complex \eqref{eq:BasicComplexDeformation2} in terms of basic cohomology is merely transverse-elliptic, hence the moduli space near an acyclic smooth point, with $h^0_B=h^2_B=0$, can in principle have nonzero dimension $-\indT (A)$.
\end{remark}

\begin{remark}
    For most steps in our arguments, it suffices to assume $M$  compact  and connected, with possibly nontrivial boundary. However, in Theorem  \ref{thm:instantonRelation}--$(iii)$, taken from \cite{Calvo-Andrade2016}*{Lemma ~21}, and in Propositions   \ref{lem:TopologicalCharge} and \ref{prop:omegaMap}, one actually needs  $M$ to be closed.
\end{remark}

\begin{remark}
    For a Sasakian $7$-manifold  with  positive transverse scalar curvature, the second basic cohomology group $H_B^2 $ should vanish, and therefore $\cM^*$ is smooth. This is announced here as Conjecture \ref{conj: t-sc=0 => H_B^2=0}, to be expanded in subsequent work. In particular,   $\indT(A)=-\text{dim}(H^1_B)$, by Theorem \ref{thm:Intro}--(ii). 
\end{remark}

Regarding the various notions of instanton related by Theorem \ref{thm:instantonRelation}, 
Theorem \ref{thm:Intro} has the following significance: 

\begin{corollary}
    Let $\cE\to M$ be a holomorphic Sasakian bundle  over a  $7$-dimensional cCY  manifold $(M^7,\eta,\xi,g,\Phi)$,  endowed with its natural $\rG_2$-structure \eqref{eq:G2equation}. Among Chern connections in  $\cA(\bE)$, the three notions of instanton coincide: SDCI, HYM connections and $\rG_2$-instantons. The complex \eqref{eq:Complexd7}  describes their local deformations, and Theorem \ref{thm:Intro} describes their moduli space.
\end{corollary}
\subsection{Geometric structures on the moduli space}
We are furthermore interested in geometric structures on the SDCI moduli space $\cM^\ast$ [cf. \eqref{eq:modulispace}], in the hope that this  might lead in the future to constructing new invariants of Sasakian manifolds. 

By means of comparison, under suitable assumptions, the moduli space of ASD contact instantons on a Sasakian $5$-manifold $(M^5,\eta,\xi,g,\Phi)$ is K\"ahler \cite{Baraglia2016}*{\S~4.3}, and moreover hyper-Kähler in the transverse Calabi-Yau case. We will see that, even though this does not generalise verbatim to the $7$--dimensional contact Calabi-Yau case,  $\cM^\ast$ still inherits a natural Kähler structure: 
\begin{theorem} 
\label{thm:M Kahler}
    In the situation of Theorem \ref{thm:Intro}, the moduli space $\cM^\ast$  of irreducible SDCI carries a natural Riemannian metric [cf. \eqref{eq:inner product in TAalpha}], a complex structure $\cJ$ induced by the transverse complex structure $J:=\Phi\vert_{\rH}$, and a symplectic $2$-form $\Omega$, such that $(\cM^*,\cJ,\Omega)$ is a Kähler manifold.
\end{theorem}
\bigskip
\noindent\textbf{Outline:}  
    In \S  \ref{sec: eigenspaces of Lphi}, we describe the local splitting \eqref{eq:omega2}  of  $\Omega^2(M)$  under the contact structure and the operator $L_\sigma:=\ast(\sigma\wedge\cdot)$ from \eqref{eq:Phioperator}. Another natural decomposition of  $\Omega^2(\fg_E)$ comes from the transverse complex structure induced by $\Phi\in\End(TM)$,  and both are related by \eqref{eq:decompositionPQ}. Furthermore, the endomorphism $\Phi$  provides a notion of transverse holomorphicity for complex vector bundles over Sasakian manifolds   [Appendix \ref{apendixA}],   hence also notions of  unitary and integrable connections. We show, in Proposition \ref{prop:flatness}, that imposing these conditions on  a connection forces its curvature component   in     $\Omega_6^2(\fg_E)$ to vanish, hence the only nontrivial theory in this context is the SDCI case. 

    Parts \emph{(i)} and \emph{(ii)} of Theorem \ref{thm:instantonRelation} are proven respectively  in Propositions \ref{prop:ContactthenG2} and \ref{prop:HYM-ContactCorrespondence}. The proof of Theorem \ref{thm:Intro} is organised as follows: \emph{(i)} is the content of Proposition \ref{prop:deformationSpace}, which uses an auxiliary elliptic complex [Proposition \ref{prop:ElliptExtendComplex}] to establish that this local model has finite dimension; \emph{(ii)} is an immediate consequence [Corollary \ref{coro:H1=Hb}] of Proposition \ref{prop:gysinSequence}; and  \emph{(iii)} requires a thorough  study of the moduli space of the obstruction theory of SDCI, under the $5$-dimensional paradigm from \cites{Baraglia2016}, culminating in Proposition \ref{prop:vanishingH2}. Finally, in \S \ref{sec:Geometry} we study the geometry of the moduli space, 
    showing that $\cM^\ast$ inherits a complex structure [ Proposition \ref{prop:J integrable}] and a Kähler $2$-form [Proposition \ref{prop:Omega closed}], thus proving Theorem \ref{thm:M Kahler}.

\bigskip
\noindent\textbf{Funding:} This work was supported by the Higher Education Improvement Coordination - Brasil (CAPES) - Finance Code 001; the Brazilian National Council for Scientific and Technological Development (CNPq) [141215/2019-4 to L.P., 307217/2017-5 to H.S.E.]; and São Paulo Research Foundation (Fapesp) [2017/20007-0, 2018/21391-1 to H.S.E.].
\section{Preliminaries on Sasakian geometry}
We follow the standard references for Sasakian geometry   \cites{Boyer2008,blair2010contact}. A Sasakian structure on a smooth manifold $M^{2n+1}$ is a quadruple $( \eta,\xi, g,\Phi)$ such that $(M, g)$ is a Riemannian manifold, $(M, \eta)$ is a contact manifold with Reeb field $\xi$   and  $\Phi\in \End(TM)$ is a transverse complex structure, satisfying the following compatibility relations:  
\begin{multicols}{2}
\raggedcolumns
\begin{enumerate}[(i)]
    \item $g(\xi,\xi)=1,$
    \item $\Phi\circ\Phi=-\rI_{TM}+\eta\otimes \xi,$
    \item $g(\Phi X,\Phi Y )= g(X,Y)-\eta(X)\eta(Y), $
    \item $\nabla^g_X\xi=-\Phi X,$
    \item $(\nabla^g_X\Phi)(Y)= g(X,Y)\xi-\eta(Y) X,$
\end{enumerate} 
\end{multicols}
\noindent where $ X, Y$ are vector fields on $M$ and $\nabla^g$ is the Levi-Civita connection of
$g$. In that case we say $(M,\eta,\xi, g,\Phi)$ is a \emph{Sasakian manifold}, equivalently, $M$ is Sasakian manifold if and only if the metric cone $(\R_{+}\times M,dr^{2}+r^{2}\cdot g)$ is Kähler. 

\subsection{Horizontal forms and the contact instanton equation}

\subsubsection{Vertical and horizontal forms}
Hereinafter we set  $\sigma:=\eta\wedge d\eta$. For a $p$-form $\alpha$ and a vector field $v$ on $M$, the metric $g$ is compatible,  in the sense that $\ast(v\lrcorner \alpha)=v^{\flat}\wedge\ast\alpha$, where $v^{\flat}:=g(v,\cdot)$ and $\ast$ is the Hodge operator of $g$.  Notice that $\xi^{\flat}=\eta$, thus applying the above formula to $v=\xi$ and $\alpha=\ast\beta$ for a $p$-form $\beta$, we obtain
    \[
    \ast(\xi\lrcorner \ast\beta)=\xi^{\flat}\wedge(\ast^{2}\beta)=(-1)^{p(n-p)}\eta\wedge \beta,
    \]
or, equivalently, 
\begin{equation}
\label{eq:compaEtaystar}
    i_{\xi}(\ast \beta)=(-1)^{p}\ast(\eta\wedge\beta),
    \quad\forall \beta\in\Omega^p(M).
\end{equation}
Furthermore, the contact structure induces a natural operator 
\begin{equation}
    \label{eq:projection T}
    T:= \eta\wedge i_\xi(\cdot)\colon\Omega^{p}(M)\to\Omega^{p}(M), 
\end{equation}
which is a projection:
\begin{align*}
    T^{2}(\alpha)
    &=(\eta\wedge i_\xi) (\eta\wedge i_\xi\alpha)
    =(\eta\wedge i_\xi) (i_\xi\eta\wedge\alpha-i_\xi(\eta\wedge\alpha))
    =(\eta\wedge i_\xi)(\alpha)\\
    &=T(\alpha).
\end{align*}
As such, it splits any $\alpha\in\Omega^\bullet(M)$ into \emph{horizontal} and \emph{vertical} components:
\begin{equation}
\label{eq:HorizontalVerticalDecomposition}
    \alpha
    =(1-\eta\wedge i_\xi)\alpha+\eta\wedge i_{\xi}\alpha=\alpha_{H}+\alpha_{V} 
    = \underbrace{i_{\xi}(\eta\wedge\alpha)}_{\alpha_{H}}+\underbrace{ \eta\wedge i_{\xi}\alpha}_{\alpha_{V}}, 
\end{equation} 
this provides a splitting $\Omega^\bullet(M)=\Omega^\bullet_H(M)\oplus\Omega^\bullet_V(M)$, where $\Omega^\bullet_H(M)$ and  $\Omega^\bullet_V(M)$  are the horizontal  and  vertical parts, respectively. 

\subsubsection{The contact instanton equation in $7$ dimensions}

From now on, unless otherwise stated, we will fix $\dim M=7$. Equation \eqref{eq:HorizontalVerticalDecomposition} suggests a natural `instanton equation' as follows: consider  $\alpha\in\Omega^2(M)$, applying  the contraction  \eqref{eq:compaEtaystar} to  $d\eta\wedge\alpha\in\Omega^4(M)$, we obtain 
$$ 
i_\xi(\ast(d\eta\wedge\alpha))
=\ast \left( \sigma\wedge \alpha \right),
\qwithq
\sigma=\eta\wedge d\eta.
$$ 
This motivates the introduction of operator $L_{\sigma} \colon\Omega^2_H(M)\to\Omega^2_H(M)$ in  \eqref{eq:Phioperator}. 
In next section  we show  that $\pm 1$ are eigenvalues of $L_\sigma$ and that $L_{\sigma\vert_{\Omega^2_V(M)}}=0$, and this extends  in  the natural way to $\Omega^2(\fg_E)$. If $\alpha=F_A$ is   the curvature of a connection $A$ on a suitable vector bundle $\bE\to M$, a natural instance of the \emph{contact instanton equation} is
$
\ast(\sigma\wedge F_A)=i_\xi(\ast( d\eta\wedge F_A))=\pm F_A ,
$
or,  equivalently,  
\begin{equation}
\label{eq:contacIns1}
    \ast F_A=\pm \sigma\wedge F_A.
\end{equation} 
Notice that  $\eta\wedge(d\eta)^3\neq 0$ is a volume form on $M$, and $\ast(\sigma\wedge d\eta)=cd\eta$ for a constant $c\neq 1$. This contrasts with the   $5$-dimensional case in \cites{Baraglia2016} and classical $4$-dimensional gauge theory, in both of which $d\eta$ is selfdual.

\subsubsection{Transverse  Hodge star  operator and Sasakian Kähler identities}
\label{sec: Sasakian Kähler identities}
  Let us briefly describe the transverse complex geometry  on a Sasakian manifold $(M,\eta,\xi,g,\Phi)$, referring to \cites{Boyer2008} for a thorough treatment. 
  Relatively to the Reeb foliation, the usual Hodge star induces a \emph{transverse  Hodge star} operator $ \ast_T\colon \Omega_H^k(M) \to \Omega_H^{m-(k+1)}(M)  $ by the formula \cite{tondeur2012foliations}*{\S~12}
\begin{equation}
\label{eq:transverse star}
    \ast_T(\beta)  =(-1)^{m-1-k}\ast(\beta\wedge\eta).
\end{equation}
Both operators are compatible, in the sense that 
\begin{equation}
\label{eq:star transverse star}
    \ast\alpha= \ast_T\alpha\wedge \eta,
    \quad\forall
    \alpha\in\Omega^\bullet_H(M).
\end{equation} 

The following result appears in the literature in various guises [ibid.], but since it will play a central role in \S\ref{sec:khaler M}, we include a full proof here.  
\begin{lemma}
\label{lem:eq star_T}
    Let $(M,\eta,\xi,g,\Phi)$ be a Sasakian manifold and denote by $J:=\Phi\vert_{\rH}$ the restriction of  $\Phi\in \End(TM)$ to the horizontal distribution, cf. \eqref{eq:TangentDecomposition}. If $\beta\in\Omega^1_H(\fg)$ is a transverse $1$--form,  then 
    $$
    \ast_T(\beta)= \frac{1}{2}J\beta\wedge \omega^2.
    $$
\end{lemma}
\begin{proof}
    By excess in degree, for any $\beta\in\Omega^1_H(\fg)$  the transverse $7$--form    $\beta\wedge\omega^3$ is zero, so 
\begin{align*}
    0 &=\beta^\sharp\lrcorner(\beta\wedge\omega^3)\\
    &=\beta(\beta^\sharp)\wedge\omega^3-3\beta\wedge(\omega^2\wedge(\beta^\sharp\lrcorner\omega)).
\end{align*}
    The conclusion now follows by a short computation: 
\begin{align*}
    \beta\wedge\ast_T\beta&=(\beta,\beta)\frac{\omega^3}{3!}=\frac{1}{6}\beta(\beta^\sharp)\omega^3 =\frac{1}{2}\beta\wedge\omega^2\wedge(\beta^\sharp\lrcorner\omega)\\
    &=\frac{1}{2}\beta\wedge\omega^2\wedge J\beta.
\qedhere
\end{align*}
\end{proof}
 Now, acting on $(p,q)$-forms, we have well-defined operators 
\begin{equation}
    \label{eq: partial_B}
\begin{array}{lll}
  \partial_B :\Omega^{p,q}_B\to\Omega^{p+1,q}_B   &\text{and } &
  \bar{\partial}_B :\Omega^{p,q}_B\to\Omega^{p,q+1}_B,
 \end{array}
\end{equation}
which naturally extend to $\fg$--valued forms.  In terms of the transverse Hodge star \eqref{eq:transverse star}, the operators  \eqref{eq: partial_B} have adjoints
\begin{equation}
    \label{eq:adjoin partial_B}
    \begin{array}{ll }
  \partial_B^\ast\colon       &\Omega^{p,q}_B\to\Omega^{p-1,q}_B  \\  [4pt]
                              &  \partial_B^\ast:= -\ast_T \partial_B^\ast\ast_T   
 \end{array}
 \qandq
 \begin{array}{ll }
     \bar{\partial}_B^\ast\colon &\Omega^{p,q}_B\to\Omega^{p,q-1}_B\\ [4pt] 
                                 &\bar{\partial}_B^\ast:=-\ast_T\bar{\partial}_B^\ast\ast_T.
 \end{array}
\end{equation}
  Let $\langle\cdot,\cdot\rangle_\fg $ be an invariant metric on $\fg$, an inner product   on $\Omega^\bullet(\fg_E)$ is defined  by: 
\begin{equation}
\label{eq:innerProduct}
    (\alpha,\beta)_M=\displaystyle\int_M\langle\alpha\wedge\ast\beta\rangle_\fg.
\end{equation}
  Note that from \eqref{eq:star transverse star} the inner product in  \eqref{eq:innerProduct} can be    rewritten as 
$ 
  (\alpha,\beta)_M=\displaystyle\int_M\langle\alpha\wedge\ast_T\beta\rangle\wedge\eta.
$ 
  Denoting the exterior product with $\omega:=d\eta$ by
  \begin{equation}
      \label{eq:mult omega}
      \rL_\omega\colon \alpha\in \Omega^k(\cF_\xi)\mapsto \alpha\wedge \omega\in \Omega^{k+2}(\cF_\xi),
  \end{equation} 
  its adjoint with respect to the transverse Hodge star \eqref{eq:transverse star}  is the \emph{transverse Lefschetz operator}
  \begin{equation}
      \label{eq:adjoin mult omega}
  \Lambda:=\rL_\omega^\ast=-\ast_T \rL_\omega\ast_T.
  \end{equation}
For later use in \S \ref{sec:Geometry}, we recall the Sasakian Kähler identities:
\begin{lemma}[{\cite{Boyer2008}*{Lemma~7.2.7}}]
\label{lem:deltaB caudrado}
On a Sasakian manifold, the following properties hold:
\begin{multicols}{2}
\raggedcolumns
\begin{enumerate}[(i)]
\item
  $[\Lambda, \partial_B]=-\ii \bar{\partial}_B^\ast$
\item   
  $[\Lambda, \bar{\partial}_B]= \ii\partial_B^\ast$
\item
  $\partial_B\bar{\partial}_B^\ast+\bar{\partial}_B^\ast\partial_B = \partial_B^\ast\bar{\partial}_B+\bar{\partial}_B\partial_B^\ast=0$
\item
  $\partial_B^\ast\partial_B+\partial_B\partial_B^\ast=\bar{\partial}_B^\ast\bar{\partial}_B+\bar{\partial}_B\bar{\partial}_B^\ast$
\end{enumerate}
\end{multicols}
\begin{itemize}
\item[(v)]
  Defining $\Delta_{\partial_B} :=\partial_B^\ast\partial_B +\partial_B\partial_B^\ast$ and $\Delta_{\bar{\partial}_B} :=\bar{\partial}_B^\ast\bar{\partial}_B +\bar{\partial}_B\bar{\partial}\textbf{}^\ast$, then:
$$
  \Delta_{d_B}=2\Delta_{\bar{\partial}_B}=2\Delta_{ \partial_B}.
$$
\end{itemize}
\end{lemma}
\subsubsection{Eigenspaces of  \texorpdfstring{$L_{\sigma}=\ast(\sigma\wedge\cdot)$}{Lg} from the contact structure}
\label{sec: eigenspaces of Lphi}
We will now examine how  \eqref{eq:contacIns1} splits into components according to the eigenspaces of  $L_\sigma$ defined in \eqref{eq:Phioperator}. Let $(x_1,\dots,x_7)$ be Sasakian Darboux coordinates on $M$ \cite{blair2010contact}*{Theorem ~3.1}, such that the contact form $\eta$ is given by      
$$
\eta=dx^7-(x_4dx^1+x_5dx^2+x_6dx^3).
$$
Let $X_i:=\frac{\partial}{\partial x_i}$ 
and  $dx^{i_1\cdots i_k}:=dx^{i_1}\wedge \cdots\wedge dx^{i_k}$; in particular, $X_7=\xi$ is the Reeb vector field, and the transverse symplectic $2$-form is expressed by 
$\omega:=d\eta=dx^{14}+dx^{25}+dx^{36}$.
In these coordinates, the projection $T$ defined in \eqref{eq:projection T} acts as follows:
$$
\begin{array}{ll}
    T(dx^{ij})=\eta\wedge i_\xi(dx^{ij})=0,
    &\qforq 1\leq i<j\leq 6;\\
    T(dx^{i7})=\eta\wedge i_\xi(dx^{i7})=-\eta\wedge dx^i,
    &\qforq
    i=1,\cdots, 6.
\end{array}
$$
Therefore, the decomposition \eqref{eq:HorizontalVerticalDecomposition} determines a $15$-dimensional \emph{horizontal space} $\Omega^2_{H}$  and a $6$-dimensional \emph{vertical space} $\Omega^2_{V}$:  
\begin{equation}
\label{eq:H-Vdecomposition}
    \Omega^2(M)= \underbrace{
    \f{Span}\{  dx^{ij}\vert 1\leq i<j\leq 6 \}}_{\Omega^2_H:=\Omega^2_{15}}\oplus \underbrace{\f{Span}\{ dx^{i7}\vert i=1\cdots 6\} }_{\Omega^2_{V}}.
\end{equation}
Moreover, $L_\sigma$ acts on the horizontal space 
$$
\Omega^2_{H}=\f{Span}\{ 
dx^{12},dx^{13},dx^{14},dx^{15},dx^{16},
dx^{23},dx^{24},dx^{25},dx^{26},
dx^{34},dx^{35},dx^{36},
dx^{45},dx^{46},
dx^{56}\},
$$  
as follows:
$$
\begin{array}{*{5}c}
L_{\sigma}(dx^{12})=dx^{45} & L_{\sigma}(dx^{13})=dx^{46} & L_{\sigma}(dx^{15})=dx^{24} & L_{\sigma}(dx^{14})= -(dx^{36}+dx^{25}) &  L_{\sigma}(dx^{45})=dx^{12}  \\[3pt]
L_{\sigma}(dx^{16})=dx^{34} & L_{\sigma}(dx^{23})=dx^{56} & L_{\sigma}(dx^{24})=dx^{15} & L_{\sigma}(dx^{25})=-(dx^{36}+dx^{14})  & L_{\sigma}(dx^{46})=dx^{13}  \\[3pt]
L_{\sigma}(dx^{26})=dx^{35} & L_{\sigma}(dx^{34})=dx^{16} & L_{\sigma}(dx^{35})=dx^{26} & L_{\sigma}(dx^{36})=-(dx^{25}+dx^{14})  & L_{\sigma}(dx^{56})=dx^{23}. 
\end{array}
$$
For immediate convenience, let us fix the following notation: 
$$ 
v:=dx^{14}+dx^{25}+dx^{36}  
$$
$$
\begin{array}{*{3}c}
v_1 := dx^{12}-dx^{45} & v_2 := dx^{15}-dx^{24}  & v_3 := dx^{13}-dx^{46} \\[3pt]
v_4 := dx^{16}-dx^{34} & v_5 := dx^{23}-dx^{56}  & v_6 := dx^{26}-dx^{35}
\end{array}
$$
$$
\begin{array}{*{4}c}
w_1 := dx^{12}+dx^{45}  & w_2 := dx^{15}+dx^{24}  & w_3 := dx^{13}+dx^{46}  & w_4 := dx^{16}+dx^{34}\\[3pt]
w_5 := dx^{23}+dx^{56}  & w_6 := dx^{26}+dx^{35}  & w_7 :=dx^{14}- dx^{36}  & w_8 := dx^{25}- dx^{36}. 
\end{array}
$$
It is easy to check that 
\[
\begin{array}{cl}
L_\sigma(v_i)=-v_i,     &\qforq i=1,\cdots,6;  \\
L_\sigma(w_j)=w_j,      &\qforq j=1,\cdots,8;  \\
L_\sigma(v)=-2v.
\end{array}
\]
Hence  the operator $L_\sigma$ defined in \eqref{eq:Phioperator} splits $\Omega^2_H(M)$ into eigenspaces associated to  $\{-2 ,-1 ,1\} $,   respectively: 
\begin{equation}
\label{eq:omega2} 
    \Omega^2_6:= \Omega^{2,-}_{H,6} =  \f{Span}\{  v_1,\dots,v_6\},
    \quad 
    \Omega^2_8:=\Omega^{2,+}_{H,8} =  \f{Span}\{ w_1\dots,w_8  \},
    \quad 
    \Omega^2_1:=\Omega^2_{H,1}=\f{Span}\{ v\}.
    \quad 
\end{equation}
Therefore $2$-forms decompose irreducibly as 
\begin{equation}
\label{eq:2formsDecomposition}
   \Omega^2(M) =\Omega^{2}_{1} \oplus\Omega^{2}_{6} \oplus \Omega^{2}_{8} \oplus\Omega^{2}_V.
\end{equation}
Of course, this decomposition extends naturally to $\fg_E$-valued  $2$-forms.
\begin{lemma}
\label{lem:orthogonality}
    The decompositions   \eqref{eq:HorizontalVerticalDecomposition} and \eqref{eq:2formsDecomposition}   are orthogonal  with respect to the   inner product   \eqref{eq:innerProduct}.
\begin{proof}
    For $\alpha\in \Omega^2_H(M)$ and $\beta\in \Omega^2_V(M)$ we have
    \begin{align*}
        (\alpha,\beta)
        &=(i_\xi(\eta\wedge\alpha),\beta)
        =(\alpha,\eta\wedge i_\xi\beta)
        =-a_i(\alpha,\eta\wedge dx^i)\\
        &=(i_\xi\alpha,\beta)=0.
    \end{align*}
  Moreover, the operator $\ast(\sigma\wedge(\cdot))\colon \Omega^2_H(M)\to\Omega^2_H(M)$ is self-adjoint:  for $\alpha,\beta\in\Omega^2_H(M)$, 
\begin{align*}
  \left(\alpha,\ast(\sigma\wedge\beta)\right)_M &= \int_M 
  \alpha\wedge\ast^2(\sigma\wedge \beta)\dvol 
  =\int_M 
  \beta\wedge \sigma\wedge\alpha \dvol \\
  &=\int_M \beta\wedge\ast^2(\sigma\wedge\alpha) \dvol\\ 
  &=\left(\ast(\sigma\wedge\alpha),\beta\right)_M,
\end{align*}
  hence its eigenspaces are  orthogonal.
\end{proof}
\end{lemma}
\subsection{Splitting of complexified differential forms}
\label{sec:complex spliting}
  We establish some notation and elementary facts about the complexified tangent bundle, which are largely adapted from \cites{Biswas2010} and reviewed in Appendix \ref{apendixA}.  The contact structure splits the tangent bundle as $TM=  H\oplus N_\xi$ \eqref{eq:TangentDecomposition}, where $H=\Ker(\eta)$ and $N_\xi$ is the real line bundle spanned by the Reeb field $\xi$. The transverse complex structure $\Phi$  satisfies $(\Phi\vert_{ H })^{2}=-1$,  so the  eigenvalues of the complexified operator $\Phi^{\C}$ are $\pm\bi$, with $\bi:=\sqrt{-1}$. The complexification $H_\C:=H\otimes_\R\C$ splits   as 
  $H_\C = H^{1,0}\oplus  H ^{0,1}$ \eqref{eq:Hpq}, so we obtain a decomposition of direct sum of vector bundles \eqref{eq:Quebra1} 
$$
  \Lambda^{d}(H_\C)^\ast  
  =\bigoplus_{i=0}^{d} (H^{i,d-i})^\ast.
$$
  This induces the decomposition of vector bundles \eqref{eq:decompostion k forms} 
\[
  \Omega^{d}(M)=\left(\bigoplus_{i=0}^{d} \Omega_{H }^{i,d-i}(M) \right) \oplus\left(\eta\otimes\left( \bigoplus_{j=0}^{d-1}\Omega_{H }^{j,d-j-1}(M)\right)\right), 
\]
  where $\Omega^{p,q}_{H}(M):=\Gamma(M,\Lambda^p(H_\C)^\ast\otimes\Lambda^q(H_\C)^\ast)$. Now,  let us study more closely the space of $2$-forms, from the `transverse complex' point of view. Still in local Darboux coordinates $(x_1,\cdots,x_7)$, we denote the transverse complex coordinates by  
\begin{equation}
\label{eq:zi}
    z_1:=x_1+\ii x_4, 
    \quad z_2:=x_2+\ii x_5 
    \qandq z_3:=x_3+\ii x_6.
\end{equation}  
  We will denote, as usual, $dz^j:=dx^j+\ii dx^{j+3}$ and $d\bar{z}^j:=dx^j-\ii dx^{j+3}$ for $j=1,2,3.$   In terms of the bases $\{v_i\}_{i=1}^6$ and $\{w_i\}_{i=1}^8$, from \eqref{eq:omega2}, the space $\Omega^{2,0}(M)$ is locally spanned by 
\begin{align*}
    dz^1\wedge dz^2 &= (dx^{12}-dx^{45})+\ii (dx^{15}-dx^{24}) = v_1 + \ii  v_2 \\
    dz^1\wedge dz^3 &= (dx^{13}-dx^{46})+\ii (dx^{16}-dx^{34}) = v_3 + \ii  v_4 \\
    dz^2\wedge dz^3 &= (dx^{23}-dx^{56})+\ii (dx^{26}-dx^{35}) = v_5 + \ii  v_6,
\end{align*}
we also compute  
\begin{align*}
    dz^1\wedge d\zbar^2  &= (dx^{12}+dx^{45}) - \ii (dx^{15} + dx^{24}) = w_1 - \ii w_2 \\
    dz^1\wedge d\zbar^3  &= (dx^{13}+dx^{46}) - \ii (dx^{16} + dx^{34}) = w_3 - \ii w_4 \\
    dz^2\wedge d\zbar^3  &= (dx^{23}+dx^{56}) - \ii (dx^{26} + dx^{35}) = w_5 - \ii w_6 \\
    dz^1\wedge d\zbar^1  &= -2\ii dx^{14} \\
    dz^2\wedge d\zbar^2  &= -2\ii dx^{25} \\
    dz^3\wedge d\zbar^3  &= -2\ii dx^{36} \\ 
    dz^2\wedge d\zbar^1  &= (dx^{12}+dx^{45}) - \ii (dx^{15} + dx^{24}) = -(w_1 + \ii w_2) \\
    dz^3\wedge d\zbar^1  &= (dx^{13}+dx^{46}) - \ii (dx^{16} + dx^{34}) = -(w_3 + \ii w_4)\\
    dz^3\wedge d\zbar^2  &= (dx^{23}+dx^{56}) - \ii (dx^{26} + dx^{35}) = -(w_5 + \ii w_6).
\end{align*}
Since $\omega:=d\eta $ is nowhere-vanishing and has type $(1,1)$ \cite{Biswas2010}*{Corollary ~3.1}, it determines an orthogonal complement  $\Omega^{1,1}_{\perp}(M)$ in 
$$
\Omega^{1,1}(M)=\left( \Omega^{0}(M)\cdot d\eta\right) \oplus\Omega^{1,1}_{\perp}(M),
$$ 
 which is expressed in local Darboux coordinates by 
\begin{align*}
    \Omega^{1,1}_{\perp}(M)
    &= \f{Span}\{\Re(dz^i\wedge \zbar^j) ,  -\im(dz^i\wedge d\zbar^j) , w_7 , w_8\vert \; 1\leq i<j\leq 3 \}\\
    &\cong \Omega^2_8.
\end{align*}
Therefore we obtain two decompositions for the horizontal $2$-forms: 
\begin{equation}
\label{eq:decompositionPQ}
    \begin{matrix}
        \Omega^{2}_H(M)=  & \underbrace{\Omega^{2,0}(M) \oplus \Omega^{0,2}(M) } & \oplus & \Omega^{0}(M)\cdot d\eta                  & \oplus                                  &  \Omega^{1,1}_{\perp}(M) \\
                   &  \rotatebox{90}{\scalebox{1}[1]{$\cong$}}&         & \rotatebox{90}{\scalebox{1}[1]{$\cong$}} &                                         & \rotatebox{90}{\scalebox{1}[1]{$\cong$}} \\
                   & \Omega^2_6                               &         & \Omega^{2}_{ 1}                          &                                         &
                   \Omega^{2 }_{ 8}  
    \end{matrix} 
\end{equation}
Note that  $\Omega^{0,2}(M)$ is spanned by
\begin{align*}
    d\zbar^1\wedge d\zbar^2 
    &= (dx^{12}-dx^{45})-\ii (dx^{15}-dx^{24}) = v_1 - \ii  v_2 ,\\
    d\zbar^1\wedge d\zbar^3 
    &= (dx^{13}-dx^{46})-\ii (dx^{16}-dx^{34}) = v_3 - \ii  v_4 ,\\
    d\zbar^2\wedge d\zbar^3 
    &= (dx^{23}-dx^{56})-\ii (dx^{26}-dx^{35}) 
    = v_5 - \ii  v_6 ,
\end{align*}
consistently with the fact that $\Omega^{0,2}(M)\cong\ol{\Omega^{2,0}}(M)$. Still by inspection, we have: 
\begin{equation}
    \label{eq:wi}
\begin{array}{*{2}c}
w_1 = \frac{1}{2}(dz^1\wedge d\zbar^2-dz^2\wedge d\zbar^1) , & w_2 = \frac{\ii}{2}(dz^1\wedge d\zbar^2+dz^2\wedge d\zbar^1)\\[3pt]
w_3 = \frac{1}{2}(dz^1\wedge d\zbar^3-dz^3\wedge d\zbar^1),  & w_4 = \frac{\ii}{2}(dz^1\wedge d\zbar^3+dz^3\wedge d\zbar^1),\\[3pt] 
w_5 = \frac{1}{2}(dz^2\wedge d\zbar^3-dz^3\wedge d\zbar^2) , & w_6 = \frac{\ii}{2}(dz^2\wedge d\zbar^3+dz^3\wedge d\zbar^2),\\[3pt] 
w_7 = \frac{\ii}{2}(dz^1\wedge d\zbar^1-dz^3\wedge d\zbar^3),& w_8 = \frac{\ii}{2}(dz^2\wedge d\zbar^2-dz^3\wedge d\zbar^3)    
\end{array} 
\end{equation}  
\begin{equation}
\label{eq:omega}
\omega=\frac{\ii}{2}(dz^1\wedge d\zbar^1+dz^2\wedge d\zbar^2+dz^3\wedge d\zbar^3),
\end{equation} 
\begin{equation}
    \label{eq:vi}
    \begin{array}{*{3}c}
v_1 = \frac{1}{2}(dz^1\wedge dz^2+d\zbar^1\wedge d\zbar^2) , & v_2 = \frac{\ii}{2}(d\zbar^1\wedge d\zbar^2-dz^1\wedge dz^2), & v_3 =\frac{1}{2}(dz^1\wedge dz^3+d\zbar^1\wedge d\zbar^3) ,\\[3pt]
v_4 = \frac{\ii}{2}(d\zbar^1\wedge d\zbar^3-dz^1\wedge dz^3), & v_5 =  \frac{1}{2}(dz^2\wedge dz^3+d\zbar^2\wedge d\zbar^3), & v_6 = \frac{\ii}{2}(d\zbar^2\wedge d\zbar^3-dz^2\wedge dz^3).
\end{array}
\end{equation} 
From the expressions  \eqref{eq:wi} and \eqref{eq:vi}, we obtain immediately:
\begin{lemma} 
\label{lem:SDcharacterization} 
    Let $\Omega^2_8$ and $\Omega^2_6$  be the eigenspaces of the operator $L_\sigma$ from \eqref{eq:Phioperator}, associated to the eigenvalues $1$ and $-1$, respectively [cf.  \eqref{eq:omega2}]; then 
\begin{enumerate}[(i)]
\item
  A $2$-form $\alpha$ belongs to $\Omega^2_8$    if, and only if, $\alpha$ it is a real form  of type $(1,1)$ and  orthogonal to $\omega$.
\item
  A $2$-form $\alpha$ belongs to $\Omega^2_6$ if, and only if,  $\alpha=\beta +\ol{\beta} $, for some $\beta$ of type $(2,0)$.
\end{enumerate} 
\end{lemma}
\subsection{Proof of Proposition \ref{prop:instantonRelation}}
\label{sec:prop instantonRelation}
  The proof is based on Lemma \ref{lem:SDcharacterization}, which yields a characterisation curvature forms in $\Omega^2_6(\fg_E)$ and $\Omega^2_8(\fg_E)$. We denote by $\Lambda^k(T^\ast M)_\C$ the complexification of   $\Lambda^k(T^\ast M)$. Let also: 
\begin{equation}
\label{eq: pq forms}
    \Omega^k(\fg_E)
    :=\Gamma(\Lambda^k(T^\ast M)_\C \otimes \fg_E)
    \qandq
    \Omega^{p,q}(\fg_E)
    :=\Gamma(\Lambda^{p,q}(T^\ast M)_\C \otimes \fg_E).
\end{equation}
  The conjugation on $\Lambda^k(T^\ast M)_\C$  naturally induce a conjugation  on   $\Omega^k(T^\ast M)_\C\otimes\fg_E$ such that $\ol{ \Omega^{p,q}(\fg_E)}=\Omega^{q,p}(\fg_E)$. We have the following characterisation of SDCI solutions:   
\begin{proposition} 
\label{prop: selfdual = HYM}
    Let $\bE\to M$ be a Sasakian $G$-bundle with adjoint bundle $\fg_E$. A connection $A\in\cA$ satisfies the SDCI equation $\ast F_A=  \sigma\wedge F_A$ if, and only if: 
\begin{equation}
    \label{eq: selfdual = HYM}     
    F_A\in\Omega^{1,1}(\fg_E) 
    \qandq  
    \hat{F}_A:=\langle F_A,\omega \rangle=0.
\end{equation}
    In that case, moreover, $\ol{F}_A=F_A.$
\end{proposition}  
\begin{proof}
 That $A$ is a SDCI if, and only if,  \eqref{eq: selfdual = HYM} holds is an immediate consequence of Lemma \ref{lem:SDcharacterization}--(i) and  \eqref{eq:decompositionPQ}, since $\Omega^{1,1}_\perp(\fg_E)\cong\Omega^2_8(\fg_E)$. Now, assuming that is the case, the  reality condition is manifest in local Darboux coordinates, cf.  \eqref{eq:zi}. Finally, we see in \eqref{eq:wi}  that the generators $\{w_i\}_{i=1}^8$ of $\Omega^2_8$ are all real: $\ol{w_i}=w_i$.
\end{proof}
\begin{proof}[Proof of Proposition \ref{prop:instantonRelation}] 
  If $A$ is a SDCI on a Sasakian bundle $\bE\to M$ [Definition \ref{def:fibradoSasakiano}],  it  follows from   \eqref{eq:decompositionPQ} that $F_A\in\Omega^{1,1}_8(M)$, in particular  $F_A^{0,2}=0$. Furthermore, Proposition \ref{prop: selfdual = HYM} gives immediately $\hat{F}_A=0$, hence $A$ is HYM. Conversely, if moreover $\bE\to M$ underlies a holomorphic bundle $\cE \to M$ [Definition \ref{def:SasakianoHol}] and  $A$ is a an integrable HYM connection, then $\hat{F}_A=0$ and $F^{0,2}=0$. Since $A$ is also unitary, the curvature $F_A$ is of type $(1,1)$, cf.  \cite[Proposition~2.1.56]{Donaldson1990}, and we conclude form \eqref{eq:decompositionPQ} that $A$ is a SDCI.
  \end{proof}
\section{Gauge theory on \texorpdfstring{$7$}{Lg}-dimensional Sasakian manifolds}
\label{sec: GT on Sasakian 7-mfds}
  In $4$-dimensional gauge theory, reversing orientation of the base manifold interchanges SD and ASD connections, thus the two theories are equivalent. On the other hand, in the $5$-dimensional contact case, the SD and ASD equations are studied separately \cites{Baraglia2016} and some differences appear; for instance, the  coboundary map in the long exact sequence in  \cite[Proposition~3.3]{Baraglia2016} is only zero in the selfdual case.  On a merely contact manifold $(M^7,\eta)$,  one could a priori study connections with curvature in $\Omega^2_6$ or $\Omega^2_8$. However, on a Sasakian manifold $(M^7,\eta,g,\Phi)$, there is a natural choice favouring the SDCI  equation ($\lambda= 1$)  in \eqref{eq:ContactInstEquation}:
$$
  F_A=\ast(\eta\wedge d\eta \wedge F_A).
$$
  That the alternative theory is  trivial follows immediately from   Lemma \ref{lem:SDcharacterization}--$(ii)$:
\begin{proposition}
\label{prop:flatness}
  Let  $\cE\to M^7$ be a holomorphic Sasakian vector bundle [Definition \ref{def:SasakianoHol}]. A Chern connection $A\in\cA(\cE) $ such that $F_A\in\Omega^2_6(\fg_E)$ is necessarily flat.
\end{proposition} 
\subsection{Infinitesimal deformations of contact instantons}
\label{sec:infinitesimal deformations}
  The main purpose  this Section is to prove Theorem \ref{thm:Intro}--$(i)$. Namely,  the linearisation of the moduli space of  SDCI $\cM^\ast$ is given by  the cohomology group $H^1(\rC):=\frac{\Ker(d_7)}{\im(d_A)}$  of the  deformation  complex 
\[
  C\bullet:\quad 0\to\Omega^0(\fg_E) 
  \xrightarrow[{  }]{ d_A}\Omega^1(\fg_E) \xrightarrow[{  }]{d_7}\Omega_{6\oplus 1}(\fg_E)\to 0.
\]
  We will show that, even though the complex  $C\bullet$ is manifestly not elliptic, one can establish that $ H^1(\rC)$ is finite-dimensional by means of an auxiliary extended complex. Finally, we will relate contact instantons and $\rG_2$-instantons in the $7$-dimensional contact Calabi-Yau setting, as outlined in \cites{Calvo-Andrade2016}. Under the hypotheses of  Theorem \ref{thm:Intro},  we describe the linearisation  of the moduli space of SDCI over a Sasakian $7$-manifold, following the approach of \cites{itoh1983moduli}.
\begin{proposition}
\label{prop:kerd7}
  An element $\alpha\in \Omega^1(\fg_E)$ gives an infinitesimal SDCI deformation if, and only if, $\alpha\in\Ker(d_7)$.
\end{proposition}
\begin{proof}
  Given a connection $A$ on the $G$-bundle  $E\to M$, the induced covariant exterior derivative $d_A\colon \Omega^k(\fg_E)\to \Omega^{k+1}(\fg_E)$ squares to an algebraic curvature operator:
\begin{equation}
\label{eq:d^2}
  d_A \circ d_A(\beta)
  =[\beta\wedge F_A],\quad \beta\in  \Omega^k(\fg_E).
\end{equation}
  Furthermore, the transversal complex structure splits $d_A=\partial_A +\ol{\partial}_A$, acting on  $\Omega^{p,q}_H(\fg_E)$ [cf. \eqref{eq: pq forms}] according to bi-degree:
\begin{equation}
\label{eq:partialA}
\begin{diagram}
  \node{}
  \node[2]{\Omega^{p+1,q}_H(\fg_E)}\\
  \node{\Omega^{p,q}_H(\fg_E)} \arrow{ene,t}{\partial_A } \arrow{ese,r}{\ol{\partial}_A }   \\
  \node{ }  \node[2]{\Omega^{p,q+1}_H(\fg_E)}
\end{diagram}.    
\end{equation}
  Let $\alpha\in\Omega^{p,q}(\fg_E)$, and note that 
$$
  d_A\circ d_A\alpha = \partial_A\partial_A\alpha  +\ol{\partial}_A\ol{\partial}_A\alpha+(\partial_A\ol{\partial}_A +\ol{\partial}_A\partial_A)\alpha. 
$$
  If $A$ is SDCI, i.e.,  $F_A$ is of type $(1,1)$ and orthogonal to $\omega$ [cf. Proposition \ref{prop: selfdual = HYM}], replacing  \eqref{eq:d^2} in the above formula and comparing bi-degree, we obtain:  
$$
  \partial_A\partial_A\alpha=0, \quad\ol{\partial}_A\ol{\partial}_A\alpha=0,
  \quad 
  (\partial_A\ol{\partial}_A +\ol{\partial}_A\partial_A)\alpha= [\alpha\wedge F_A].
$$
  On the other hand, let $\{A_t\}_{|t|<\epsilon}$ be a path of SDCI connections  with $A_0=A$. An infinitesimal deformation $\alpha:=\dtzero(A_t)\in\Omega^1(\fg_E)$ varies the curvature  by 
$
  F_{A_t}=F_A+td_A\alpha+t^2\alpha\wedge \alpha.
$ 
  Decomposing $\alpha=\alpha^{(1,0)}+\alpha^{(0,1)}$, in $(1,0)$ and $(0,1)$ part respectively,
\begin{align*}
        d_A\alpha 
        &= (\partial_A+\ol{\partial}_A)(\alpha^{(1,0)}+\alpha^{(0,1)})\\
        &=\partial_A\alpha^{(1,0)}+\ol{\partial}_A\alpha^{(0,1)}+(\partial_A\alpha^{(0,1)}+\ol{\partial}_A\alpha^{(1,0)}),
\end{align*}
  so, by Lemma \ref{lem:SDcharacterization} item $(i)$, the infinitesimal deformation $\alpha$ satisfies: 
$$
  \partial_A\alpha^{(1,0)}=0, 
  \quad
  \ol{\partial}_A\alpha^{(0,1)}=0,
  \qandq
  \langle (\partial_A\alpha^{(0,1)}+\ol{\partial}_A\alpha^{(1,0)}),\omega\rangle=0.
$$
  Using the definition of $d_7$ in \eqref{eq:d7}, the above relations give  
  \begin{align*}
  d_7(\alpha^{(1,0)}+\alpha^{(0,1)})
  &=\partial_A\alpha^{(1,0)}+\omega\otimes \langle \partial_A\alpha^{(0,1)}+\ol{\partial}_A\alpha^{(1,0)},\omega\rangle
  +\ol{\partial}_A\alpha^{(0,1)}.
  \qedhere
  \end{align*}
\end{proof}
  The linearisation at $A\in\cA$ of the $\cG$-action   is $-d_A\colon \Omega^0(\fg_E)\to\Omega^1(\fg_E)$, so a natural transverse `Coulomb' slice is given by the orthogonal complement $\Ker(d_A^\ast)$ of $\im(d_A)\subset\Omega^1(\fg_E)$ \cite[p.~131]{Donaldson1990}. For small $\epsilon>0$, we set:  
\begin{equation}
\label{eq:Te slice}
    S_{A,\epsilon}:=\{\alpha\in\Omega^1(M) \;\vert\;d^\ast_A(\alpha)=0,\;\Vert\alpha\Vert<\epsilon \},
\end{equation}  
where $\Vert\cdot\Vert$ is a suitable Sobolev norm. Following \cite[(4.2.6)]{Donaldson1990}, a   neighbourhood of $[A]\in\cB$ is described by the quotient of $S_{A,\epsilon}$ under $\pi\colon A\in \Omega^1(\fg_E)\mapsto [A]\in \cB=\cA/\cG$, as in the following \textit{slice lemma} (see also \cite{Freed1991}):
\begin{lemma}
\label{lem:slice lemma}
For every $A\in \cM^\ast$, there exists $\epsilon=\epsilon(A)>0$  such that the restriction
$\pi|_{{S_{A,\epsilon}}}$
is a homeomorphism onto an open set $\pi(S_{A,\epsilon})\subset\cB$.
\end{lemma}
In view of the slice condition in \eqref{eq:Te slice}, we consider the restriction of  \eqref{eq:d7} to $\Ker(d^\ast_A)$ as follows:
$$
d_7\vert_{\Ker(d^\ast_A)} \colon \Ker(d^\ast_A)\subset \Omega^1(\fg_E)\to\Omega_{6\oplus 1}^2(\fg_E). 
$$
\begin{proposition}
\label{prop:deformationSpace}
    In the context of Theorem \ref{thm:Intro}, the infinitesimal deformations of an irreducible  selfdual  contact instanton $A$ are described by the complex 
    \begin{equation}
    \label{eq:deforamtionComplex 1}
        \rC^\bullet:\quad 0\to\Omega^0(\fg_E) \xrightarrow[{  }]{ d_A}\Omega^1(\fg_E) \xrightarrow[{  }]{d_7=p\circ d_A}\Omega_{6\oplus 1}(\fg_E)\to 0.
    \end{equation}
    The tangent space $T_{[A]}\cM^\ast$ is isomorphic to the first cohomology group  $H^1(\rC):=\frac{\Ker(d_7)}{\im(d_A)}$. 
\begin{proof}
  We know from Proposition \ref{prop:kerd7} that infinitesimal SDCI deformations of $A$ lie in $\Ker(d_7)$. On the other hand, if $A$ is an irreducible SDCI connection, then   $d_7\circ d_A=p\circ F_A=0$, so   \eqref{eq:deforamtionComplex 1} is indeed a complex. Furthermore the derivative of the action of the gauge group $\cG$ at $A$ is $-d_A\colon \Omega^0(\fg)\to \Omega^1(\fg)$, i.e.,  the   tangent space to the    $\cG$-orbit is $\im (d_A)$  \cite[p.~37]{Donaldson1990}, hence $H^1(\rC)$ represents the tangent space $T_{[A]}\cM^\ast$ at $[A]$.   
\end{proof}
\end{proposition}

Proposition \ref{prop:deformationSpace} shows Theorem \ref{thm:Intro}--(i). Observe  that, in contrast to the $4$-dimensional case, the complex $\rC^\bullet$ in \eqref{eq:Complexd7}   is not elliptic (eg. by comparing bundle ranks  in \eqref{eq:deforamtionComplex 1}), so we do not know a priory whether $H^1(\rC)$ is finite-dimensional. In the next section,  we show  in Proposition \ref{prop:ElliptExtendComplex} that this is indeed the case, but we must resort to an auxiliary construction.

\subsection{The extended complex of selfdual contact instantons}
\label{sec:extended complex}

Let us prove that the first cohomology group $ H^1(\rC)$ in \eqref{eq:Complexd7} is finite-dimensional.
\begin{lemma}
\label{lem:L|estrellaPhi}
    Let $\sigma=\eta\wedge d\eta$ the natural $3$-form on the Sasakian manifold $(M^7,\eta,\xi,g,\Phi)$, as in \eqref{eq:varphi}. The linear map
    \begin{equation}
    \label{eq:L|estrellaPhi}
    \begin{matrix}
        {L_{\ast\sigma}\colon }&{\Omega^2(M) }&{\to}&{\Omega^6(M) }\\
        { }&{\alpha }&{\mapsto } &{\alpha\wedge\ast\sigma}
    \end{matrix} 
    \end{equation}
    satisfies 
        \[
        L_{\ast\sigma}\vert_{\Omega^{2}_{8}\oplus\Omega^{2}_{6}}\equiv 0,
        \qandq
        L_{\ast\sigma}\vert_{\Omega^{2}_{1}\oplus\Omega^{2}_{V}}\colon\Omega^{2}_{1}\oplus\Omega^{2}_{V }\xrightarrow[{}]{\simeq}\Omega^6(M).   \]
\end{lemma}
\begin{proof}
  The transverse  symplectic $2$-form $d\eta=:\omega$ is given in local Darboux coordinates  \cite[Theorem.~3.1]{blair2010contact}  by $d\eta=dx^{14}+dx^{25}+dx^{36}$, so $d\eta^2=-2(dx^{1245}+dx^{1346}+dx^{2356})$, and therefore $\ast\sigma$ is given   by: 
\begin{equation}
\label{eq:ast phi}
  \ast\sigma=-\left( dx^{2356}+dx^{1346}+dx^{1245}\right)=\frac{1}{2}d\eta^2.
\end{equation}
  Using the decomposition \eqref{eq:2formsDecomposition} of $\Omega^2(M)$ to express $L_\sigma$ in the bases  $\{w_i\}_{i=1}^8$ and $\{v_j\}_{j=1}^6$ from \eqref{eq:omega2}, we obtain both claims by direct computation. 
\end{proof}

In terms of the projected differential \eqref{eq:Complexd7}, the previous Lemma gives the identification
\begin{equation}
\label{eq:identificationOfd7}
  d_7=L_{\ast\sigma}\circ d_A,
\end{equation}
since   $L_{\ast\sigma}\vert_{\Omega^{2}_{1}\oplus\Omega^{2}_{V}}\colon\Omega^{2}_{1}\oplus\Omega^{2}_{V }\xrightarrow[{}]{\sim}\Omega^6(M)$ [cf. Lemma \ref{lem:L|estrellaPhi}]  and,  in local Darboux coordinates  \cite[Theorem.~3.1]{blair2010contact}, we can define an isomorphism $T\colon\Omega_V^2\xrightarrow[]{\sim} \Omega^2_6$ by   \begin{equation}
 \label{eq:iso Omega 6 omega V}
\begin{array}{*{3}c}
  T(dx^{17})=-v_5 & T(dx^{27})=-v_3, & T(dx^{37})=-v_1  \\[3pt]
  T(dx^{47})=-v_6 & T(dx^{57})=-v_4, & T(dx^{67})=-v_2  
\end{array} .
\end{equation} 
These identifications can be  summarised  in the following diagram:
 \[
  \begin{diagram}
  \node{\Omega^1(\fg_E)}\arrow{se,t}{ d_A}\arrow[2]{e,t}{d_7:=p\circ d_A}\node[2]{\Omega^2_6\oplus\Omega^2_1}\\
  \node{} 
  \node{\Omega^2(\fg_E)}\arrow{ne,t}{P} \arrow{s,l}{L_{\ast\sigma}} 
  \node{}\\ 
  \node{} 
  \node{\Omega^6(M)}\arrow{e,b}{L_{\ast\sigma}^{-1}} \node{\Omega^2_V\oplus\Omega^2_1.} \arrow[2]{n,b}{T\oplus \mathbbm{1} } 
  \end{diagram}
\]
It is worth mentioning that a \emph{canonical} isomorphism $\Omega_{V}^2\simeq\Omega_{6}^{2}$ is given by the global transverse holomorphic volume form, in the  special case of contact Calabi-Yau $7$-manifolds, cf. Definition \ref{def:cCY} and Lemma \ref{lem:L|epsilon} below.
  
  From the definition \eqref{eq:modulispace} of $\cM$, we know that this space is described  (modulo gauge) near an instanton $A$ as the zero locus of the map
\begin{equation}
\label{eq:mapZeroSection}
    \Psi\colon \alpha\in S_{A,\epsilon}\subset\cA 
    \to p(d_A\alpha+\alpha\wedge\alpha)\in\Omega^2_{6\oplus 1}(\fg_E),
\end{equation}
  where the neighbourhood $ S_{A,\epsilon}$ from \eqref{eq:Te slice} is transversal to $\cG $-orbits. Let us check that $\Psi$ is a Fredholm map, so that standard theory provides a finite-dimensional local model for $\cM$.  The linearisation of $\Psi$ at the origin is 
$$ 
D(\Psi)_0=p\circ d_A=d_7,
$$ 
and  $d_7|_{\Ker(d^\ast_A)} \colon \Omega^1(\fg_E)\to \Omega^2_{6\oplus 1}(\fg_E)$  is shown to be Fredholm via the `Euler characteristic' map 
\begin{equation}
\label{eq:Euler characteristic}
    \mathbbm{D}_A:=d_7\oplus d^\ast_A\colon\Omega^1(\fg_E)\to\left(\Omega^2_{6\oplus 1}\oplus\Omega^0\right)(\fg_E)
\end{equation}
associated to the complex $\rC^\bullet$ \eqref{eq:Complexd7}.  Note that: 
$$  
H^0(\rC)\cong \Ker(d_A), 
\quad  
H^1(\rC):=\frac{\Ker(d_7)}{\im(d_A)}\cong \Ker(\mathbbm{D}_A)
\qandq 
H^2(\rC)\cong \text{Coker}(d_7).
$$
By \eqref{eq:identificationOfd7}, we identify $d_7$ with  $L_{\ast\sigma}\circ d_A\colon \Omega^1(\fg_E)\to \Omega^6(\fg_E)$, then  we  can consider the \emph{extended complex}
\begin{equation}
\label{eq:extendedComplex}
    \begin{tikzcd}
        D^\bullet\colon 0\arrow[r] 
        &\Omega^0(\fg_E) \arrow[r,"d_A"]
        & \Omega^1(\fg_E) \arrow[r, "L_{\ast\sigma}\circ d_A"]
        & \Omega^6(\fg_E) \arrow[r,"d_A"]
        & \Omega^7(\fg_E) \arrow[r]
        &  0
    \end{tikzcd}. 
\end{equation}
If  \eqref{eq:extendedComplex} is elliptic, then $\mathbbm{D}_A$ is Fredholm and, in particular, $\Ker(\mathbbm{D}_A)$ is finite-dimensional. To see that explicitly, we will need the following elementary technical facts:
\begin{lemma}
\label{lem:adjunto d7}
    \quad
    \begin{enumerate}[(i)]
        \item 
        Let $L_{\ast\sigma}\colon \Omega^2(M)  \to \Omega^6(M)$ the operator defined in  \eqref{eq:L|estrellaPhi}, then $[L_{\ast\sigma},d_A]=0$.
    
        \item 
        The formal adjoint of $d_7$ defined by \eqref{eq:d7} is given by  $d_7^\ast=\ast d_7\ast\colon\Omega^6(\fg_E)\to\Omega^1(\fg_E)$.
    \end{enumerate}
\end{lemma}
\begin{proof}
    For $(i)$,  note from \eqref{eq:ast phi} that $d(\ast\sigma)=d(d\eta^2)=0$, so the assertion is straightforward: 
    $$
    d_A(L_{\ast\sigma}\alpha)
    =d_A(\ast\sigma\wedge\alpha) 
    =(d\ast\sigma)\wedge\alpha
    +\ast\sigma\wedge (d_A\alpha)
    =L_{\ast\sigma}(d_A\alpha).
    $$
    For  \emph{(ii)} we use  \emph{(i)} and   the identification in   \eqref{eq:identificationOfd7}; for any $\alpha\in\Omega^1(\fg_E)$ and $\beta\in \Omega^6(\fg_E)$, 
\begin{align*}
    \langle d_7\alpha,\beta  \rangle
    &= (\ast\sigma\wedge d_A\alpha)\wedge\ast\beta= d_A\alpha\wedge\ast^2(\ast\sigma\wedge\ast\beta)\\
    &= d_A\alpha\wedge \ast(\ast L_{\ast\sigma}\ast(\beta))\\
    &= \langle d_A\alpha,\ast L_{\ast\sigma}\ast(\beta) \rangle\\
    &= \langle  \alpha,\ast d_A L_{\ast\sigma}\ast(\beta) \rangle\\
    &= \langle  \alpha,\ast d_7\ast(\beta) \rangle
    \qedhere
\end{align*}
\end{proof}

\begin{proposition}
\label{prop:ElliptExtendComplex}

    If $A$ is a connection with curvature $F_A\in\Omega^2_{8}$, the extended complex \eqref{eq:extendedComplex} is elliptic.
\end{proposition}
\begin{proof}
    We follow the argument of \cite[Proposition~1.22]{SaEarp2009}. Fix  a non zero  section $\varsigma$ of $ \pi\colon T^\ast M\setminus\{0\}\to M$, so we have the symbol complex  
\begin{equation}
    \label{eq:symbol complex}
    0\to\pi^\ast\left( \Omega^0(\fg_E)\right)_\varsigma \xrightarrow[{  }]{ \varsigma\wedge\cdot}\pi^\ast\left(\Omega^1(\fg_E)\right)_\varsigma  \xrightarrow[{  }]{\ast\sigma\wedge\varsigma\wedge\cdot}\pi^\ast\left(\Omega^6(\fg_E)\right)_\varsigma \xrightarrow[{  }]{\varsigma\wedge\cdot} \pi^\ast\left(\Omega^7(\fg_E)\right)_\varsigma \to 0.
\end{equation}
    To see the exactness of \eqref{eq:symbol complex} at the middle, take $\alpha\in \Omega^1(\fg_E)$ such that $\ast\sigma\wedge\varsigma\wedge\alpha=0$. We need to show that $\alpha$ lies in $(\Omega^0(\fg_E)).\varsigma$. Note that  
 $\varsigma\wedge\alpha\in\Ker(L_{\ast\sigma})=\Omega^2_8\oplus\Omega^2_6$ [cf. \eqref{eq:L|estrellaPhi}],  so, by definition of the eigenspaces  $\Omega^2_8$ and $\Omega^2_6$ [cf.  \eqref{eq:omega2}], we obtain:  
\begin{equation}
\label{eq:ElliptExtendComplex equation}
    \pm\sigma\wedge\varsigma\wedge\alpha =\ast(\varsigma\wedge\alpha).
\end{equation}
    We show that \eqref{eq:ElliptExtendComplex equation} implies   that    $\varsigma\wedge\alpha=0$ and this finish the proof.  Let $\{e^i\}_{i=1}^7$ be a basis of $T^\ast_x M$ with   $e^1 = \varsigma$. Then  $\mu \in\Lambda^k T_x^\ast M$, can be written 
    $$
    \mu = e^1\wedge\gamma+\beta
    $$
    where $\gamma$ and $\beta$ are products just involving  $e^2,\cdots,e^7$. Let   $\alpha=\varsigma\wedge\gamma+\beta$, $\eta=\varsigma\wedge \gamma'+\beta'$ and $d\eta=\varsigma\wedge\gamma''+\beta''$ where $\gamma,\gamma',\gamma''$ and $\beta,\beta',\beta''$ are products just involving  $e^2,\cdots,e^7$, hence \eqref{eq:ElliptExtendComplex equation} becomes 
\begin{align*}
    \eta \wedge d\eta\wedge\varsigma\wedge\alpha &= 
    (\varsigma\wedge \gamma'+\beta' )\wedge  (\varsigma\wedge\gamma''+\beta'')\wedge\varsigma\wedge(\varsigma\wedge\gamma+\beta)  \\
    &=  \beta' \wedge  \beta'' \wedge\varsigma\wedge  \beta 
\end{align*} 
so $  \beta' \wedge  \beta'' \wedge\varsigma\wedge  \beta =\ast(\varsigma\wedge\beta)$
the left-hand side of the above equality involves $\varsigma$, while the right-hand side does not, so $  \varsigma\wedge\beta =0$. Hence
$$
   \varsigma\wedge\alpha  =\varsigma\wedge(\varsigma\wedge\gamma+\beta)=\varsigma\wedge \beta  =0 
$$
as claimed, and so $\alpha=f\wedge\varsigma.$
\end{proof}
Together, Propositions \ref{prop:deformationSpace} and \ref{prop:ElliptExtendComplex} prove Theorem \ref{thm:Intro}--$(i)$, i.e., that the space of infinitesimal deformations of $A$ is finite-dimensional and   isomorphic to the first cohomology  $H^1(\rC)$ of  \eqref{eq:Complexd7}.   In \S \ref{sec:moduli space CI}, we address such deformations from another perspective, showing that this local model is isomorphic to the first cohomology group of the complex \eqref{eq:complexDeformation} [cf. Theorem  \ref{thm:Intro}--$(ii)$]. As observed in Remark \ref{rem:theorem intro}, to assume that  $H^2(\rC)$ vanishes would in general be much too strong, leading to a $0$-dimensional local model. Instead, we can show the smoothness of the moduli space under a weaker  obstruction theory, in terms of  $H^2_B$  [cf. Proposition \ref{prop:vanishingH2}].


\subsection{ \texorpdfstring{$\rG_2$}{Lg}-instantons on contact Calabi-Yau manifolds}
\label{sec:contact and G2}

We have seen that the Sasakian structure $(\eta,\xi,g,\Phi)$ on $M^7$ naturally induces a moduli space of SDCI. In the \emph{contact Calabi-Yau} case, moreover, the Sasakian $7$-manifold carries in fact a transverse $\SU(3)$-structure, and hence  a natural $\rG_2$-structure \cites{habib2015some}. As such, it may indeed be seen as somewhat of an interpolation between $CY^3$-geometry and $\rG_2$-geometry. We explain the relationship between $\rG_2$-instantons and contact instantons [see Proposition \ref{prop:HYM-ContactCorrespondence}] in that context, following the approach of \cites{Calvo-Andrade2016}.

\begin{definition}
\label{def:cCY}
    A Sasakian manifold $(M^{2n+1},\eta,\xi,\Phi,\varepsilon)$ is said to be a \emph{contact Calabi-Yau manifold} (cCY) if $\varepsilon$ is a nowhere-vanishing transverse form of horizontal type $(n,0)$ [cf. \eqref{eq:TangentDecomposition}] such that 
$$
\varepsilon\wedge\bar{\varepsilon}
=(-1)^{\frac{n(n+1)}{2}}\ii^{n}\omega^n
\qandq 
d\varepsilon=0.
$$
\end{definition} 

It is well-known that, for a Calabi-Yau 3-fold $(\rZ, \omega, \varepsilon)$, the product $\rZ\times \rS^1$  has a natural torsion-free  $\rG_2$-structure defined by: $\varphi:= dt\wedge\omega+\Im(\varepsilon),$
where $ t$ is the coordinate on $\rS^1$. The Hodge dual of $\varphi$ is 
\begin{equation}
    \label{eq:psi cCY}
\psi:=\ast\varphi=\frac{1}{2}\omega\wedge\omega+dt\wedge\Re(\varepsilon)
\end{equation}
and the induced metric $g_\varphi=g_\rZ + dt\otimes dt $ is the Riemannian product metric on $\rZ\times S^1$ with holonomy $\Hol(g_\varphi)=\f{SU}(3)\subset\rG_2$. The contact Calabi-Yau structure essentially emulates all of these features, albeit its $\rG_2$-structure has some symmetric torsion. Sasakian manifolds with transverse holonomy $\SU(n)$ are studied by Habib and Vezzoni;  a number of facts from \cite[{\S~6.2.1}]{habib2015some}  can be summarised as follows:
\begin{proposition} 
\label{prop:G2estruturaCCY} 
 Every cCY manifold $(M^7,\eta,\xi,\Phi,\varepsilon)$  carries a cocalibrated $\rG_2$-structure
\begin{equation}
\label{eq:G2structure}
    \varphi 
    :=\eta\wedge \omega+\Im(\epsilon)
    =\sigma + \Im(\epsilon),
\end{equation} 
with torsion $d\varphi= \omega\wedge\omega$   and Hodge dual  $ 4$-form
$ \psi=\ast\varphi = \frac{1}{2}\omega\wedge\omega+ \eta\wedge\Re(\varepsilon)$. Here $\omega:=d\eta$ and $\sigma:=\eta\wedge \omega$, as in \eqref{eq:varphi}.
\end{proposition}

\begin{lemma}
\label{lem:L|epsilon}
On a cCY manifold $(M^7,\eta,\xi,\Phi,\varepsilon)$, the operator 
\begin{equation}
\label{eq:L|epsilon}
\begin{matrix}
    {L_{\epsilon }\colon }&{\Omega^2(M)}&{\to}&{\Omega^2(M) }\\
    { }&{\alpha }&{\mapsto }&{\ast(\alpha\wedge \Im(\epsilon))},
\end{matrix}
\end{equation}
satisfies 
$L_{\epsilon}\vert_{\Omega_{6}^{2}\oplus\Omega_{8}^{2}\oplus \Omega_{1}^2}\equiv 0$
and $L_{\epsilon}\vert_{ \Omega_{V}^2(M)}\colon \Omega_{V}^2(M)\xrightarrow[{}]{\sim}\Omega_{6}^{2}(M)$ is an isomorphism.
\end{lemma}

\begin{proof}
  In the local Darboux transverse complex  coordinates $z_i$ defined in \eqref{eq:zi}, we have $\epsilon =dz^1\wedge dz^2\wedge dz^3$. In particular, $\Im(\epsilon)=(dx^{234}+dx^{126})-(dx^{135}+dx^{456})$. Using the decomposition \eqref{eq:2formsDecomposition} of  $2$-forms to compute $L_\epsilon$ in the bases $\{w_i\}_{i=1,\dots,8}$ and $\{v_j\}_{j=1,\dots,6}$ of  \eqref{eq:omega2}, we obtain immediately 
  $
  L_{\epsilon}\vert_{\Omega_{6}^{2}\oplus\Omega_{8}^{2}\oplus \Omega_{1}^2}\equiv 0
  $
  and, by inspection, 
\[
\begin{array}{*{3}c}
L_{\epsilon }(dx^{17})=-v_5, & L_{\epsilon }(dx^{27})=-v_3, & L_{\epsilon }(dx^{37})=-v_1,  \\[3pt]
L_{\epsilon }(dx^{47})=-v_6, & L_{\epsilon }(dx^{57})=-v_4, & L_{\epsilon }(dx^{67})=-v_2.  
\end{array} 
\]
So $L_{\epsilon}\vert_{ \Omega_{V}^2(M)}\colon \Omega_{V}^2(M)\xrightarrow[{}]{\sim}\Omega_{6}^{2}(M)$ is an isomorphism and coincides with $T$ in \eqref{eq:iso Omega 6 omega V}.
\end{proof}
 
The  \emph{$\rG_2$-instanton equation} on a cCY $7$-manifold reads:
\begin{equation}
    \label{eq:G2 instantons}
\ast(\underbrace{(\eta\wedge \omega 
+ \Im(\epsilon))}_{\varphi }\wedge F_A)
=  F_A
\end{equation}
or, equivalently, 
$ 
F_A\wedge\psi =0,
$
where  $\psi$ is the dual $4$-form \eqref{eq:psi cCY}.   
\begin{proposition}
\label{prop:ContactthenG2} 
    On a cCY $7$-manifold, the SDCI equation \eqref{eq:ContactInstEquation}  implies the $\rG_2$-instanton equation \eqref{eq:G2 instantons}.  
\end{proposition}
\begin{proof}
If $A$ is a SDCI, then  
$ F_A=\ast(\sigma\wedge F_A)\in\Omega^2_8(\fg_E)
\subset \Ker L_{\epsilon}$, by Lemma \ref{lem:L|epsilon}. Therefore
\begin{align*}
    \ast(\varphi\wedge F_A)
    &=\ast( \sigma\wedge F_A )+L_{\epsilon }(F_A)\\
    &= F_A. \qedhere
\end{align*}
\end{proof}
If the complex Sasakian  bundle $\bE\to M^7$ has a holomorphic structure, then, at least among Chern connections (mutually compatible with the holomorphic structure and a Hermitian metric),  the sets of solutions of both equations actually coincide:     

\begin{proposition}
\label{prop:HYM-ContactCorrespondence}
Let $\cE\to M$ be a  Sasakian holomorphic vector bundle [cf. Definition  \ref{def:SasakianoHol}]  
on a   cCY   $7$-manifold   with its natural $\rG_2$-structure \eqref{eq:G2structure}. Then a Chern connection $A$  on $\cE$ is a $\rG_2$-instanton  if, and only if, $A$ is a SDCI as in  \eqref{eq:contacIns1}.
\begin{proof}
If  $A$ is the Chern connection, then $F_A\in\Omega^{1,1}(M)$ [cf. Proposition \ref{prop: omega (1,1)}], so taking account of the bi-degree of the transverse holomorphic volume form $\epsilon$ [cf. Definition \ref{def:cCY}], it follows that   $F_A\wedge \Im(\varepsilon)=0$. Therefore 
\begin{align*}
    F_A\wedge \varphi  
    &= F_A\wedge(\underbrace{\eta\wedge d\eta 
    + \Im(\epsilon)}_{\varphi })
    =F_A\wedge\sigma+F_A\wedge \Im(\varepsilon)\\
    &=F_A\wedge\sigma. \qedhere
\end{align*}
\end{proof}
\end{proposition}
\subsection{The Yang-Mills and Chern-Simons functionals}
\label{sec: YM and CS}
We will describe two natural gauge-theoretic action functionals on a Sasakian $G$-bundle $\bE\to M^7 $ [Definition \ref{def:fibradoSasakiano}],  adapting the approach of \cites{sa2014generalised}. This section culminates at topological Yang-Mills energy bounds on holomrphic Sasakian bundles, in terms of integrable (anti-)selfdual contact instantons.
  
The \emph{Yang-Mills functional} acts on the space $\cA$ of connections on $\bE$ and it is defined by  
\begin{equation}
\label{eq:YMfunctional}
  \Y\colon A\in\cA \mapsto \Vert F_A\Vert^2:=
  \displaystyle\int_M\langle F_A\wedge\ast F_A\rangle_{\fg_E}.
\end{equation}
  The curvature $F_A$ splits orthogonally \eqref{eq:HorizontalVerticalDecomposition}   into horizontal and vertical parts, respectively:  
$$
  F_H=i_\xi(\eta\wedge F_A)
  \qandq F_V=\eta\wedge i_\xi F_A,
$$  
  so $\Y(A)$ has at least two independent components. Moreover, the horizontal part further splits orthogonally, by Lemma \ref{lem:orthogonality}:
$$
  F_H= \underbrace{\frac{1}{2}(F_H+\ast(\sigma\wedge F_H))}_{F_H^+}+\underbrace{\frac{1}{2}(F_H-\ast( \sigma\wedge F_H))}_{F_H^-},
$$
  hence 
\begin{equation}
\label{eq:YM decomposition}
  \Y(A)
  =\Vert F_H^{+}\Vert^2+\Vert F_H^{-}\Vert^2
  +\Vert F_V\Vert^2.
\end{equation}
  Given   $A\in \cA$, the \emph{ charge} of $A$ is defined by 
\begin{equation}
    \label{eq:topolgical charge}
    \kappa(A):= \displaystyle\int_M\tr (F_A^2) \wedge\sigma.  
\end{equation} 
\begin{proposition}
  Let $A\in\cA$ be a connection on $\bE$, the   charge of A \eqref{eq:topolgical charge} is determined by the horizontal curvature:
\begin{equation}
\label{eq:topological charge relation 1}
  \kappa(A) = \Vert F_H^{+}\Vert^2 -\Vert F_H^{-}\Vert^2. 
\end{equation}
  Furthermore,  
\begin{equation}
\label{eq:topological charge relation 2}
  \vert \kappa(A)\vert
  \leq 
  \Vert F_H^{+}\Vert^2
  \leq
  \Y(A).
\end{equation}
   These bounds are saturated if and only if   $A$ is ASDCI or SDCI.  
\end{proposition}
\begin{proof}
It suffices to show \eqref{eq:topological charge relation 1}, because then \eqref{eq:topological charge relation 2} follows from \eqref{eq:YM decomposition}. Using the definition of $\kappa(A)$ in \eqref{eq:topolgical charge} and the definition of $F_H^+$ given above, we compute:  
\begin{align*}
    \Vert F_H^{+}\Vert^2 
    &=\frac{1}{2}\int_{M}\tr \left( F_H^{+}\wedge\ast (F_H+\ast(\sigma\wedge F_H))\right) \\
    &=\frac{1}{2} \left( \int_{M}\tr (F_H^{+} \wedge\ast F_H) 
     +\int_{M}\tr \left( F^{+}_H\wedge\sigma\wedge F_H\right)   \right)\\
    &=\frac{1}{4} \left( \int_{M}\tr \left( (F_H+\ast(\sigma\wedge F_H)) \wedge\ast F_H \right) 
    +\int_{M}\tr \left( ( F_H+\ast(\sigma\wedge F_H)\wedge\sigma\wedge F_H\right)   \right)\\
    &= \frac{1}{4}\left(\Vert F_H \Vert^2+\int_{M}\tr \left(\ast(\sigma\wedge F_H)\wedge \ast F_H \right)
+\int_{M}\tr \left(\sigma\wedge F_H^2 \right)+ \Vert \sigma\wedge F_H\Vert^2\right) \\
    &= \frac{1}{4}\left( \Vert F_H \Vert^2 +2\kappa(A)+  \Vert \sigma\wedge F_H\Vert^2\right). 
\end{align*}

Analogously,   
\begin{align*}
\Vert F_H^{-}\Vert^2 &=
\frac{1}{2}\int_{M}\tr \left( F_H^{-}\wedge\ast (F_H-\ast(\sigma\wedge F_H))\right) \\
&= \frac{1}{4}\left( \Vert F_H \Vert^2 -2\kappa(A)+  \Vert \sigma\wedge F_H\Vert^2\right). 
\end{align*}
Therefore,    
$ 
\kappa(A) = \Vert F_H^{+}\Vert^2 -\Vert F_H^{-}\Vert^2. 
$ 
Now  using the inequality $ \vert a^2-b^2 \vert\leq a^2+b^2 $, we obtain the following bounds:
$$
\vert \kappa(A)\vert\leq\int_{M}\tr (F_H\wedge\ast F_H)
\qandq \vert \kappa(A)\vert\leq\int_{M}\tr (F_A\wedge\ast F_A). 
$$
Under the hypothesis $F_V = 0$, if $A$ is a SDCI, then $F^{-}_H = 0$, and it follows from \eqref{eq:topological charge relation 1}  that the latter bound   is saturated. 
\end{proof}
Now, fix a reference connection $A_0\in\cA$ on $\bE$. The \emph{Chern-Simons action} is defined by: $\f{CS}(A_0)=0$ and 
$$
\begin{array}{rll}
 \f{CS}                &\colon                                                                & \cA  \cong A_0+\Omega^1( \fg_E)\to  \R  \\
  \f{CS}(A_0+\alpha)&:=     & \frac{1}{2}\displaystyle\int_M \tr (d_{A_0}\alpha\wedge \alpha+\frac{2}{3}\alpha\wedge \alpha\wedge \alpha)\wedge\ast\sigma.
\end{array}
$$

\begin{lemma}
\label{lem:TopologicalCharge}
Let  $\bE\to M^7$ be a Sasakian $G$-bundle over a  closed  $7$-manifold. For  any connection $A'=A+\alpha\in\cA$, for $\alpha\in\Omega^1(\fg_E)$, on the underlying complex vector bundle $E$, we have: 
\begin{equation}
\label{eq:kapaA kapa A'}
    \kappa(A)=\kappa(A')+\int_M\tr (F_A\wedge \alpha\wedge d\eta^2). 
\end{equation}
\end{lemma}
\begin{proof}
Let  $A\in\cA$    and let $\alpha\in\Omega^1(\fg_E)$ be a variation of $A$. From standard Chern-Weil theory, we know that: 
$$
\tr (F^2_{A+\alpha})-\tr (F^2_A)=d(\tr \delta)   
$$
where  
$
\delta=\delta(A,\alpha)\in\Omega^1(\fg_E)
$ 
given by 
 $
\delta = F_A\wedge \alpha+\frac{1}{2}d_A \alpha\wedge \alpha+\frac{1}{3}\alpha \wedge \alpha\wedge \alpha.
 $ 
Since $M$ is closed, by Stokes' theorem   we obtain  
\begin{align*}
    \kappa(A)  
    &:=  \int_{M}\tr (F_A^2)\wedge\sigma
     =\int_{M}\left(  \tr (F_{A+\alpha}^2)-d\tr \delta\right) \wedge      \sigma\\
    &=  \kappa(A')-\left(\int_{M}d(\tr \delta\wedge\sigma)-\int_M\tr \delta\wedge d\sigma      \right)\\ &=  \kappa(A')+\int_M\tr \delta\wedge d\sigma. 
\end{align*}
Now, we analyse the term $\int_M\tr \delta\wedge d\sigma$ for a transverse  $1$-form $\alpha\in \Omega^1(\fg_E)$:
\begin{align*}
    \int_M\tr\delta\wedge d\sigma&=\int_M\tr (F_A\wedge \alpha+\frac{1}{2}d_A \alpha\wedge \alpha+\frac{1}{3}\alpha \wedge \alpha\wedge \alpha)\wedge d\eta^2\\
    &=\int_M\tr (F_A\wedge \alpha)\wedge d\eta^2+\frac{1}{2}\int_M\tr (d_A \alpha\wedge \alpha)\wedge d\eta^2 +\frac{1}{3}\int_M\tr (\alpha \wedge \alpha\wedge \alpha)\wedge d\eta^2\\
    &=\int_M\tr (F_A\wedge \alpha)\wedge d\eta^2.
\end{align*}
The last equality holds since the  $7$-form $\tr (\alpha \wedge \alpha\wedge \alpha)\wedge d\eta^2$ is basic; the same is true for $\tr (d_A \alpha\wedge \alpha)\wedge d\eta^2$. 
\end{proof}

\begin{corollary}
\label{cor:TopologicalCharge}
Among Chern connections, the charge is independent of the Hermitian structure; it is a holomorphic Sasakian topological invariant,   denoted   by $\kappa(\cE)$. 
\end{corollary}
\begin{proof}
Fixing a reference Chern connection $A_0\in \cA(\cE)$, we know from Proposition \ref{prop: omega (1,1)} that $F_{A_0}\wedge d\eta^2$ has type $(3,3)$, so the defect term $\int_M\tr (F_{A_0})\wedge d\eta^2\wedge \alpha$ vanishes, by excess in bi-degree, for any $\alpha\in\Omega^{1,0}(\fg_E)$. Therefore $\kappa(A_0)=\kappa(A_0+\alpha)$.
\end{proof} 
 
Now, let $A\in\cA $ be a connection on $\bE$. From \eqref{eq:YM decomposition} and \eqref{eq:topological charge relation 1} we have 
\begin{align*}
    \Y(A) 
    &=  \kappa(A)+2\Vert F_H^{-}\Vert^2+\Vert F_V\Vert^2\\
    &=  -\kappa(A)+2\Vert F_H^{+}\Vert^2+\Vert F_V\Vert^2
\end{align*}
It follows that   $\Y$ attains an absolute minimum either at  selfdual  contact instantons, i.e. when $F_H^{-}=0 $ and $F_V=0$, or at  anti-selfdual  contact instantons, i.e. when $F_H^{+}=0 $ and $F_V=0$. Furthermore,   from Corollary \ref{cor:TopologicalCharge}, among Chern connections we can replace $\kappa(A)=\kappa(\cE)$ in the above equalities, thus the sign of $\kappa(\cE)$ obstructs the existence of one or the other type of solution.  

\newpage
\section{The moduli space of contact instantons in \texorpdfstring{$7$}{Lg}-dimensions}
\label{sec:moduli space CI} 

We showed in \S \ref{sec: GT on Sasakian 7-mfds} that there exists a finite-dimensional local model for the moduli space near an irreducible SDCI. In this Section we will show part \emph{(ii)} of Theorem \ref{thm:Intro}, namely, that its dimension can be computed from an associated  transverse elliptic complex. Our strategy is inspired by \cite[\S~3]{Baraglia2016}.
\subsection{The associated elliptic complex of a contact instanton}
Consider a Sasakian $G$-bundle $\bE\to (M^7,\eta,\xi,g,\Phi)$ and $\fg_E$ its adjoint bundle. Recalling that   $\Omega^k_H(\fg_E)$ denotes the space of \emph{horizontal forms}, i.e.,  $\alpha\in\Omega^k( \fg_E)$ for which $i_\xi\alpha=0$, we introduce the following operators: 
\begin{equation}
\begin{matrix}
    {d_V\colon}&{\Omega_H^k( \fg_E)}&{\to}&{\Omega_H^{k}( \fg_E)}&{}\\
    {}&{\alpha}&{\mapsto}&{i_\xi(d_A \alpha)}&{}
\end{matrix}, 
\qquad
\begin{matrix}
    {d_T\colon}&{\Omega_H^k( \fg_E)}&{\to}&{\Omega_H^{k+1}( \fg_E)}&{}\\
    {}&{\alpha}&{\mapsto}&{d_A \alpha-\eta\wedge d_V\alpha}&{}
\end{matrix}    
\end{equation}
Moreover, $\alpha\in\Omega^k_H(\fg_E) $ is called \emph{basic}  if $i_\xi d_A\alpha=0$; we denote by $\Omega^k_B(\fg_E)$ the space of basic forms.
\begin{remark}
  The operators $d_T$ and $d_V$  are defined in \cite[\S~3.1]{Baraglia2016}  for a $5$-dimensional  contact manifold. Note that the restriction of the Lie derivative along the Reeb field $\cL_{\xi}\vert_{\Omega^k_H(\fg_E)}=d_V$, thus $\Omega^k_B(\fg_E)=\Ker(d_V)$. Moreover, we can write $d_T(\alpha)=i_\xi(\eta\wedge d_A\alpha)$, i.e.,   $d_T\alpha$ is just the horizontal part of $ d_A\alpha $. In particular, $d_T$ coincides  with the usual covariant exterior differential  $d_A$ on basic forms.
\end{remark} 
\begin{remark}
\label{rem:closed}
Even though almost all results in this section hold for  $M^7$ just compact, in   Proposition \ref{prop:omegaMap} we need it to be actually closed. 
\end{remark}
  The graded ring $\left(\Omega^\bullet( \fg_E),\wedge\right)$ has a natural graded Lie algebra structure, given in local coordinates by the bi-linear map  
\begin{equation}
\label{eq:brackets}   
    \begin{array}{rrcl}
    [\cdot\wedge\cdot]\colon &\Omega^p(\fg_E)\times\Omega^q(\fg_E) &\longrightarrow& \Omega^{p+q}(\fg_E)\\
    & [\Theta\wedge\Psi]&=&\sum_k\left(\Theta^k_j\wedge\Psi^i_k-(-1)^{pq}\Psi^k_j\wedge\Theta^i_ k \right)
    \end{array}.
\end{equation}
  The following properties are immediate to check, for $\Phi\in\Omega^p(\fg_E)$, $\Psi\in\Omega^q(\fg_E)$ and $\Theta\in\Omega^p(\fg_E)$:
$$  
  [[\Phi\wedge \Psi]\wedge  \Theta]+(-1)^{pq+qr}[[\Psi\wedge \Theta]\wedge \Phi]+(-1)^{qr+pr}[[ \Theta\wedge \Phi]\wedge \Psi]=0,
$$
$$
  [\Phi\wedge \Psi]=(-1)^{pq}[\Psi\wedge \Phi]. 
$$
The next result is a $7$-dimensional adaptation of \cite[Lemma~3.1]{Baraglia2016}, under the assumption that $(M^7,\eta,\xi,g) $ is a  \emph{K-contact}   manifold, i.e., that the Reeb field $\xi$ is Killing. Every Sasakian manifold is $K-$contact but, at dimensions greater than $3$, the $K-$contact condition  is strictly weaker.  Yet, if the manifold is compact and Einstein, both notions again coincide \cites{boyer2001einstein}:
\begin{lemma}
\label{Lem:Lema I}
    Let $\rI=\f{Span}\{\Omega^2_8 \}\subset\Omega^\bullet( \fg_E)$ be the algebraic ideal of the graded Lie algebra $\left(\Omega^\bullet( \fg_E),\wedge\right)$   generated by  $\Omega^2_8(\fg_E)$; if the Reeb field $\xi$ is a Killing vector field, then $d_A (\rI)\subset \rI$. 
\end{lemma}   
\begin{proof}
    Since $\rI$ lies in the image of the exterior product $\Omega^2_8\otimes\Omega^\bullet(\fg_E)\to\Omega^\bullet(\fg_E)$, by linearity it suffices to show that $d_A(\alpha\otimes \psi )\in \rI$, where  $\alpha\in\Omega^{2}_{8}$ in \eqref{eq:omega2} and $   \psi\in \Omega^0( \fg_E)$. Indeed,
\begin{align*}
    d_A(\alpha\otimes \psi)
    &=d_A\alpha\otimes \psi +\alpha\wedge d_A\psi\\
    &=\eta\wedge d_V\alpha\otimes \psi+d_T\alpha\otimes\psi+\alpha\wedge d_A\psi,
\end{align*}
    and the last two terms clearly lie in $\rI$.  In order to show that $\eta\wedge d_V\alpha\in \rI$, we check that $d_V\alpha= i_\xi d_A\alpha=\cL_\xi\alpha $ belongs to $\Omega^2_8$. By the eigenspace decomposition \eqref{eq:omega2}, if $\alpha\in \Omega^2_8$,  then $\alpha=\ast(\sigma\wedge \alpha)$. Furthermore, since $\xi$ is Killing, the Lie derivative $\cL_\xi$ commutes with the star Hodge operator, so 
\begin{align*}
    \cL_\xi\alpha &=\cL_\xi\ast(\sigma\wedge\alpha)=\ast\cL_\xi(\sigma\wedge\alpha)\\
    &=\ast(\cL_\xi\eta\wedge d\eta\wedge\alpha+\eta\wedge \cL_\xi d\eta\wedge\alpha+\sigma\wedge \cL_\xi\alpha)\\
    &=\ast(\sigma\wedge \cL_\xi\alpha).
\end{align*}
    We deduce  that $\cL_\xi\alpha$ satisfies the instanton equation  \eqref{eq:contacIns1} for $\lambda=1$, i.e.,   $\cL_\xi\alpha\in\Omega^2_8$.
\end{proof} 

Lemma \ref{Lem:Lema I} shows that $d_A$ descends to a derivation $D\colon \rL^k\to \rL^{k+1}$ on the quotient  
$
\rL^\bullet=\rL^\bullet( \fg_E)
:=\Omega^\bullet( \fg_E)/\rI,
$ 
which is indeed a complex, since $d_A^2=F_A\in \Omega^2_8\subset \rI$. Therefore the quotient has the structure of a differential graded Lie algebra:  
\begin{equation}
\label{eq:complex L D}
     (\rL^\bullet,D):=(\Omega^\bullet( \fg_E)/\rI,d_A)
\end{equation}

\begin{remark}
The statement of Lemma \ref{Lem:Lema I}  follows analogously if one takes, instead, the ideal $\rI=\f{Span}\{\Omega^2_6 \}$, with $\rL^\bullet= \Omega^{\bullet}( \fg_E)/\rI$. This will have no further bearing in this article, since after all we know that the corresponding instantons in that case are trivial [Proposition \ref{prop:flatness}], but, whatever the case, we have explicit characterisations for the spaces $\rL^k$ as follows.  
\end{remark} 
\begin{proposition}
\label{prop:Lidentification}
The spaces $\rL^k$ of complex  \eqref{eq:complex L D}, for 
$k=0,\dots,3$,  admit the following decompositions:
\begin{enumerate}[(i)]
  \item For $\rI=\f{Span}\{\Omega^2_8(\fg_E)\}$, 
   $$
\begin{array}{ll}
\rL^0  \cong\Omega^0(\fg_E), & \rL^2 \cong\Omega^{2}_{6}(\fg_E)\oplus\Omega^{2}_{1}(\fg_E)\oplus\eta\wedge\Omega^{1}_H(\fg_E), \\[3pt]
\rL^1  \cong\Omega^1(\fg_E), & \rL^3 \cong\eta\wedge\left(\Omega^{2}_{6}(\fg_E)\oplus\Omega^{2}_{1}(\fg_E)\right).  
\end{array}
    $$ 
\item  For $\rI=\f{Span}\{\Omega^2_6(\fg_E)\}$,
$$
\begin{array}{ll}
  \rL^0 \cong\Omega^0(\fg_E),    &  \rL^2\cong   \Omega^{2}_{8}(\fg_E)\oplus\Omega^{2}_{1}(\fg_E)\oplus\eta\wedge\Omega^{1}_H( \fg_E),  \\[3pt]
   \rL^1\cong\Omega^1(\fg_E),    &\rL^3\cong \eta\wedge\left( \Omega^{2}_{8}(\fg_E)\oplus\Omega^{2}_{1}(\fg_E)\right).
\end{array}
 $$
 \end{enumerate}
In either case,  $\rL^k=0$ for $k\geq 4$.
\end{proposition}

\begin{proof}
For $(i)$, the identifications are immediate for $\rL^0,\rL^1$ and $\rL^2$, combining $\Omega^{k}( \fg_E)=\Omega^{k}_H( \fg_E)\oplus \eta\wedge\Omega^{k-1}_H(\fg_E)$, as in \eqref{eq:decompostion k forms}, and   the  natural decomposition  of  $\fg_E$-valued $2$-forms $\Omega^2(\fg_E)$ induced by \eqref{eq:2formsDecomposition}: 
\begin{equation}
\label{eq:2formsDecomposition2}
    \Omega^{2}(\fg_E)=\underbrace{\Omega^{2}_{6}(\fg_E)\oplus\Omega^{2}_{8}(\fg_E)\oplus\Omega^{2}_{1}(\fg_E)}_{\Omega^2_H(\fg_E)}\oplus\;\Omega^{2}_V(\fg_E).
\end{equation} 
  To show that  $\Omega^3_H(\fg_E)\subset \rI$, let $\{z_j\}_{j=1}^3$ be the  Darboux transverse complex coordinates as in  \eqref{eq:zi}. Following \cite[\S~2.1]{bedulli2007ricci}, we  denote by  
$$
  \Theta_+=(dx^{123}+dx^{246})-(dx^{345}+dx^{156})\qandq \Theta_-=(dx^{126}+dx^{234})-(dx^{456}+dx^{135})
$$
the   real and imaginary parts of the local transverse holomorphic volume form $dz^1\wedge dz^2\wedge dz^3$, respectively. 
  For every point $x\in M$,   $(H_x,\omega=d\eta)$ is a transverse symplectic vector space, and 
$ 
  J:=\Phi\vert_ {H}
$  
  is a complex structure. The group $\SU(3)$ acts irreducibly on $\Lambda^5H^\ast$ and $H^\ast$, while $\Lambda^2H^\ast$ and $\Lambda^3H^\ast$ decompose as follows \cite[\S~2]{bedulli2007ricci}:
$$
  \Lambda^2H^\ast   = \Lambda^2_1H^\ast\oplus\Lambda^2_6 H^\bullet\oplus\Lambda^2_8H^\ast, \quad
$$ 
where 
\begin{align*}
   \Lambda^2_1H^\ast  &= \R\cdot\omega \\
   \Lambda^2_6H^\ast  &= \{\alpha\in \Lambda^2H^\ast\vert J^\ast\alpha=-\alpha \}\\
   \Lambda^2_8H^\ast  &= \{\alpha\in\Lambda^2H^\ast\vert J^\ast\alpha=\alpha,\;\alpha\wedge\omega^2=0  \}\\
   \Lambda^3H^\ast    &= \Lambda^3_{\Re}H^\ast\oplus\Lambda^3_{\Im} H^\ast\oplus\Lambda^3_6H^\ast\oplus\Lambda^3_{12}H^\ast
\end{align*} 
and  
\begin{align*}
   \Lambda_{\Re}^3 H^\ast &= \R\cdot \Theta_+ \\
   \Lambda_{\Im}^3 H^\ast   &= \R\cdot\Theta_-=\{ \alpha\in \Lambda^3 H^\ast \vert \alpha\wedge\omega=0,\;                  \alpha \wedge \Theta_+=c\omega^3,\;c\in\R \}  \\
    \Lambda_{6}^3 H^\ast         &=\{\alpha\wedge\omega \vert \alpha\in\Lambda^1 H^\ast  \} \\
    \Lambda_{12}^3 H^\ast       &=\{\alpha \in\Lambda^3 H^\ast \vert \alpha\wedge\omega=0, \; \alpha\wedge\Theta_+=0, \; \alpha\wedge\Theta_-= 0 \}.  
   \end{align*}   
Our goal is to show that all subspaces in the decomposition of $\Lambda^3 H^\ast$ lie  in $\rI$.
The following exterior multiplication table (row $\wedge$ column) is easy to compute: 

\[
\begin{tabular}{|l|c|c|c|}
\hline
    $ \wedge$       & $\omega$ & $ \Theta_+$ & $ \Theta_-$ \\
\hline
$dx^{345}$ & $0$                 & $0$                   & $dx^{123456}$         \\[3pt]
\hline
$dx^{156}$ & $0$                 & $0$                   & $dx^{123456}$         \\[3pt]
\hline
$dx^{126}$ & $0$                 & $dx^{123456}$         & $0$                   \\[3pt]
\hline
$dx^{135}$ & $0$                 & $-dx^{123456}$        & $0$                   \\[3pt]
\hline
$dx^{456}$ & $0$                 & $-dx^{123456}$        & $0$                   \\[3pt]
\hline
$dx^{234}$ & $0$                 & $dx^{123456}$         & $0$                   \\[3pt]
\hline
$dx^{123}$ & $0$                 & $0$                   & $-dx^{123456}$        \\[3pt]
\hline
$dx^{246}$ & $0$                 & $0$                   & $-dx^{123456}$        \\[3pt]
\hline
\end{tabular}
\]
From that we obtain a set of generators for $\Lambda^3_{12}H^\ast$:
$$
\left\{\begin{array}{cccc}
   dx^{126}+dx^{135}, & dx^{126}+dx^{456}, & dx^{126}-dx^{234}, & dx^{135}-dx^{456} \\ 
        dx^{135}+dx^{234}, & dx^{456}+dx^{234}, & dx^{345}-dx^{156}, & dx^{345}+dx^{123}\\ 
        dx^{345}+dx^{246}, & dx^{156}+dx^{123}, & dx^{156}+dx^{246}, & dx^{123}-dx^{246}. 
\end{array}\right\} 
$$
In terms of the  $\{w_i\}_{i=1}^8$ of  \eqref{eq:omega2}, this translates into
\begin{align*} 
    \Lambda_{12}^3H^\ast      
    &=\f{Span}\{ w_6\wedge dx^1,w_1\wedge dx^6,-w_4\wedge dx^2,w_3\wedge dx^5,-w_2\wedge dx^3,w_5\wedge dx^4,\\
    &\quad\quad\quad\quad w_4\wedge dx^5, w_1\wedge dx^3,-w_6\wedge dx^4,w_5\wedge dx^1,w_2\wedge dx^6,-w_3\wedge dx^2  \}
    \subset\rI.
\end{align*}
To see that $\Lambda_{6}^3H^\ast \subset\rI$, note that  $\omega\wedge dx^i=dx^{14i}+dx^{25i}+dx^{36i}$. Equivalently, in terms of $w_7$ and $w_8$ in \eqref{eq:omega2},
\begin{align*}
    \omega\wedge dx^i&= w_7\wedge dx^i+w_8\wedge dx^i+3dx^{36i}.
\end{align*} 
Let us have a closer look at terms of the form $dx^{36i}$. For $i=3$ and $i=6$, clearly $dx^{36i}=0$; for  $i=1$ and $i=4$, necessarily  $dx^{36i}=\pm w_7\wedge dx^i$;  for $i=2$ and $i=5$, one has  $dx^{36i}=\pm w_8\wedge dx^i$, hence in all instances  
 $\Lambda_{6}^3H^\ast  =\{w_7\wedge dx^i+w_8\wedge dx^i+3dx^{36i}\}_{i=1}^6 
\subset\rI$.  
Finally,  we observe that:
$$
J^\ast(X_i\lrcorner \Theta_+) =-J^\ast(J^2X_i\lrcorner\Theta_ +)=-JX_i\lrcorner J^\ast\Theta_+=-JX_i\lrcorner\Theta_-=X_i\lrcorner \Theta^+,
$$
whereas 
\begin{align*}
    (X_i\lrcorner\Theta_+)\wedge\omega^2  
    &= X_i\lrcorner(\Theta_+\wedge\omega^2)+\Theta_+\wedge(X_i\lrcorner\omega^2)=\Theta_+\wedge(X_i\lrcorner\omega^2)\\
    &=\Theta_+\wedge(X_i\lrcorner\omega\wedge\omega+\omega \wedge X_i\lrcorner\omega)\\
    &=2(\Theta_+\wedge\omega)\wedge X_i\lrcorner\omega, 
\end{align*}
i.e.,  $(X_i\lrcorner\Theta_+)\in \Lambda^2_8 H^\ast$, and thus  $ \Lambda_{\Re}^3 H^\ast \subset\rI$. That $ \Lambda_{\Im}^3 H^\ast\subset\rI$ follows analogously.

For $(ii)$, it is still manifest that $\rL^0,\rL^1,\rL^2\subset\rI$. As to  $\rL^3$, we can inspect directly the basis elements $\{v_i\}_{i=1}^6$  for $\Omega^2_6$ introduced in  \eqref{eq:omega2}. We claim  that every element in the Darboux coordinate basis for $\Omega^{3}_H(\fg_E)$ is obtained from an element of $\Omega^1_6(\fg_E)$. Indeed,
$$
\left\lbrace \begin{array}{*{3}c}
dx^{124}=v_1\wedge dx^4, &       \; dx^{125}=\;\;\;v_1\wedge dx^5, &      \; dx^{134}=\;\;\;v_3\wedge dx^4   \\[3pt]
dx^{136}=v_3\wedge dx^6, &       dx^{145}=-v_1\wedge dx^1,         &      dx^{146}= -v_4\wedge dx^4          \\[3pt]
dx^{235}=v_5\wedge dx^5, &       \; dx^{236}=\;\;\;v_5\wedge dx^6, &      \; dx^{245}=\;\;\;v_2\wedge dx^5   \\[3pt] 
dx^{256}=v_6\wedge dx^5, &       \; dx^{346}=\;\;\;v_3\wedge dx^3, &      \; dx^{356}=\;\;\;v_6\wedge dx^6
\end{array}\right\rbrace \subset \rI,
$$
as well as  
$$
\left\lbrace \begin{array}{*{3}c}
 -2dx^{345} =v_2\wedge dx^6-v_6\wedge dx^4+v_4\wedge dx^5,          &       \;\;\;\;\;2dx^{156} =\;\; v_2\wedge dx^6-v_6\wedge dx^4-v_4\wedge dx^5, \\[3pt]
\;\;\;\;2dx^{126} = v_1\wedge dx^6+v_3\wedge dx^5+v_6\wedge dx^1, &      \;\;\;\;\;2dx^{135} =\;\;v_1\wedge dx^6+v_3\wedge dx^5-v_6\wedge dx^1,  \\[3pt]
-2dx^{456} =v_1\wedge dx^6+v_4\wedge dx^2+v_5\wedge dx^4,           &  -2dx^{234} =  v_1\wedge dx^6+v_4\wedge dx^2-v_5\wedge dx^4,   \\[3pt] 
-2dx^{123} =v_3\wedge dx^2-v_2\wedge dx^6-v_5\wedge dx^1,           &      -2dx^{246} = v_3\wedge dx^2+v_2\wedge dx^6+v_5\wedge dx^1, 
\end{array}\right\rbrace \subset \rI.
$$
Similarly, one can easily check that  $\Omega^{4}_H(\fg_E),\; \Omega^5_H(\fg_E)$ and $\Omega^6_{H}(\fg_E)$ lie in $\rI$, thus $\rL^k=0$ for $k\geq 4$. 
\end{proof} 
\begin{lemma}
\label{lem:mapa eta omega contraction}
The followings  maps preserve $\rI$ (whether $\rI$ be  generated by  $\Omega^2_8(\fg_E)$ or by $\Omega^2_6(\fg_E)$)  
$$
i_\xi\colon\Omega^\bullet(\fg_E)\to\Omega^{\bullet-1}(\fg_E),  
\quad
\eta\wedge \colon\Omega^{\bullet}(\fg_E)\to\Omega^{\bullet+1}(\fg_E) 
\qandq   
\omega\wedge\colon\Omega^{\bullet}(\fg_E)\to\Omega^{\bullet+2}(\fg_E),
$$
so these maps  descend to  the quotient $L^\bullet$.  
\end{lemma}
\begin{proof} 
If $\alpha\wedge\beta\in \rI$, for some $\alpha\in\Omega^\bullet(\fg_E)$ and $\beta\in \Omega^{2}_8(\fg_E)$,  
\begin{align*}
    i_\xi( \alpha\wedge\beta)
    &=i_\xi\alpha\wedge\beta+(-1)^{\f{deg}(\alpha)}\alpha\wedge i_\xi\beta\\
    &= i_\xi\alpha\wedge\beta\in \rI.
\end{align*} 
Therefore  the   map   $i_\xi$ is a well-defined contraction of bi-degree $(0,-1)$ on $\rL^\bullet$. We also denote by $L_H^k$ the kernel of contraction inside of $\rL^k$, so that $\rL^{k,0}=\rL^k_H$. 
\end{proof}
We denote the induced maps on   the quotient  $\rL^\bullet$  as follows:  
\begin{equation}
    \label{eq:mapa eta omega}
    L_\eta:=\eta\wedge \colon\rL^{\bullet}(\fg_E)\to\rL^{\bullet+1}(\fg_E) 
\qandq   
L_\omega:=\omega\wedge\colon\rL^{\bullet}(\fg_E)\to\rL^{\bullet+2}(\fg_E),
\end{equation}
In summary, by Lemma \ref{prop:Lidentification} we have the following 
\begin{proposition}
\label{prop:complexDeformation}
If a connection $A$ has curvature  $F_A\in\Omega^2_8(\fg_E)$, the associated complex \eqref{eq:complexDeformation} is     elliptic: 
$$
\begin{tikzcd}
 \rL^\bullet \colon \quad
 0\arrow[r] 
 &\rL^0 \arrow[r,"D_0"]
 & \rL^1 \arrow[r, "D_1"]
 & \rL^2 \arrow[r, "D_2"]
 &\rL^3 \arrow[r]
 &  0
\end{tikzcd}.
$$
\end{proposition}
\begin{proof}
Denote by $ \pi\colon T^\ast M\setminus\{0\}\to M$ the fibre bundle obtained by removal of the zero section in $T^*M$, and by   $p: \Omega^\bullet(\fg_E)\to \rL^\bullet$ the quotient projection, i.e., $D_k=p\circ d_A^k$. Given $\varsigma\in  T^\ast M\setminus\{0\}$, the first order symbol functor $\sigma_1$ satisfies:  
$$
\sigma_1(p \circ d_A^k)=p\left(\sigma_1(d_A^k)\right)  =p(\varsigma\wedge\cdot),
\qforq
k=0,\dots,3.
$$
The fibre $\pi^\ast\left( \rL^k\right)_\varsigma$ is isomorphic to $\Lambda^k T_x^\ast M  / \rI$, where $\rI=\langle \Omega^2_8\rangle$,  hence the associated $1$-symbol is
 $$
 \sigma_1(D_k)(x,\varsigma)=p(\varsigma\wedge\cdot) \colon  \Lambda^k T_x^\ast M  / \rI \to \Lambda^{k+1} T_x^\ast M  / \rI,
 $$
and we assert that the associated symbol complex is exact:  
\begin{equation}
    \label{eq:symbol complex L}
    \begin{tikzcd}
   \sigma_1(\rL^\bullet) \colon\quad
   0                            \arrow[r] 
 & \Lambda^0 T_x^\ast M  / \rI   \arrow[r, "p(\varsigma\wedge \cdot)"]
 & \Lambda^1 T_x^\ast M  / \rI  \arrow[r, "p(\varsigma\wedge \cdot)"]
 & \Lambda^2 T_x^\ast M  / \rI  \arrow[r, "p(\varsigma\wedge \cdot)"]
 & \Lambda^3 T_x^\ast M  / \rI \arrow[r                           ]
 &  0
\end{tikzcd}.
\end{equation}
Using the identifications from Proposition  \ref{prop:Lidentification}, the above complex becomes 
$$
\begin{tikzcd}
  \Lambda^0(T_x^\ast M)                                           \arrow[r, "\varsigma\wedge  "]
 & \Lambda^1(T_x^\ast M)                                           \arrow[r, "p(\varsigma\wedge \cdot)"]
 & (\Lambda^2_{6\oplus 1} \oplus\eta\wedge(\Lambda^1_H ))(T_x^\ast M)   \arrow[r, "p(\varsigma\wedge \cdot) "]
 &  \eta\wedge(\Lambda^2_{6\oplus 1}(T_x^\ast M)).
 \end{tikzcd}
$$
Exactness  at positions $k=0,1,2$ is shown by the same argument, so we prove explicitly for  $k=2$ and $k=3$.

\medskip
\noindent \underline{$k=2$}: 
Let $\alpha\in (\Lambda^2_{6\oplus 1} \oplus\eta\wedge(\Lambda^1_H ))(T_x^\ast M)$ such that $\varsigma\wedge \alpha\in \langle\Lambda^2_8\rangle$. Since $\varsigma$ is non zero, this forces to  $\varsigma\wedge \alpha=0$, and we know   that the de Rham  complex is elliptic, so there exists $\beta\in \Lambda^1(T_x^\ast M) $ such that $\varsigma\wedge\beta=\alpha$. Hence we have exactness at $ (\Lambda^2_{6\oplus 1} \oplus\eta\wedge(\Lambda^1_H ))(T_x^\ast M)$.  

\medskip
\noindent \underline{$k=3$}: 
This is similar to the $k=2$ case, but we need to show  that $\beta\in(\Lambda^2_{6\oplus 1} \oplus\eta\wedge(\Lambda^1_H ))(T_x^\ast M)$. A priori, $\beta=\beta_1+\beta_2$, with
$$
\beta_1\in (\Lambda^2_{6\oplus 1} \oplus\eta\wedge(\Lambda^1_H ))(T_x^\ast M)\qandq \beta_2\in (\Lambda^2_{8}\oplus\Lambda_V)(T_x^\ast M).
$$
Since $\beta_2$ projects to zero under $p$, we have $\alpha=p(\varsigma\wedge\beta_1)$, so \eqref{eq:symbol complex} is exact. 
\end{proof}

The above complex in Proposition \ref{prop:complexDeformation} is referred to as the  \emph{associated complex}  to the  selfdual contact  instanton   $A$ (Proposition \ref{prop:complexDeformation} can be shown in the same way if  $F_A\in\Omega^2_6(\fg_E)$).  We denote by $H^k:=H^k(\rL)$ the cohomology groups of $\rL^\bullet$  \eqref{eq:complexDeformation},   $H^1 $ can be interpreted as the space of infinitesimal deformations of the contact instanton $A$, so that 
\begin{equation}
    \label{eq: h1}
    h^1:=\dim_\R (H^1) 
\end{equation}
represents  the expected dimension of the moduli space.

\subsection{Deformation  theory of  SDCI }
 \label{sec:Deformation theory}

The  splitting \eqref{eq:TangentDecomposition} of the  tangent space  $TM=H\oplus N_\xi$  by the Reeb vector field defines a bi-grading on 
$
\Omega^\bullet(\fg_E)$:
\begin{equation}
\label{eq:bi-grading}    
\Omega^{k,d}(\fg_E):=\Gamma(M,\Lambda^k H^\ast\otimes\Lambda^d N_\xi^\ast\otimes \fg_E),
\qforq
(k,d)\in \{0,\dots,6\}\times\{0,1\}.
\end{equation}
In particular,  
$$
 \Omega^{k,0}(\fg_E)=\Omega^k_H(\fg_E)
 \qandq 
 \Omega^{k,1}(\fg_E)=\eta\wedge\Omega^k_H(\fg_E),
 \qforq
 k=0,1,2,3. 
$$  
Since the ideal $\rI $ is bi-graded, the bi-grading descends to the quotients  $\rL^\bullet$ to define components $\rL^{k,d}$. In our case of main interest, $\rI=\langle\Omega^2_8(\fg_E) \rangle $ induces the complex \eqref{eq:complexDeformation}:
\begin{equation} 
    \label{eq:Lk,0andL^k1}
   \begin{matrix}
\Omega^{0}_H(\fg_E) &  \to & \Omega^{1}_H(\fg_E) \oplus \eta\wedge\Omega^{0}_H(\fg_E) & \to &\Omega^{2}_{6\oplus 1}(\fg_E) \oplus   \eta\wedge\Omega^{1}_H(\fg_E) &\to & \eta\wedge \Omega^{2}_{6 \oplus 1}(\fg_E)\\
   & & \begin{matrix}
{\rotatebox{90}{\scalebox{1}[1]{$\cong$}}}\\
{\rL^{1,0}}
\end{matrix}\quad\quad\quad \begin{matrix}
{\rotatebox{90}{\scalebox{1}[1]{$\cong$}}}\\
{\rL^{0,1}}
\end{matrix}  \quad&  & \begin{matrix}
{\rotatebox{90}{\scalebox{1}[1]{$\cong$}}}\\
{\rL^{2,0}}
\end{matrix}\quad\quad\quad \begin{matrix}
{\rotatebox{90}{\scalebox{1}[1]{$\cong$}}}\\
{\rL^{1,1}}
\end{matrix} \quad & & \begin{matrix}
{\rotatebox{90}{\scalebox{1}[1]{$\cong$}}}\\
{\rL^{2,1}}
\end{matrix}   \quad
  \end{matrix} 
\end{equation}
  Define in the quotient the followings maps 
\begin{equation}
\label{eq:DV DT}
\begin{array}{rrcl}
    D_V\colon &\rL_H^\bullet &\to      &\rL_H^\bullet  \\
    & \alpha & \mapsto &i_\xi D\alpha,
\end{array}
\qandq
\begin{array}{rrcl}
    D_T\colon &\rL_H^\bullet &\to      &\rL_H^{\bullet+1}  \\
    & \alpha & \mapsto & D\alpha-\eta\wedge D_V\alpha,
\end{array}    
\end{equation} 
\begin{lemma}
  Let $D_T$ and $D_V$ be the operators defined in \eqref{eq:DV DT}, then we have the following identities  
\begin{equation}
    \label{eq:DVcomuts DT}
    D_TD_V=D_VD_T
\end{equation}
  and 
\begin{equation}
    \label{eq:DT2}
  D_T^2\alpha=-\omega\wedge D_V\alpha\quad\text{for}\quad \alpha\in\Omega^0(\fg_E).   
\end{equation}
\end{lemma}
\begin{proof}
  From $D^2=0$, we obtain  
\begin{align*}
  0 &=(D_T+\eta\wedge D_V)^2=D_T^2+D_T(\eta\wedge D_V)+\eta\wedge D_V(D_T)+(\eta\wedge D_V)^2\\
  &=D_T^2+\omega\wedge D_V-\eta\wedge D_T D_V+-\eta\wedge D_V D_T+\eta\wedge D_V(\eta\wedge D_V) \\ 
  &= D_T^2+\omega\wedge D_V+\eta\wedge(D_V D_T-D_T D_V)+\eta\wedge(D_V\eta\wedge D_V-\eta \wedge D_V^2)\\
  &= D_T^2+\omega\wedge D_V+ \eta\wedge(D_V D_T-D_T D_V),
\end{align*}
  so, $D_T^2+\omega\wedge D_V=0= D_V D_T-D_T D_V$,  hence \eqref{eq:DVcomuts DT} and \eqref{eq:DT2} follow.
\end{proof}
  The basic forms in $\rL^k$ are denoted by $\rL^k_B=\Ker(D_V)\subset \rL^k$ and, by Proposition \ref{prop:Lidentification},     $\Ker(D_V) $ can be identified with $\Omega^\bullet_B(\fg_E)$. Since  $D_V$ and $D_T$ commute \eqref{eq:DVcomuts DT}, 
$$
  D_T(\Omega^\bullet_B(\fg_E))\subset \Omega^{\bullet+1}_B(\fg_E).
$$ 
Moreover, $D_T$   restricts to $D_B\colon \rL_B^k\to \rL_B^{k+1}$, so  $D_B$ defines a \emph{basic deformation complex}. Considering $\rI=\langle\Omega^2_8(\fg_E)\rangle$,  the associated basic complex   \eqref{eq:BasicComplexDeformation2}  becomes  
$$
\begin{tikzcd}
 \rL^\bullet_B\colon 0\arrow[r] 
 &\Omega^0_B( \fg_E) \arrow[r,"D_B"]
 & \Omega^1_B( \fg_E) \arrow[r, "D_B"]
 & (\Omega^2_{6\oplus1})_B (\fg_E) \arrow[r] 
 &  0.
\end{tikzcd}
$$ 
Denote by  $H_B^\bullet:=H^\bullet(\rL_B)$ the cohomology of  the basic complex $ \rL^\bullet_B$. This is not an elliptic complex, yet it is elliptic transversely to the Reeb foliation. In particular, its cohomology  $H_B^\bullet$ is finite-dimensional 
\cite[Theorem~ 3.2.5]{kacimi1990operateurs}, and  we will see that this complex computes the dimension of the moduli space of contact instantons.  
\begin{lemma} 
Let $D_T$ and $D_V$ be the operators defined in \eqref{eq:DV DT}, then the formal adjoints of $D_V$ and $D_T$ are given by:
\begin{equation}
    \label{eq:adjoint DV}
    D_V^\ast=-D_V,
\end{equation} 
\begin{equation}
    \label{eq:adjoint DT}
    D_T^\ast=-\ast_T D_T\ast_T.
\end{equation}
Furthermore, 
\begin{equation}
\label{eq:comD_VD_T}
D_V D_T^\ast=D_T^\ast D_V.
\end{equation}
\end{lemma}
\begin{proof}
  For \eqref{eq:adjoint DV} note that, since $\cL_\xi$ coincides with $D_V$ on $\Omega_H^{\ast}(\fg_E)$, if $\xi$ is a Killing vector then $D_V$ and $\ast_T $ commute, and so 
\[
0=\int_M \cL_\xi(\alpha,\ast_T\beta)\wedge\eta=\int_M(D_V\alpha,\ast_T\beta)\wedge\eta+\int_M(\alpha,D_V\ast_T\beta)\wedge\eta.
\]
  Furthermore, since $D_TD_V=D_VD_T$ we obtain \eqref{eq:comD_VD_T}. To show \eqref{eq:adjoint DT}, let  $\alpha\in \Omega^{k-1}_H(\fg_E)$ and $\beta\in \Omega^k_H(\fg_E)$:  
\begin{align*}
    ( \alpha,(\ast_T D_T\ast_T)\beta ) &=\int_M\alpha\wedge\ast_T^2\left(D_T\ast_T\beta\right)\wedge\eta \\
    &=\int_M \alpha\wedge\ast_T^2\left[D(\ast_T\beta)-\eta\wedge D_V(\ast_T \beta) \right]\wedge\eta\\
    &=\int_M \alpha\wedge\ast_T^2 D(\ast(\eta\wedge\beta)) \wedge\eta\\
    &=\int_M \alpha\wedge\ast_T\left[i_\xi(\ast D\ast(\eta\wedge\beta))\right]\wedge\eta=(\alpha,i_\xi D^\ast \eta(\beta))\\
    &=(\eta^\ast Di_\xi^\ast(\alpha),\beta),
\end{align*}  
and the assertion  follows by observing that $i_\xi D\eta(\alpha)=i_\xi(D(\eta\wedge\alpha))=i_\xi(\omega\wedge\alpha-\eta\wedge D\alpha)=-D_T\alpha$.
\end{proof}

We now adapt a number of  fundamental insights from \cite[Proposition~3.3]{Baraglia2016}. We begin by introducing the transverse Laplacian on the complex $\rL^\bullet$:
\begin{definition}
\label{def:TransversalLaplacian}
The \emph{Laplacian} of $D$, with respect to the inner product defined in \eqref{eq:innerProduct}, is \begin{equation}
\label{eq:laplacian}
    \Delta
    :=DD^\ast+D^\ast D
    \colon \quad
    \rL^\bullet\to \rL^\bullet,
\end{equation}
and  the \emph{transverse Laplacian} is defined  by: 
\begin{equation}
\label{eq:laplacianT}
    \Delta_T
    :=D_TD^\ast_T+D^\ast_T D_T-D_V^2
    \colon \quad 
    \rL^\bullet\to \rL^\bullet.
\end{equation}
\end{definition} 

Clearly $\Delta_T$ and   $\Delta$ have the same symbol, so $\Delta_T$    is an elliptic operator. Let us denote the spaces of  $\Delta-$harmonic  and $\Delta_T-$harmonic forms,  respectively, by 
$$
\cH^{k}:=\ker(\Delta)\subset\rL^k
\qandq
\cH^{k}_T:=\ker(\Delta_T)\subset\rL^k. 
$$   
Since $\langle \Delta_T\alpha,\alpha \rangle=\Vert D_T\alpha \Vert^2+\Vert D_T^\ast\alpha \Vert^2+\Vert D_V\alpha \Vert^2$,
\begin{equation}
\label{eq:DT-harmonic}
    \alpha\in \rL^{k,0}
    \quad\text{is}\quad  
    \Delta_T\text{-harmonic if, and only if,}
    \quad 
    D_T\alpha=D_T^\ast\alpha=D_V\alpha=0.
\end{equation}
On basic forms, in particular, $\Delta_T$-harmonicity   is equivalent to $\Delta$-harmonicity. 
Note that $\Delta_T$ respects the bi-grading defined in \eqref{eq:bi-grading}, hence we can split 
$\cH^k_T$ into components $\cH^{k,d}_T$.
Since    $L_\eta\colon \rL^{k,0} \to\rL^{k,1}$ \eqref{eq:mapa eta omega} is an isomorphism,   we know from the outset that 
\begin{equation}
\label{eq:harmonicsByEta}
    \cH^{k,0}_T\cong \cH^{k,1}_T=\eta\wedge\left(\cH^{k,0}_T\right), \qforq k\in\{1,\cdots, 6\}.
\end{equation}
\begin{lemma}
\label{lem:previo1} 
There exists an isomorphism  $\phi\colon  \cH_T^{k,0}\tilde{\to} H_B^k$, where $\cH_T^{k,0}$ is  just above defined  and   $H_B^k$ is the $k$-th cohomology group of the basic complex \eqref{eq:BasicComplexDeformation2}.
\end{lemma} 
\begin{proof}
  Let the morphism $ \phi\colon  \cH_T^{k,0}\to H_B^k$ send $\alpha\in \cH_T^{k,0}$ to its equivalence  class $[\alpha]\in H_B^k$. This is well-defined, by \eqref{eq:DT-harmonic}, since $\cH^k_T\subset \ker(D_V)=\Omega^k_B(\fg_E)$ and $D $ coincides with $D_T$ on basic forms, so indeed $\cH^k_T\subset \Ker(D_B)$, therefore  $\alpha$ defines a equivalence  class in $H^k_B$.\\ 

  We first check   injectivity of $\phi$. If $\phi(\alpha)=[\alpha]=0\in H_B^k$, then   
$$
\alpha=D_B\beta=D_T\beta+\eta\wedge D_V\beta,
\qwithq
\beta\in \Omega^{k-1}_B(\fg_E).
$$     
 $D_V\beta=0 $ because   $\beta$ is a  basic form, hence 
$$
\alpha=D_T\beta= D\beta.
$$
Applying $D^\ast_T$ to $\alpha$ follows that $0=D^\ast_T\alpha=D^\ast_T D_T\beta$, now taking inner product with $\beta$ in the last equality we obtain that $D_T\beta=0= \alpha$. \\

To check surjectivity,  let $\alpha$ be closed and basic, i.e, $\alpha\in \rL^{k,0}$ and $D_T\alpha=0=D_V\alpha.$  Elliptic theory implies that $\alpha=\beta+\Delta_T \gamma$,
where $\beta\in\Ker(\Delta_T)$. From  \eqref{eq:comD_VD_T},  we have $D_V\Delta_T=\Delta_TD_V$, hence 
$$
D_V\alpha=0=D_V\beta+D_V\Delta_T\gamma= \Delta_TD_V\gamma.
$$
It follows that $D_V\gamma$ is $\Delta_T$-harmonic, in particular $D_V^2\gamma=0$, and so  
\begin{align*}
    0
    &= D_T\alpha
    =D_T\Delta_T\gamma
    = D_T D_T^\ast D_T\gamma+D_T^2 D_T^\ast\gamma-D_TD_V^2\gamma\\
    &=-\omega\wedge D_V D_T^\ast\gamma+ D_TD_T^\ast D_T\gamma=-\omega\wedge D_T^\ast D_V \gamma+ D_TD_T^\ast D_T\gamma\\
    &=D_TD_T^\ast D_T\gamma.
\end{align*} 
This implies 
$$
0=\langle D_TD_T^\ast D_T\gamma,D_T\gamma \rangle=\Vert D_T^\ast D_T\gamma\Vert^2,
$$
so, $\alpha=\beta+D_T D_T^\ast\gamma$. Now, $D_T^\ast\gamma$ is a basic form, since $D_VD_T^\ast\gamma=D_T^\ast D_V\gamma=0$, so in fact $\alpha=\beta+D_B( D_T^\ast\gamma)$. This shows that every element in $H_B^k$ has a $\Delta_T$-harmonic representative, thus  $\phi$ is surjective.
\end{proof}

\begin{lemma}
\label{lem:complex A}
Consider the vector spaces  $A^k:=\cH_T^{k,0}\oplus\cH_T^{k-1,0}$, together with the differential $\hat{\rd} $   defined by:
 \begin{equation}
      \label{eq:complex A}
\begin{array}{crcl}
    \hat{\rd}\colon   & A^k & \to     & A^{k+1}   \\
    & (\alpha,\beta) & \mapsto & \hat{\rd}(\alpha,\beta):=  (\omega\wedge\beta,0).
\end{array}
\end{equation}
The following hold:
\begin{enumerate}[(i)]
    \item 
    $(A^\bullet,\hat{\rd})$  form a  chain complex.
    
    \item 
    There is a chain  map   $j^\bullet\colon (A^\bullet,\hat{\rd})\to(\rL^\bullet,D)$ into the complex \eqref{eq:complex L D}, defined by
\begin{equation}
    \label{eq:chain map}
    j^k\colon (\alpha,\beta)\in A^k\mapsto \alpha+\eta\wedge\beta\in  \rL^k,
    \qforq
    k=0,1,2,3.
\end{equation}

    \item 
    $\displaystyle \chi(A)=\sum_{k=0}^3(-1)^k\dim(H^k(A))=0.$
\end{enumerate}
\end{lemma}
\begin{proof} \quad
\begin{enumerate}[(i)]
    \item
    Clearly $\hat{\rd}^2=0$, and $\hat{\rd}$ is well-defined in view of the isomorphism $\cH_T^{k,0}\cong H_B^k$ from Lemma \ref{lem:previo1}:
    \[\begin{diagram}
        \node{(\alpha,\beta)\in\cH_T^{k,0}\oplus\cH_T^{k-1,0}} \arrow[2]{e,t,..}{} \arrow{s,l}{\rotatebox{90}{\scalebox{1}[1]{$\cong$}}} \arrow{ese}\node[2]{(\omega\wedge\beta,0)\in\cH_T^{k+1,0}\oplus\cH_T^{k,0}} \\
        \node{H_B^k\oplus H_B^{k-1}}\arrow[2]{e,t,..}{} \node[2]{(\omega\wedge\beta,0)\in H_B^{k+1}\oplus H_B^{k}}\arrow{n,r}{\rotatebox{90}{\scalebox{1}[1]{$\cong$}}}
    \end{diagram}.\]

    \item
    The maps $j^k$ in \eqref{eq:chain map} fit in the following diagram:
    \begin{equation}
    \label{eq:diagram complex A,d}
    \begin{diagram}
        \node{\cH_T^{0,0}\oplus\{0\}} \arrow{e,t}{\hat{\rd}^0=0} \arrow{s,l}{j^0}  
        \node {\cH_T^{1,0}\oplus\cH_T^{0,0}}\arrow{e,t}{\hat{\rd}^1} \arrow{s,r}{j^1}\node {\cH_T^{2,0}\oplus\cH_T^{1,0}}\arrow{e,t}{\hat{\rd}^2=0} \arrow{s,r}{j^2}\node {\{0\}\oplus\cH_T^{2,0}} \arrow{s,r}{j^3}\\
        \node{\rL^0}\arrow{e,b}{D} \node{\rL^1}\arrow{e,b}{D}\node{\rL^2}\arrow{e,b}{D}\node{\rL^3.}
    \end{diagram}
    \end{equation}
    To see that \eqref{eq:diagram complex A,d}  is commutative, the only nontrivial step to check is $Dj^1=j^2\hat{\rd}.$ 
    Given $(\alpha,\beta)\in A^1$, by definition, $D\alpha=0=D_V\beta$, hence
    \begin{align*}
        Dj^1(\alpha,\beta)&=D(\alpha+\eta\wedge\beta)=D\alpha+\omega\wedge\beta-\eta
        \wedge D_V\beta\\
        &=\omega\wedge\beta=j^2(\omega\wedge\beta,0)\\
        &=j^2\hat{\rd}(\alpha,\beta).
    \end{align*}

    \item
    We know, from Lemma \ref{lem:previo1}, that each chain  $A^k$ in \eqref{eq:complex A} is finite-dimensional, for $k=0,1,2,3$, and 
    \begin{align*}
        \chi(A)
        &:=\sum_{k=0}^3(-1)^k\dim(A^k)
        = \sum_{k=0}^3(-1)^k \left(\dim\left(\cH^{k,0}_T\right) +\dim\left(\cH^{k-1,0}_T\right)\right)\\
        &= 0.        
    \end{align*}
    The assertion now follows from the rank-nullity Theorem.
\qedhere
\end{enumerate}
\end{proof} 

\begin{lemma}
\label{lem:lema auxiliar}
Let $\rL^k$, $k= 1,2 $ be the chain spaces  defined in Proposition \ref{prop:Lidentification}. For each $[\alpha]\in H^k$, its  harmonic representative can be written as:
\begin{equation}
    \label{eq: dem previo2}
    \alpha=\beta+\eta\wedge\gamma,
    \qwithq
    D\alpha=D^\ast\alpha=0,
\end{equation}
such that  $D_V\beta=D_T^\ast\beta=0$ and   $D_V\gamma=D_T\gamma=0$.
\end{lemma}
\begin{proof}
The harmonic representative of $[\alpha]\in H^k$ has the general form $\alpha=\beta+\eta\wedge\gamma$ by definition of $\rL^k$ (see Proposition \ref{prop:Lidentification} item $(i)$). We know from \eqref{eq:DV DT} that $D\alpha=D_T\alpha+\eta\wedge D_V\alpha$, so
\begin{align*}
    0=D\alpha &=D_T\beta+\eta\wedge D_V\beta+ \omega\wedge\gamma-\eta\wedge(D_T\gamma+\eta\wedge D_V\gamma)\\
    &=D_T\beta+\omega\wedge\gamma+\eta\wedge\left(D_V\beta- D_T\gamma \right) 
\end{align*} 
and 
\begin{align*}
    0=D^\ast\alpha&=D_T^\ast\beta+(\eta\wedge D_V)^\ast\beta+D_T^\ast (\eta\wedge\gamma)+(\eta\wedge D_V)^\ast(\eta\wedge\gamma)\\
    &=D_T^\ast\beta-D_V(i_\xi(\eta\wedge\gamma))\\
    &=D_T^\ast\beta-D_V\gamma.
\end{align*}
In summary, we obtain the following relations: 
\begin{equation}
    \label{eq:relations dem previo 2}
\begin{array}{ccc}
    (i)\; D_T\beta+\omega\wedge\gamma=0,     
    & (ii)\;  D_V\beta-D_T\gamma=0, 
    & (iii)\;D_T^\ast\beta-D_V\gamma=0.  
\end{array}
\end{equation} 
Applying $D_T$ to $(iii)$ and $D_V$ to $(ii)$, and commuting by \eqref{eq:DVcomuts DT},  we obtain $D_T D_T^\ast\beta-D_V^2\beta=0$. Taking the inner product with $\beta$, we have 
$$
\Vert D_V\beta\Vert^2 +\Vert D_T^\ast\beta\Vert^2=0, 
$$
therefore $D_V\beta=D_T^\ast\beta=0$, and so  $D_V\gamma=D_T\gamma=0$.  
\end{proof}

\begin{proposition}
\label{prop:previo2}
The chain map $j^\ast\colon (A^\bullet,\hat{\rd})\to (\rL^\bullet,D)$ defined in \eqref{eq:chain map} is a quasi-isomorphism, i.e., the induced map in cohomology $j^k\colon H^k(A^\bullet)\to H^k(\rL^{\bullet})$ is an isomorphism, for each $k=0,1,2,3$.  
\end{proposition}
\begin{proof} There are four cases to consider: 

\medskip
\noindent\underline{$k=0$}:  
This case  is trivial, since $j^0$ is essentially the inclusion map. 

\medskip\noindent \underline{$k=3$}: 
We must show that $H^3\cong \{0\}\oplus\cH^{2,0}_T$. From Hodge theory, we have isomorphisms $H^k\cong\cH^k$, so we  show that $\cH^3\cong \cH^{2,0}_T$. Indeed, $j^3$ is just  exterior multiplication by $\eta$, so
\begin{align*}
\cH^{2,0}_T &= \lbrace\gamma\in \rL^{2,0}=\rL^2_H\;\vert\; \Delta_T\gamma
             =0 \rbrace 
             = \left\lbrace \gamma\in  \Omega^{2}_6\oplus \Omega^{2}_{1}\; :\; D_T\gamma
             =D^\ast_T\gamma=D_V\alpha=0 \right\rbrace\\
            &\cong \left\lbrace \alpha\in \eta\wedge\left( \Omega^{2}_6\oplus \Omega^{2}_{1}\right) \; :\; D^\ast_T\alpha
             =D_V\alpha
             =0 \right\rbrace \\
            &=\cH^3\subseteq \rL^3.
\end{align*} 

\noindent \underline{$k=1$}: 
By \eqref{eq: dem previo2}, a class $[\alpha]\in H^1$ has a harmonic representative of the form, 
$$
    \alpha=\beta+\eta\wedge\gamma \in\Omega^1_H(\fg_E)\oplus\eta\wedge\Omega^{0}_H(\fg_E),
$$
and we first need to be sure that $\beta,\gamma$   are $\Delta_T$-harmonic. By     \eqref{eq:DT-harmonic} and Lemma \ref{lem:lema auxiliar}, it only remains to check that $D_T^{\ast}\gamma=0$ and  $D_T\beta=0$. The former holds because  
$\gamma\in \Omega^{0}_H(\fg_E)\subset \ker(D_T^\ast)$.
For the latter, take the inner product with $\beta$ after the following computation:
\begin{align*}
0 &=D^\ast_TD_T\beta+D_T^\ast(\omega\wedge\gamma)=D^\ast_TD_T\beta-\ast_TD_T\ast_T(\omega\wedge\gamma)\\
  &=D^\ast_TD_T\beta-\ast_T D_T(\ast(\eta\wedge\omega\wedge\gamma) =D^\ast_TD_T\beta-\ast_T D_T (\omega^2\wedge\gamma)\\
  &=D^\ast_TD_T\beta.
\end{align*}

Now, let us check that $j^1\colon H^1(A^\bullet)\to H^1(\rL^\bullet)$ is an isomorphism. From relation $(i)$ in \eqref{eq:relations dem previo 2}, we have    $\omega\wedge \gamma=0$, and thus  $(\beta,\gamma)\in\cH^{1,0}_T\oplus\cH^{0,0}_T$ and $\hat{\rd}(\beta,\gamma)=(\omega\wedge\gamma,0)=(0,0)$, i.e., $j^1$ is surjective. For injectivity, suppose that 
 $
[j^1(\beta,\gamma)]=[\beta+\eta\wedge\gamma]=0\in H^1(\rL^\bullet),
 $ 
so $\beta+\eta\wedge\gamma=D\delta$, for some $\delta\in \Omega^0(\fg_E)$, and so 
\begin{align*}
    0     
    &=\beta- D\delta+\eta\wedge D_V\delta +\eta\wedge\gamma-\eta\wedge  D_V\delta  
     =\beta-i_\xi(\eta\wedge D\delta)+\eta\wedge(\gamma-D_V\delta)\\
    &=\beta- D_T\delta +\eta\wedge(\gamma-D_V\delta).
\end{align*}
Comparing types, we have $\beta=D_T\delta$ and $\gamma=D_V\gamma$, but $\beta$ and $\gamma$ are $\Delta_T$-harmonic, so indeed $(\beta,\gamma)=(0,0)$.  

\medskip
\noindent \underline{$k=2$}: 
From Lemma  \ref{lem:lema auxiliar},
$ 
\alpha=\beta+\eta\wedge\gamma\in \left( \Omega^2_6\oplus\Omega^2_1\right)(\fg_E) \oplus\eta\wedge\Omega^1_H(\fg_E)=\rL^2, 
$
where $D_V\beta=D_T^\ast\beta=0$ and   $D_V\gamma=D_T\gamma=0$. Note that $\omega\wedge \gamma =0$, since  $\omega\wedge \gamma \in \rL^{3}=\rL^{2,1}\cong \eta\wedge(\Omega^2_6\oplus\Omega^2_1)(\fg_E)$, so relation $(i)$ in \eqref{eq:relations dem previo 2} implies that $D_T\beta=0$, i.e., $j^2\colon H^2(A^{\bullet})\to H^2(\rL^\bullet)$ is surjective. For injectivity, recall that the complex  $(\rL^\bullet,D)$ in \eqref{eq:complex L D} is elliptic [Proposition \ref{prop:complexDeformation}] on a compact odd-dimensional manifold, so 
$$
0=\sum(-1)^k\dim(H^k(\rL^\bullet)),
$$ 
similarly,   $0=\sum(-1)^k\dim(H^k(A^\bullet))$ (see item (iii) of Lemma   \ref{lem:complex A}). Since $j^k$ is an isomorphism for $k=0,1, 3$,   we  conclude that $\dim(H^2(\rL^\bullet))=\dim(H^2(A^\bullet))$, hence $j^2$ is also an isomorphism.  
\end{proof} 

\begin{remark}
Lemma   \ref{lem:previo1} and Proposition \ref{prop:previo2} can be shown in the same fashion if we adopt the ideal  $\rI=\f{Span}\{\Omega^2_6\}$, instead of $\rI=\f{Span}\{\Omega^2_8\}$.
\end{remark}

\begin{proposition}
\label{prop:omegaMap}
Assume the Sasakian manifold $M^7$ is closed (see Remark \ref{rem:closed}). Then the map 
$$
L_\omega\colon [\alpha] \in H^0_B \mapsto [\alpha\otimes\omega]\in H^2_B, 
$$ 
induced in cohomology by $L_\omega \colon \rL^k\to\rL^{k+2}$ [cf. Lemma \ref{lem:mapa eta omega contraction}], is injective. 
\end{proposition} 
\begin{proof} 
Since $D\omega=0$, the exterior product map for $\omega$ is a well-defined map in cohomology. Given $\alpha\in \ker(\omega\wedge \cdot)\subset H^0_B$, clearly $[\alpha\otimes \omega]=0\in H^2_B$, so there exists some $\beta\in \Omega^1_B$ such that $\alpha\otimes\omega=D_B\beta$. Then 
\begin{align*}
    \Vert \alpha \Vert^2 
    &=\displaystyle \int_M  \langle\alpha\wedge\ast\alpha\rangle =\displaystyle \int_M  \alpha\wedge\ast\alpha\wedge\frac{ \eta\wedge\omega^3}{3!}\\
    &= \frac{1}{6}\displaystyle\int_M \alpha\omega\wedge\ast_T(\alpha\omega) \wedge \eta \\
    &= \frac{1}{6}\displaystyle\int_M D_B\beta\wedge\alpha\omega^2 \wedge\eta. 
\end{align*}
Therefore, it suffices to show that the integrand is exact. Indeed:
\begin{align*}
    d\left(\langle\beta \wedge\alpha\omega^2\rangle \wedge\eta\right) 
    &=\left( D_B\beta\wedge \alpha\omega^2 +\beta\wedge D_B(\alpha\omega^2)\right)\wedge\eta+ \beta\wedge \alpha\omega^2\wedge\omega\\
    &=(D_B\beta\wedge\alpha\omega^2)\wedge\eta +(\beta\wedge\alpha\omega^2)\wedge \omega\\
    &= (D_B\beta\wedge\alpha\omega^2) \wedge\eta.\qedhere
\end{align*}
\end{proof} 

At an instanton, the cohomologies $H^\bullet$ and $H^\bullet_B$ of the complexes  \eqref{eq:complexDeformation} and \eqref{eq:BasicComplexDeformation2}, respectively, fit in a Gysin sequence  analogous to  \cite[Theorem~10.13]{tondeur2012foliations}: 

\begin{proposition}
\label{prop:gysinSequence}
At a SDCI, the complex \eqref{eq:complex A} induces a long exact sequence in cohomology: 
\begin{equation}
  \label{eq:gysinSequence}
    \begin{tikzcd}
 \cdots \arrow[r]
 & H^{k-2}_B(A) \arrow[r, "\omega\wedge"] 
 & H^{k}_B(A) \arrow[r]
 & H^k(A) \arrow[r]
 &  H^{k-1}_B(A) \arrow[r,"\omega\wedge"] 
 &  H^{k+1}_B(A) \arrow[r]
 & \cdots
\end{tikzcd}
\end{equation} 
\end{proposition} 
 \begin{proof} 
In order to show the exactness of \eqref{eq:gysinSequence}, we use the isomorphisms $H^\bullet(A)\cong H^\bullet $ [Proposition   \ref{prop:previo2}] and  $H^k_B\cong\cH^{k,0}_T$ [Lemma \ref{lem:previo1}]. We   recall that the differential $\hat{\rd}^k\colon A^k\to A^{k+1}$ from \eqref{eq:complex A} maps $(\alpha,\beta)\in\cH_T^{k,0}\oplus\cH_T^{k-1,0}$ to $(\omega\wedge\beta,0)\in\cH_T^{k+1,0}\oplus\cH_T^{k,0}$.    
\begin{itemize}
    \item 
First note that $ H^0(A)=\ker(\hat{\rd}^0)=\cH^{0,0}_T\cong H^0_B,$ hence  we obtain exactness  in 
$$
\begin{tikzcd}
 0 \arrow[r]
 & H^0_B \arrow[r] 
 & H^0 \arrow[r]
 & 0
\end{tikzcd}.
$$ 
Also note that $H^3(A)=\cH^{2,0}_T\cong H^2_B,$ hence we obtain exactness  in 
$$
\begin{tikzcd}
 0 \arrow[r]
 & H^3 \arrow[r] 
 & H^2_B \arrow[r]
 & 0
\end{tikzcd}.
$$ 
It remains to show  the exactness of 
$$
 \begin{tikzcd}
  0 \arrow[r, "\omega\wedge"] 
 &  H^1_B  \arrow[r]
 &  H^1    \arrow[r]
 &  H^0_B  \arrow[r,"\omega\wedge"] 
 &  H^2_B  \arrow[r]
 &  H^2    \arrow[r]
 &  H^1_B  \arrow[r, "\omega\wedge"]
 &  0.
\end{tikzcd}
$$
We proceed from left to right.
    \item
From Proposition $\ref{prop:omegaMap}$, the   map $L_\omega  \colon  H^0_B\to H^2_B$ is injective. 
Moreover,
\begin{align*}
    H^1(A)=\ker(\hat{\rd}^1)
    &=\{(\alpha,\beta)\in \cH_T^{1,0}\oplus\cH_T^{0,0} \vert (\omega\wedge\beta,0)=(0,0) \} \\
    &\cong\cH^{1,0}_T\oplus\{0\}\cong H^1_B,
\end{align*}
so mapping $H^1\to0$  gives exactness at $H^1_B$, $H^1$ and $H^0_B$. 

    \item
The  map $H^2_B\to H^2$ is induced by the inclusion $\rL^2_B\to\rL^2$, and its kernel consists of exact basic forms. By Proposition \ref{prop:previo2}, these are identified  with the image  of $L_\omega\colon \cH^{0}_T \to   \cH^{2,0}_T$, hence \eqref{eq:gysinSequence} is exact at $H^2_B$. 

    \item
We assert that  the map 
$$
(i_\xi)_\ast\colon H^2\to H^1_B,
$$
induced in cohomology by $i_\xi\colon L^2\to L^2_B$, is surjective. Indeed, if $\alpha\in \rL^1$ is basic and closed, the form $\eta\wedge \alpha$ belongs to $\rL^2$, and $D(\eta\wedge\alpha)=\omega\wedge\alpha$ is a basic $3$-form, i.e., $\omega\wedge\alpha\in \rL^{3,0}$[cf. \eqref{eq:Lk,0andL^k1}], so it is zero, and   $(i_\xi)_\ast[\eta\wedge\alpha]=[\alpha]$, as claimed.  

    \item
Finally, $H^2(A)=\frac{\cH_T^{2,0}\oplus\cH_T^{1,0}}{\im(\hat{\rd}^1)}$ and
\begin{align*}
    \im(d^1)
    &= \{(\alpha,\beta)\in \cH_T^{2,0}\oplus\cH_T^{1,0} \vert d(\theta,\gamma)=(\alpha,\beta),\quad\text{for some}\quad (\theta,\gamma)\in    \cH_T^{1,0}\oplus\cH_T^{0,0}  \}\\
    &= \{(\alpha,\beta)\in \cH_T^{2,0}\oplus\cH_T^{1,0} \vert  (\omega\wedge\gamma,0)=(\alpha,\beta),\gamma\in\cH_T^{0,0}  \}\\
    &\cong \omega\wedge(\cH^{0,0}_T)\oplus \{0\},
     \end{align*}
so,  $H^2\cong \frac{H^2_B}{\omega\wedge(\cH^{0,0}_T)}\oplus H^1_B\cong \frac{H^2_B}{\omega\wedge(H^0_B)}\oplus H^1_B$. \qedhere
\end{itemize}
\end{proof}

The following immediate consequence of Proposition  \ref{prop:gysinSequence} proves part \emph{(ii)} of Theorem \ref{thm:Intro}:  

\begin{corollary}
\label{coro:H1=Hb}
    At a SDCI, the inclusion  $\rL_B^1\to \rL^1$ induces an isomorphism $H^1_B\cong H^1$, thus the expected dimension of the moduli space \eqref{eq:modulispace} is $h^1_B:=\dim H^1_B$ \eqref{eq:BasicComplexDeformation2}.  
\end{corollary}

\subsection{Obstruction and smoothness}
\label{sec: obstructions}  

In this section we will establish that, if  the second cohomology group $H^2_B$ of the basic complex in \eqref{eq:BasicComplexDeformation2}  vanishes, then  the moduli space $\cM^\ast$ of irreducible SDCI defined in  \eqref{eq:modulispace} has a local structure of a smooth manifold (Corollary \ref{cor:ModuliSuave}). Our approach follows the general scheme of \cite[\S ~4.2]{Donaldson1990}, adapting to $7$ dimensions some crucial insights from \cite[\S 3.1 \& \S3.2]{Baraglia2016}.  
\subsubsection{The obstruction map}
\begin{lemma}
\label{lem:Maurer-Cartan}
Let $A$ be a SDCI and  $\alpha\in \rL^1\cong\Omega^1(\fg_E)$ [cf.  Proposition \ref{prop:Lidentification}]; the connection $A+\alpha$ remains a SDCI   if, and only if, $\alpha$ is a \emph{Maurer-Cartan element} of $(\rL^\bullet,D)$  [cf. \eqref{eq:brackets} and  \eqref{eq:complex L D}], i.e.,  
$$
D\alpha+\frac{1}{2}[\alpha,\alpha]=0 \in \rL^2.
$$ 
\end{lemma}
\begin{proof}
Given $\alpha\in\Omega^{1}(\fg_E)$,  the curvature of the connection  $A+\alpha$   is   
$$
F_{A+\alpha}-F_A =D\alpha+\frac{1}{2}[\alpha,\alpha] \in \Omega^2(\fg_E).
$$
Hence $A+\alpha$ still has curvature in $\Omega^2_8(\fg_E)$, if, and only if, the Maurer-Cartan torsion $D\alpha+\frac{1}{2}[\alpha,\alpha]$ lies in $\rI=\langle\Omega^2_8(\fg_E)\rangle$, and therefore it vanishes in the quotient.
\end{proof} 
\begin{notation}
\label{not:A+}
We denote by $\cA^+$ the elements in $\rL^1$ satisfying the Maurer-Cartan condition.
\end{notation}
We introduce the $\rW^2_m$-Sobolev norms  on smooth sections $\Gamma(\fg_E)$: 
\begin{equation} 
\label{eq:sobolev norm}
    \Vert s\Vert_m:=\left(  \sum_{m+1}^k\int_{M}  \vert A^m(s(x))\vert^2\dvol\right)^{\frac{1}2},
    \quad m\in\n{N},
\end{equation}
where  $\nabla^{j}s \in\Gamma(\fg_E\otimes T^{0,j}(M))$, and  $\vert\nabla^{j}s\vert$ is determined by the metric on  $M$ and the fibrewise inner product \eqref{eq:innerProduct} on  $\fg_E$.  

Let $\cA _m$ denote  the space of $\rW_m^2$-connections on $E$, and define 
$\cG _{m+1}$ as the topological group of $\rW_{m+1}^{2}$-gauge transformations. For $m>3$, Sobolev multiplication holds \cite[Appendix~\rI\rI]{Donaldson1990}, hence $\cG _{m+1}$ has the structure of an infinite-dimensional Lie group modeled on a Hilbert space; its Lie algebra is the space of $\rW_{m+1}^{2}$ sections of $\End(\fg_E)$.  Moreover, the gauge group $\cG _{m+1}$  acts smoothly on $\cA _m$, so we denote by  
$$
\cB_m=\cA _m/\cG _{m+1}
$$ 
the Hausdorff orbit space with the quotient topology \cite[Lemma~4.2.4]{Donaldson1990}.  Let $\cM_m\subset\cB_m$  denote the set of moduli of solutions to the selfdual ($\lambda=1$) contact instanton equation \eqref{eq:introInstEquation} and, accordingly, denote by $\cM^\ast_m \subset \cM_m$ the stratum of irreducible elements.

We recall that   a contact instanton $A\in\cA$ defines a complex $(\rL^\bullet,D)$ [cf. \eqref{eq:complex L D}]. The $\rL^k$ spaces consist of smooth sections of vector bundles on $M$ (Proposition \ref{prop:Lidentification}), and we denote their $\rW^2_m$-completions by   $\rL^\bullet_m$.  From the formal adjoint of $D$, 
 $$
 D^{\ast}\colon \rL^{k}_m\to  \rL^{k-1}_{m-1},
 $$
we define the Laplacian $\Delta\colon \rL^k_m\to \rL^k_{m-2}$ as in   \eqref{eq:laplacian}. Denote by $G$  its Green operator, i.e., the inverse of  $\Delta$ on the orthogonal complement  $\ker(\Delta)^\perp$,   the orthogonal projection onto harmonic sections  is denote by
\begin{equation}
    \label{eq:map H}
 \bH\colon \rL^{k}_m\to \Ker(\Delta)\subset\rL^{k}_m   
\end{equation}
finally  set  
\begin{equation}
    \label{eq:delta}
    \delta:=D^{\ast}G.
\end{equation} 
\begin{definition}
\label{def:kuranishi map}
The \textbf{Kuranishi map} is defined by 
\begin{equation}
\label{eq:mapF}
\begin{array}{rrcl}
    \sF\colon & \rL^1_m   &\to&    \rL^1_m \\ 
    &\sF(\alpha)    &:=     & \alpha+\frac{1}{2}\delta[\alpha,\alpha]; 
\end{array} 
\end{equation} 
\end{definition}

\begin{lemma}
\label{lem:map F}
    The Kuranishi   map in Definition \ref{def:kuranishi map}  is invertible near $0\in\rL^1_m $ \eqref{eq:brackets}. 
\end{lemma}
\begin{proof}
    For $m>3$, $\sF$ is a smooth map from the Hilbert space $\rL^1$ to itself. 
    The linearisation of $\sF$ at $0\in\rL^1_m$ is the identity: 
    $$
    \sF'(0)(\beta)
    =(\beta+t\delta[\beta,\beta])\vert_{t=0}=\beta,
    $$
    hence there exists a smooth inverse $\sF^{-1}$ near  $0\in\rL^1_m$. Take $c>0$ sufficiently small, such that
\begin{equation}
\label{eq:Uc}
    U_\epsilon=\{\beta\in\rL^1_m \vert \;\Delta\beta=0,\; \Vert \beta\Vert_m<\epsilon\}\subset\cH^1 \cap \f{Dom} (\sF^{-1}).
\end{equation}
    For $\beta\in U_\epsilon$ and  $\alpha:=\sF^{-1}(\beta)$, i.e., $\beta=\alpha+\frac{1}{2}\delta[\alpha,\alpha]$,  
    \begin{align*}
        0
        &=\Delta\beta=\Delta\alpha+\frac{1}{2}\Delta\delta[\alpha,\alpha]
        =\Delta\alpha +\frac{1}{2}D^\ast[\alpha,\alpha] -\frac{1}{2}D^\ast \bH[\alpha,\alpha]\\
        &=\Delta\alpha +\frac{1}{2}D^\ast[\alpha,\alpha],
    \end{align*}
    so, $\alpha$ is a smooth point, by elliptic regularity.  
\end{proof}
We define \emph{the obstruction map}   of the deformation complex $(\rL^\bullet,D)$ [cf. \eqref{eq:complex L D}]  by: 
\begin{equation}
\label{eq:obstructionmap}
\begin{array}{rcll}
    \Psi         \colon &U_\epsilon   &\to&    \cH^2 \\     &\Psi(\alpha)  &:=     & \bH[\sF^{-1}(\alpha),\sF^{-1}(\alpha)], 
\end{array}  
\end{equation}
where $U_\epsilon$ is defined in \eqref{eq:Uc} and $\cH^2=\Ker(\Delta)\subset \rL^2$.   
\begin{lemma}
\label{lem:Psi(0)}
For $m>3$, let $U_\epsilon$ be a neighbourhood  of $0\in\rL^1_m$ as in \eqref{eq:Uc}, on which the inverse  $\sF^{-1}$  of the map \eqref{eq:mapF} is defined. Then $\sF^{-1}$ maps $\Psi^{-1}(\{0\})$ diffeomorphically to a neighbourhood $W$ of the set 
$$
S=\{\alpha \in\rL^1\vert \alpha\in \Ker(D^\ast)\}\cap \cA^+. 
$$
where $\cA^+$ is defined in Notation \ref{not:A+}. 
\end{lemma} 
\begin{proof}
Given $\alpha\in\rL^k_m$,
\begin{align*}
D\delta(\alpha)&=(\Delta-D^\ast D)G\alpha=(1-\bH)\alpha-D^\ast GD\alpha\\
&=(1-\bH-\delta D)\alpha,
\end{align*}
i.e.,
\begin{equation}
\label{eq:Ddelta}
D\delta=1-\bH-\delta D.
\end{equation}
Now,  we show that $\sF^{-1}(\Psi^{-1}(0))\subset S$. If $\beta\in\Psi^{-1}(0)$, set  $\alpha=\sF^{-1}(\beta)$, and apply $D^\ast$ to   $\beta=\alpha+\frac{1}{2}\delta[\alpha,\alpha]$: 
\begin{align*}
0 &= D^\ast\beta=D^\ast\alpha+\frac{1}{2}(D^{\ast})^2G[\alpha,\alpha]= D^\ast\alpha.
\end{align*}
On the other hand, applying $D$ to $\beta $ and using   \eqref{eq:Ddelta}, we obtain 
\begin{align*}
    0  &  = D\beta = D\alpha+\frac{1}{2}D\delta[\alpha,\alpha]  =   D\alpha+\frac{1}{2}[\alpha,\alpha]  +\frac{1}{2}\Psi(\beta)-\delta[D\alpha,\alpha]\\
       &=D\alpha+\frac{1}{2}[\alpha,\alpha]-\delta[D\alpha,\alpha],
\end{align*}
so,  if $\delta[D\alpha,\alpha]=0,$ then $\sF^{-1}$ maps $\Psi^{-1}(0)$ into $S $. To see  that $\delta[D\alpha,\alpha]=0,$ put $x=\delta[D\alpha,\alpha]$ and  note that 
\begin{align*}
\delta[D\alpha,\alpha]&=\delta[-\frac{1}{2}[\delta,\delta]+\delta[D\alpha,\alpha],\alpha] =\delta[\delta[D\alpha,\alpha],\alpha],
\end{align*}
i.e.,  $x=\delta[x,\alpha]$. Now, for each $m>3$,   there exists $K>0$ such that
$$
\Vert x\Vert_m=\Vert \delta[x,\alpha]\Vert_m\leq K\Vert x\Vert_m\Vert \alpha\Vert_m,
$$
so we can take $c$ small enough that $\Vert \alpha\Vert_m<\frac{1}{K}$, hence  $x=0$, as claimed.\\ 
 
Conversely, let $\alpha\in S $ and set $\beta=\sF(\alpha)$, as above. Since  $D^\ast\beta=D^\ast\alpha=0$, 
\begin{align*}
D\beta&= D\alpha+\frac{1}{2}D\delta[\alpha,\alpha]=D\alpha+\frac{1}{2}(1-\bH-\delta D)[\alpha,\alpha] \\
&=-\frac{1}{2}[\alpha,\alpha]+\frac{1}{2}[\alpha,\alpha]-\bH[\alpha,\alpha]-\delta[D\alpha,\alpha]\\
&=-\bH[\alpha,\alpha].
\end{align*} 
The left-hand  side in the above equality is $D$-exact, whereas  the right-hand side is harmonic, hence 
$
D\beta=-\bH[\alpha,\alpha]=\Psi(\beta)=0.
$ 
Therefore  $\beta$ is harmonic and lies in $\Psi^{-1}(0).$
\end{proof} 
In the context of Lemma \ref{lem:Psi(0)}, as shown in \cite[\S 4.2.3]{Donaldson1990}, the tangent model is characterised as follows  
\begin{proposition}
\label{prop:neighbourhoods}  
Let $[A]\in\cM_m$ be a SDCI gauge modulus, and  denote its isotropy by    
$$
\Gamma_A:=\Aut(G)/Z(G),
$$
where $Z(G)$ is the centre of $G$. Let $c>0$ be small enough so that the inverse $\sF^{-1}$ is defined on $U_\epsilon$ \eqref{eq:Uc}. Then $\Psi^{-1}(0)/\Gamma_A$ is a neighborhood of $[A]$ in $\cM_m$. Furthermore, for $m>3$,  the natural map  $\cM_{m+1}\to\cM_m$ is a homeomorphism, so we may suppress the subscript $m$ in $\cM_m$.
\end{proposition}
\begin{remark}
\label{rem:slice via kuranishi}
The Maurer-Cartan condition $D\alpha+\frac{1}{2}[\alpha,\alpha]=0$ [cf. Lemma \ref{lem:Maurer-Cartan}] is equivalent to 
$$
D(\sF(\alpha))=0 \qandq \bH[\alpha,\alpha]=0.
$$
Hence, by the slice Lemma \ref{lem:slice lemma}, a chart about $[A]\in \cM^\ast$ is given by $\sF(S_{A,\epsilon}^+)$, where  $S_{A,\epsilon}^+:=S_{A,\epsilon}\cap\cA^+$.    
\end{remark}
\begin{corollary}
\label{cor:ModuliSuave}
If $[A]\in\cM^\ast$ is an irreducible SDCI such that the obstruction map \eqref{eq:obstructionmap} vanishes, then  $\cM^{\ast}$ is a smooth manifold near $[A]$.
\end{corollary}

\subsubsection{Cohomological vanishing of obstruction}
\label{sec:vanishing}
We now obtain sufficient cohomological conditions for the vanishing of the  obstruction map \eqref{eq:obstructionmap}, and thus the local smoothness  of  $\cM^\ast$  \eqref{eq:modulispace}. The transverse Laplacian $\Delta_T$ from \eqref{eq:laplacianT} satisfies, by ellipticity, 
$$
\rL^k=\ker(\Delta_T)\oplus\ker(\Delta_T)^\perp.
$$
Denote by $\bH_T\colon \rL^k\to\ker(\Delta_T) $ the projection onto $\Delta_T$-harmonic sections, by $G_T$ the Green operator of $\Delta_T$,  and set $\delta_T:=D_T^\ast G_T$.  
By the same argument applied to  \eqref{eq:mapF} in Lemma \ref{lem:map F}, the map
$$
\sF_T \colon\alpha\in  \rL^1_m \mapsto \alpha+\frac{1}{2}\delta_T[\alpha,\alpha]\in   L^1_m, 
$$
is an isomorphism near $0\in  \rL^1_m$, for $m>3$. Taking $c>0$ small enough, we set  
\begin{equation}
    \label{eq:Uc transverse}
U_\epsilon=\{\gamma\in\cH_T^1\;\vert\;\Vert \gamma\Vert_{\rL^2_m}<\epsilon\}    
\end{equation}
 such that $\sF^{-1}_T$ is defined on $U_\epsilon$. 

\begin{proposition}
\label{prop:vanishingH2}
    At an irreducible   SDCI   such that $H_B^2=0$, the obstruction $\Psi$  in \eqref{eq:obstructionmap}  vanishes.
\end{proposition} 
\begin{proof} 
By \eqref{eq:complexDeformation} and Corollary \ref{coro:H1=Hb}, infinitesimal contact instanton deformations are parametrised by $H^1 =H^1_B$. Fix $c>0$ small enough that $\sF^{-1}_T$ is defined on $U_\epsilon$, as in \eqref{eq:Uc transverse}. Since $D_T=0$ on $\rL^2_{m-1}, $ $\Delta_T$ simplifies to  $\Delta_T=D_TD_T^\ast-D_V^2$, and,  using the fact that $D_V$ commutes with the Green operator, we  obtain 
\begin{align*}
D_T\delta_T&=D_TD_T^\ast G_T=(\Delta_T+D_V^2)G_T\\
&=(1-\bH_T)+G_TD_V^2.
\end{align*}
Furthermore, the assumption and Lemma \ref{lem:previo1}  together imply that $0= H_B^2 \cong \cH^2_{T} $, so $\bH_T$ vanishes  on $\rL^2_{m-1}$, and  the above equation simplifies to 
\begin{equation}
\label{eq:8}
D_T\delta_T=1+G_TD_V^2.
\end{equation}
For $\gamma\in U_\epsilon$, let $\alpha:=\sF_T^{-1}(\gamma)$, so 
$\gamma=\alpha+\frac{1}{2}\delta_T[\alpha,\alpha]$. Since   $\gamma$   is $\Delta_T$-harmonic, it follows from \eqref{eq:DT-harmonic} that
$$
0=D_V\gamma=D_T\gamma=D_T^\ast\gamma.
$$ 
Now,  since   $D_V$ commutes with $D_T$ \eqref{eq:DVcomuts DT} and also with  $D_T^\ast$ and $G_T$, a similar argument to Lemma \ref{lem:Psi(0)}) shows that $\delta_T[D_V\alpha,\alpha]=0$, and so  
\begin{align*}
0 &=D_V\gamma=D_V\alpha+\frac{1}{2}D_V\delta_T[\alpha,\alpha]\\
   &=D_V\alpha+\frac{1}{2}  D_V D_T^\ast G_T[\alpha,\alpha] 
      =D_V\alpha+ \delta_T[D_V\alpha,\alpha] \\
   &=D_V\alpha.
 \end{align*}
Thus  $D_T\alpha=D\alpha$,  and we conclude from \eqref{eq:8} that 
\begin{align*}
0&=D_T\gamma=D_T\alpha+\frac{1}{2}D_T\delta_T[\alpha,\alpha]=D\alpha+ \frac{1}{2}( 1+G_TD_V^2)[\alpha,\alpha]\\
 &= D\alpha+ \frac{1}{2}[\alpha,\alpha],
\end{align*}
 i.e., $\alpha$ is a  Maurer-Cartan element [cf.  Lemma \ref{lem:Maurer-Cartan}].  
Furthermore, since   $D_T^2=-\omega\wedge D_V $ [cf. \eqref{eq:DT2}] and $D_V^\ast=-D_V$ [cf. \eqref{eq:adjoint DV}], we have    
$(D_T^\ast)^2=D_V\Lambda(\cdot)$, where $\Lambda$ is the formal  adjoint   of   $L_\omega $ \eqref{eq:mapa eta omega}. This implies 
\begin{align*}
0&=D_T^\ast\gamma=D_T^\ast\alpha+\frac{1}{2}D_T^\ast\delta_T[\alpha,\alpha]\\
 &= D_T^\ast\alpha+\frac{1}{2}(D_T^\ast)^2G_T[\alpha,\alpha]=D_T^\ast\alpha+\frac{1}{2}D_V\Lambda(G_T[\alpha,\alpha]) \\
 &=D_T^\ast\alpha.
\end{align*}
In summary,  $\alpha=\sF^{-1}_T(\gamma)\in S\cap \cA^+=\{\alpha\in\rL^1_m \vert D^\ast\alpha=0,\;D\alpha+\frac{1}{2}[\alpha,\alpha]=0 \}$, and it follows that   
$\sF_T^{-1}$ defines a map from $U_\epsilon$ into $S $. In Lemma \ref{lem:Psi(0)} it was shown that $\sF$  maps a neighborhood of $0\in S $ to a neighborhood of $\Psi^{-1}(0)$, hence, for $c>0$ small enough, 
\begin{align*}
\Psi\circ \sF\circ \sF_T^{-1}(\alpha)&=\Psi(\sF \sF_T^{-1}(\alpha))=\bH[\sF^{-1}\sF\sF_T^{-1}(\alpha),\sF^{-1}\sF\sF_T^{-1}(\alpha)]\\
&=\bH_T[\sF_T^{-1}(\alpha),\sF_T^{-1}(\alpha)]\\
&=0,
\end{align*}
since the linearisation of $\sF\circ \sF^{-1}$ at $0$ is the identity. Therefore  $\Psi=0$ on $U_\epsilon$, for sufficiently small $c$ .
\end{proof}

If $A$ is a contact instanton, the \emph{transverse index} of $A $ is defined as the index of the basic complex in \eqref{eq:BasicComplexDeformation2}, namely, 
\begin{equation}
    \label{eq:ind_T}
\indT(A)= \text{dim}(H^0_B)-\text{dim}(H^1_B)+\text{dim}(H^2_B).
\end{equation}
When  $A$ is irreducible we have $ H^0_B =\{0\}$.
If moreover $H^2_B=\{0\}$, it follows  from Proposition  \ref{prop:vanishingH2} and Corollary \ref{cor:ModuliSuave} that $\cM^\ast$ is locally a smooth manifold of dimension   computed by the transverse index  \eqref{eq:ind_T}: 
$$
\dim \cM^\ast= \dim H^1_B=-\indT (A).
$$  

\newpage
\section{Geometric structures on the contact instanton moduli space \texorpdfstring{$\cM^\ast$}{Lg}}
\label{sec:Geometry}
Arguably the most attractive prospect in studying the moduli space of  contact instantons  is the future construction of topological invariants for Sasakian $7$-manifolds, so we are particularly interested in geometric structures on $\cM^\ast$. By means of comparison, we know from \cite[\S~4.3]{Baraglia2016} that, under suitable assumptions, the moduli space of ASD contact instantons on a Sasakian $5$-manifold $(M^5,\eta,\xi,g,\Phi)$ is K\"ahler, and moreover hyper-K\"ahler in the transverse Calabi-Yau case. In the same spirit, we will show that $\cM^\ast$ carries a natural K\"ahler strucutre. 
\subsection{Riemannian metric on   \texorpdfstring{$\cM^\ast$}{Lg}}
We  define a Riemannian  metric on the (smooth stratum of the) moduli space of SDCI $\cM^\ast$  following a standard approach, e.g. \cites{itoh1988geometry,Baraglia2016}. From the slice Lemma \ref{lem:slice lemma}, we know that neighborhoods of a smooth point $[A]\in \cM^\ast \subset \cB^\ast$ have the form $\pi(S_{A,\epsilon})\subset\cB^\ast$, where  $\pi\colon\cA\to\cB$ is the quotient map  and
$$  
S_{A,\epsilon}=\{\alpha\in\rL^1 \vert \Vert\alpha\Vert_k<\epsilon \}\cap \Ker(D_A^\ast).
$$ 
In view of Proposition \ref{prop:Lidentification}, $\rL^1\cong\Omega^1(\fg_E)$ can be endowed with the inner product  \eqref{eq:innerProduct}: 
$$
(\alpha,\beta)=\displaystyle\int_M\langle\alpha\wedge\ast_T\beta\rangle\wedge \eta.
$$  
  Since $(\cdot, \cdot)$ is $\cG$-invariant, it induces naturally an inner product on the quotient space $\cB$, described explicitly on each neighbourhood $S_{A,\epsilon}$ as follows. By elementary Hodge theory, each   tangent space has  a unique orthogonal decomposition:
$$
  T_A\cA^\ast\cong\Omega^1(\fg_E)
  =\im(D_A)\oplus\Ker(D^\ast_{A}).
$$
  For a basic $1$-form $\beta\in \Omega^1_B(\fg) $, let us denote, accordingly, $\beta=\beta^v+\beta^h,$   with   $\beta^v\in \im(D_{A})$  and $\beta^h\in \Ker(D^\ast_{A}) $ (Figure \ref{fig:metricSlice}). We may now define a positive semi-definite inner product on $\Omega^1(\fg_E)$  which, a priori, depends on the choice of slice neighborhood:
  \begin{equation}
 \label{eq:inner product in TA}
  (\alpha,\beta)_{A}:=(\alpha^h,\beta^h), 
\end{equation} 
  \begin{figure}[h]
    \centering
    \includegraphics[width=0.5\textwidth]{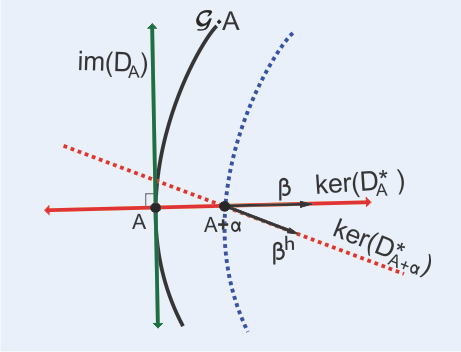} 
    \caption{\small{Slice-dependent inner product on $\cB^\ast$ obtained from $\Ker(D_{A+\alpha}^\ast)$.}} 
    \label{fig:metricSlice}
 \end{figure}
 
Fix a slice $S_{A,\epsilon}$. On the tangent space $T_{A+\alpha}S_{A,\epsilon}$ at  $A+\alpha\in S_{A,\epsilon}$, an inner product is defined by the restriction of \eqref{eq:inner product in TA}:
\begin{equation}
\label{eq:inner product in TAalpha}
    (\beta_1,\beta_2)_{A+\alpha}
    =(\beta_1^h,\beta_2^h), 
    \qwhereq 
    \beta_1^h,\beta_2^h\in \Ker(D_{A+\alpha}^\ast).    
\end{equation}
  Although the choice of the slice $S_{A,\epsilon}$ is not unique, we nonetheless have the following:
  
\begin{proposition}
\cite[Proposition~3.1]{itoh1988geometry}
  The restriction to a slice neighborhood of the inner product \eqref{eq:inner product in TA} defines a positive definite inner product on $\cB^\ast$, which is independent of the choice of slice neighborhoods.   
\end{proposition}
 \begin{proof}[Sketch of Proof]
  For clarity, we  outline the main ideas of the proof, referring the reader to the original reference for the full argument. Given another slice   $S_{A',\epsilon'}$ such that $\emptyset\neq \pi(S_{A,\epsilon})\cap\pi(S_{A',\epsilon'}) \subset\cB$, then there is a transition  smooth map 
  $$
  g\colon \alpha\in S_{A,\epsilon}\cap\pi^{-1}(\pi(S_{A',\epsilon'}))\mapsto g_\alpha\in \cG,
  $$
  satisfying   $g_\alpha(A+\alpha)=A'+\alpha'\in S_{A',\epsilon'}.$   This induces a diffeomorphism between slice neighborhoods  
  $$
  \begin{array}{rcl}
    \Psi\colon \quad S_{A,\epsilon}\cap\pi^{-1}\left(\pi(S_{A',\epsilon'})\right)
    &\to& S_{A',\epsilon'}\cap\pi^{-1}\left(\pi(S_{A ,\epsilon })\right)\\
      \Psi(A+\alpha) &:= & g_\alpha(A+\alpha)  
  \end{array} 
  $$
  whose derivative is 
  $$
  \begin{array}{rcl}
    \Psi_\ast\colon \quad T_{A+\alpha}S_{A,\epsilon}&\to& T_{A'+\alpha'}S_{A',\epsilon'}  \\
     (\Psi_\ast)_{A+\alpha}(\beta)&= & g_\alpha\beta+D_{A'+\alpha'}\gamma,
  \end{array}
  $$
  where $\gamma:=\dtzero \gamma_t\in  \Omega^0(\fg_E) $ for some curve $\gamma_t\in\Omega^0(\fg_E)$ with $\gamma_0=0$.   Hence $\Psi_\ast(\beta)^h=g_\alpha(\beta)$. Since $(\cdot,\cdot)$ in \eqref{eq:inner product in TA} is gauge-invariant, the induced inner product \eqref{eq:inner product in TAalpha} does not depend on the choice of the slice neighborhood.
  \end{proof}
  
Proceeding as in Lemma \ref{lem:map F},  the linearisation of the Kuranishi map    is given by: 
\begin{equation}
\label{eq:derivative kuranishi}
\begin{array}{rcl}
    \sF_{\ast}\colon \quad T_\alpha\left(\Omega^1(\fg_E)\right)
    &\to&
    T_{\sF(\alpha)} \left(\Omega^1(\fg_E)\right)  \\
   (\sF_\ast)_{A+\alpha }(\beta) 
    &:= &
    \beta+\delta[\alpha,\beta]
\end{array}
\end{equation}
where $\delta:=D^\ast_A G$ was defined in \eqref{eq:delta}. Hence, the tangent space to a slice neighborhood at $A+\alpha$ is characterised as 
  \begin{equation}
      \label{eq:slice at A+alpha}
      T_{A+\alpha}(S_{A,\epsilon}^+)=\{\beta\in\Omega^1(\fg_E) \vert D^\ast_{A}\beta=0, D_{A}\beta+[\alpha,\beta]=D_{A+\alpha}\beta=0  \},
  \end{equation}
where $S_{A,\epsilon}^+:=S_{A,\epsilon}\cap \cA^+$ [cf. Notation \ref{not:A+}] denotes the elements in $\rL^1$ satisfying the Maurer-Cartan condition. Taking $\alpha=0$ in  \eqref{eq:slice at A+alpha}, we see that $ T_{A}(S_{A,\epsilon}^+)=\cH^1$.
\begin{proposition}
\label{prop:Differntial kuranishi}
    The differential $\sF_{\ast}$ of the Kuranishi Map [cf. Definition \ref{def:kuranishi map}] restricted to $T_{A}(S_{A,\epsilon}^+)$ is given by
    \[
    \begin{array}{rcl}
       \sF_{\ast}\colon \quad T_{A}(S_{A,\epsilon}^+)
       &\to &
       T_{A}(U_\epsilon)  \\
       \beta &\mapsto&  \bH_A\beta 
    \end{array}
    \]
  where we identify $A+\alpha$ with $\alpha$ and $U_\epsilon$   is the neighbourhood on which the inverse $\sF^{-1}$ is defined \eqref{eq:Uc}  and $\bH_A\beta$ is the harmonic part of $\beta$ \eqref{eq:map H}. 
\end{proposition}
  \begin{proof}
    Given $\beta\in T_{A}(S_{A,\epsilon}^+)$, from \eqref{eq:slice at A+alpha} we have $D_A\beta=-[\alpha,\beta]$. On the other hand, $D^\ast_A$  commutes with the green operator $G_A$, so the restriction of \eqref{eq:derivative kuranishi} to  $T_{A}(S_{A,\epsilon}^+)$ takes the following form: 
    \begin{align*}
        \sF_{\ast}(\beta)
        &=\beta+\delta[\alpha,\beta]=\beta-D^\ast_A G D_A\beta\\
        &=\beta-D^\ast D_A G \beta= \beta-\Delta_A G \beta+ D_A D^\ast_A G \beta\\
        &=\bH\beta+ D_A  G D^\ast_A\beta\\
        &=\bH\beta 
        \qedhere
    \end{align*}  
  \end{proof}
  Following the argument of \cite[Proposition~3.3]{itoh1988geometry},  we split $\beta\in T_\alpha(S_{A,\epsilon}^+)$ like in \eqref{eq:inner product in TAalpha}  to see that $\beta^h\in \cH_{A+\alpha}^1$, then the inner product $(\cdot,\cdot)$ on $T_\alpha(S_{A,\epsilon}^+)$ is written as 
  \begin{equation}
      \label{eq:inner product in TA by harmonics}
      (\beta,\beta^1)_{A+\alpha}=(\bH_\alpha\beta,\bH_\alpha\beta^1).
  \end{equation}
  Thus, in view of   \eqref{eq:inner product in TA by harmonics},  we can compute   the inner product \eqref{eq:inner product in TAalpha} of elements $\alpha$ and $\beta$  in the deformation space  $H^1(\rL)$  by taking their  harmonic representatives in $\cH^1_T$. 
\subsection{Complex structure    on  \texorpdfstring{$\cM^\ast$}{Lg}}

\subsubsection{A natural almost-complex structure}

As a straightforward application of the transverse Kähler identities [Lemma \ref{lem:deltaB caudrado}], we can show the following result,  which will lead to a well-defined complex structure on $\cM^\ast$ [cf. Definition \ref{def:complex structure on M} below]:
\begin{lemma}
\label{lem:iso harmonic-basic}
    Let $H^1_B(\fg_E)$ be the first cohomology group of the basic complex \eqref{eq:BasicComplexDeformation2}, and  
    denote by
    $$
    \Delta_{\bar{\partial}_A}
    := \bar{\partial}_A\bar{\partial}_A^\ast+\bar{\partial}_A^\ast\bar{\partial}_A\colon \Omega^{p,q}_B(M,\fg_E)\to\Omega^{p,q}_B(M,\fg_E)
    $$ 
    the Laplacian associated to $\bar{\partial}_A$ acting on basic forms, and  by 
    $ 
    \cH^{p,q}_{\bar{\partial}_A}
    :=\Ker \Delta_{\bar{\partial}_A}\subset\Omega^{p,q}_B(M,\fg_E)
    $ 
    its space of basic $\bar{\partial}_A$-harmonic sections.
    Then there exists an isomorphism:
  $$  
  H^1_B(\fg_E)\cong\cH^{0,1}_{\bar{\partial}_A}\oplus\ol{\cH^{0,1}_{\bar{\partial}_A}}.
  $$  
\end{lemma}
\begin{proof}
    We decompose a  basic $1$-form as follows:   
    $$  
    \alpha= \ol{\beta^{0,1}}+\gamma^{0,1}
    \in\Omega^{1,0}_B(M,\fg_E)\oplus \Omega^{0,1}_B(M,\fg_E)=\Omega^1_B(M,\fg_E). 
    $$  
    Recall from Lemma \ref{lem:previo1} that there is an isomorphism $H^1_B(\fg_E)\cong\cH^{1,0}_T$, and that, acting on basic forms, $D_T=D\vert_{\Omega^k_B(\fg_E)}=D_B$. Therefore, a basic $1$-form $\alpha$ is $D_T$-harmonic if, and only if,   $D^\ast\alpha=0 $ and $D\alpha=0$, firstly  from
\begin{align*}
  0 &=D\alpha 
     = (\partial_A+ \bar{\partial}_A) (\ol{\beta^{0,1}}+\gamma^{0,1})\\ 
    &=\partial_A\ol{\beta^{0,1}}+ \bar{\partial}_A\gamma^{0,1}
    + (\partial_A\gamma^{0,1} + \bar{\partial}_A\ol{\beta^{0,1}}),    
\end{align*}
  we obtain  
\begin{equation}\label{eq:iso harmonic-basic 1'}
\bar{\partial}_A\gamma^{0,1}=0,  \quad \bar{\partial}_A\beta^{0,1}=0, 
\end{equation} 
  and
\begin{equation}
    \label{eq:iso harmonic-basic 1}  \Lambda(\partial_A\gamma^{0,1}+\bar{\partial}_A\ol{\beta^{0,1}})=0.
    \end{equation} 
  Secondly, applying  the transverse Kähler identities [cf. Lemma \ref{lem:deltaB caudrado}] to $\bar{\partial}_A^\ast\gamma^{0,1}$ and $\partial_A^\ast\ol{\beta^{0,1}}$, respectively we obtain:
    \begin{align*}
        \bar{\partial}_A^\ast\gamma^{0,1} &=\ii(\partial_A\Lambda\gamma^{0,1} - \Lambda\partial_A\gamma^{0,1}) = - \ii \Lambda\partial_A\gamma^{0,1},    \\
        \partial_A^\ast\ol{\beta^{0,1}} &= \ii(\bar{\partial}_A\Lambda\ol{\beta^{0,1}} - \Lambda\bar{\partial}_A\ol{\beta^{0,1}}) 
        = \;\;\ii \Lambda\bar{\partial}_A\ol{\beta^{0,1}}.
    \end{align*}
  hence, using  the above equalities, the condition $0=D^\ast\alpha$ becomes 
\begin{align*} 
  0 &=D^\ast\alpha 
      =\bar{\partial}_A^\ast\ol{\beta^{0,1}} +  \partial_A^\ast\gamma^{0,1}
        + (\bar{\partial}_A^\ast\gamma^{0,1} + \partial_A^\ast\ol{\beta^{0,1}})  \\
    &= \bar{\partial}_A^\ast\gamma^{0,1} + \partial_A^\ast\ol{\beta^{0,1}}  
     =\ii \Lambda\bar{\partial}_A\ol{\beta^{0,1}} - \ii \Lambda\partial_A\gamma^{0,1}\\ 
    &=\Lambda(\bar{\partial}_A\ol{\beta^{0,1}}-\partial_A\gamma^{0,1})    
\end{align*}
  by comparing the last equality with   \eqref{eq:iso harmonic-basic 1}   we concluded that  $ \Lambda \partial_A\gamma^{0,1} =0$ and $\Lambda \bar{\partial}_A\ol{\beta^{0,1}}=0 $, and therefore  $\bar{\partial}_A^\ast\gamma^{0,1}=0$ and $\bar{\partial}_A^\ast\beta^{0,1}=0$. In summary, from \eqref{eq:iso harmonic-basic 1'} and the last assertion 
$$
  \begin{array}{cl}
  \bar{\partial}_A^\ast\gamma^{0,1}=0,  &  \bar{\partial}_A\gamma^{0,1}=0\\[4pt]
  \bar{\partial}_A^\ast\beta^{0,1}=0, &\bar{\partial}_A\beta^{0,1}=0.
  \end{array}
$$
  i.e., $\alpha=\gamma^{0,1}+\ol{\beta^{0,1}}\in \cH^{0,1}_{\bar{\partial}_A}\oplus\ol{\cH^{0,1}_{\bar{\partial}_A}}.$
\end{proof}
  We seek to define a complex structure on the tangent space $T_{[A]}\cM^\ast$. We recall its identification $T_{[A]}\cM^\ast \cong H^1(\rL) $ with the first cohomology group of the complex \eqref{eq:complex L D} [See Theorem \ref{thm:Intro} item $(ii)$] which, in turn, is isomorphic to the first basic cohomology group $H_B^1$ [see Corollary \ref{coro:H1=Hb}], and,  finally, that  $H_B^1\cong \cH^{1,0}_T$  [see Lemma \ref{lem:previo1}]:
  $$  
  T_{[A]}\cM^\ast\cong H^1(\rL)\cong H_B^1\cong \cH^{1,0}_T.
  $$ 
\begin{definition}
\label{def:complex structure on M}
    We define an almost complex structure  $\cJ$  on $T_{[A]}\cM^\ast\cong H_B^1$ to be that one  induced by the   complex structure $J:=\Phi\vert_\rH$,
    where $\rH$ is the horizontal bundle induced by  the Sasakian structure of $M$ [cf. \eqref{eq:TangentDecomposition}]. 
\end{definition}

  To see that the almost complex structure $\cJ$ in Definition \ref{def:complex structure on M} is well defined, it suffices to check that  the space of  $D_T$-harmonic $1$--forms is closed under the action of $J$, i.e., $J(\cH^{1,0}_T)\subset\cH^{1,0}_T$, in view of the identification  $H^1_B\cong \cH^{1,0}_T$. 
  Indeed, from the isomorphism $H^1_B(\fg_E)\cong\cH^{0,1}_{\bar{\partial}_A}\oplus\ol{\cH^{0,1}_{\bar{\partial}_A}}$   [cf. Lemma  \ref{lem:iso harmonic-basic}], a basic $1$--form $\alpha\in\Omega^1(\fg_E)$ is $D_T$--harmonic if, and only if, $\alpha=\gamma+\bar{\beta}$, for some $\bar{\partial}_A$--harmonic forms $\alpha,\beta\in \Omega^{0,1}_B(\fg_E)$, therefore $J\alpha\in \cH^{0,1}_{\bar{\partial}_A}\oplus\ol{\cH^{0,1}_{\bar{\partial}_A}}$.

\subsubsection{Integrability of $\cJ$}

For each $A\in \cA$, we defined  operators $D,D_T$ and $D_V$, cf. \eqref{eq:DV DT}. Analogously, to a deformation   $A+\alpha\in S_{A,\epsilon}^+$ there correspond operators $D_{A+\alpha},D_{T,A+\alpha}$ and $D_{V,A+\alpha}$, which we abbreviate respectively by  
$$ 
D_{\alpha},D_{T,\alpha}\qandq  D_{V,\alpha}.
$$   
These differentials give rise to cohomology groups, which we denote for clarity by 
$$
  H^k_{\alpha}:=H^k_{A+\alpha}
  \qandq
  H^k_{B,\alpha}:=H^k_{B,A+\alpha}.
$$ 
Furthermore, for $\alpha\in U_\epsilon$ as in \eqref{eq:Uc}, we denote by $\alpha^+:=\sF^{-1}(\alpha)\in S_{A,\epsilon}^+$ its image under the inverse Kuranishi map.  Since  $\cH^1_{T}$ embeds into the affine space $\rL^1\cong \Omega^1(\fg)$ [Proposition \ref{prop:Lidentification}], each tangent vector $\lambda$ at $A$ induces a vector field $\ol{\lambda}$ canonically  extended over $U_\epsilon$. Since $\sF$ is a diffeomorphism from  $S_{A,\epsilon}^+$ onto $U_\epsilon$,  $\ol{\lambda}$ provides a locally defined  vector field over $S_{A,\epsilon}^+$  (see Figure \ref{fig:cartamoduli}). We denote the push-forward of $\ol{\lambda}$ by  
  $$
  \lambda^+:=(\sF^{-1}_\ast)_{\alpha}\left(\ol{\lambda}\right)\in T_{\alpha^+}(T_{A,\alpha}^+).
  $$

\begin{figure}[h]
\centering    
\includegraphics[width=0.7\textwidth]{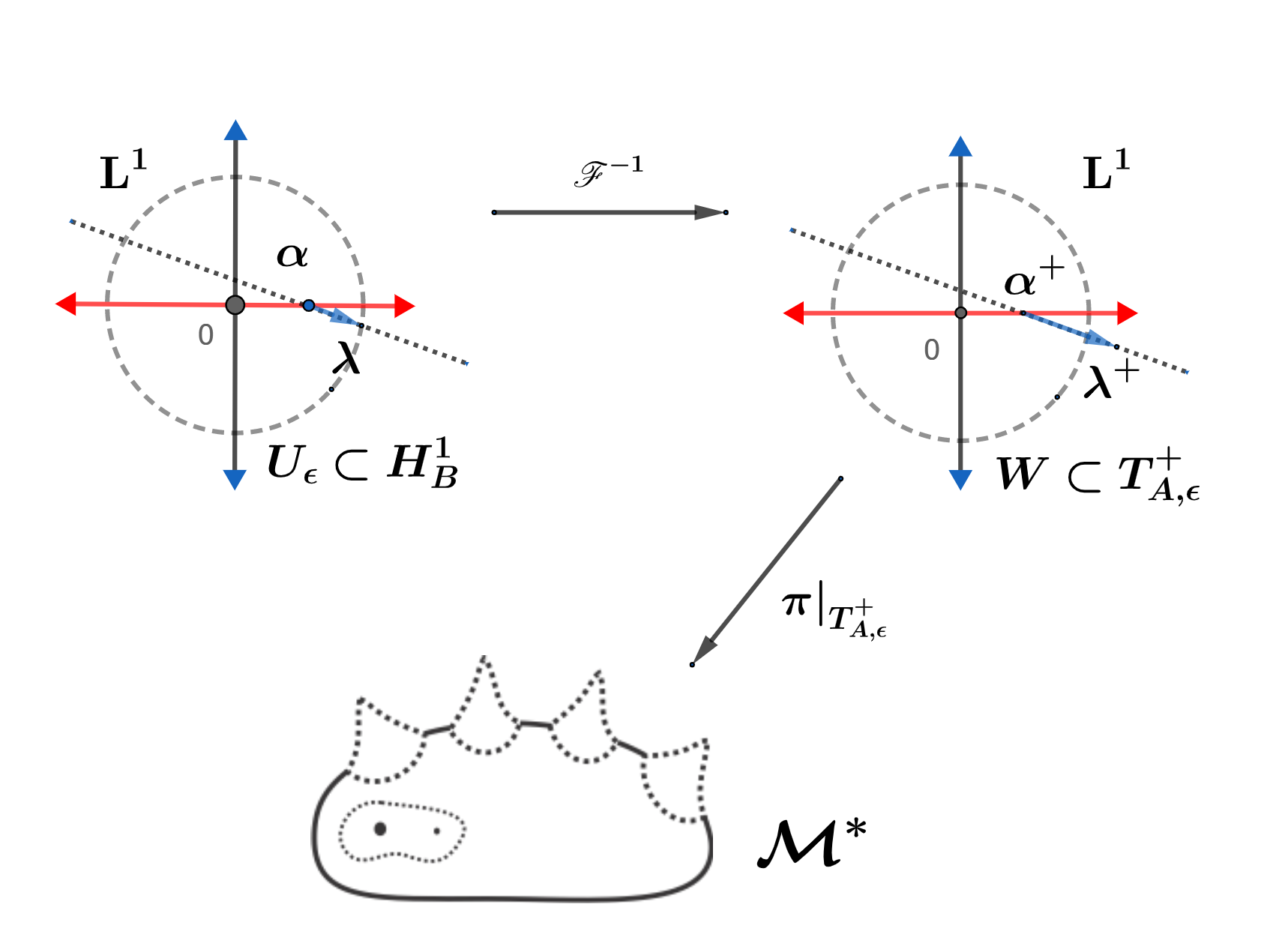} 
\caption{\footnotesize{Chart about $[A]\in \cM^\ast$ modeled by the Kuranishi map on $U_\epsilon\subset\cH^1$, with $\alpha^+=\sF^{-1}(\alpha)$ and $\lambda^+=(\sF^{-1}_\ast)_{\alpha}(\ol{\lambda})$.}}
\label{fig:cartamoduli}
\end{figure}

  The Riemannian metric in  \eqref{eq:inner product in TA by harmonics} is defined in terms of the projection onto the harmonic part, so we need formulas for the harmonic representatives of tangent vectors to the slice $S_{A,\epsilon}^+$. The next results are an adaptation of  \cite[Proposition~4.7]{Baraglia2016} to Sasakian $7$-manifolds, but we include full proofs for the sake of completeness.  
\begin{lemma}
\label{lem:previo Moduli kahler 1} 
    Let $\lambda\in \cH^1_T$ be a tangent vector at $\alpha\in U_\epsilon\subset H^1_B(\fg_E)$, representing a vector field $\ol{\lambda}$ canonically extended to $U_\epsilon$.  
    Then, the harmonic representative  of $\lambda^+:=\sF^{-1}_{\ast\alpha}\left(\ol{\lambda}\right)$ is given by
    \begin{equation}
    \label{eq:previo Moduli kahler 1}
        \gamma=\lambda^+ +D_{T,\alpha^+}f,
    \end{equation}  
    for some  $f(\alpha,\lambda)\in\Omega^0_B(\fg_E)$ which is a solution of 
    \begin{equation}
    \label{eq:solution f 1}
        D^\ast_{T,\alpha^+}\circ D_{T,\alpha^+}(f) +D^\ast_{T,\alpha^+}(\lambda^+)=0.
    \end{equation}  
\end{lemma}
\begin{proof}
  Let $\alpha\in U_\epsilon\subset H^1_B(\fg_E)$. Since $H^1_B(\fg_E)\cong \cH^{1,0}_T$ [cf. Lemma \ref{lem:previo1}] we can assume $\alpha$ to be $D_T$--harmonic. By Definition \ref{def:kuranishi map} of the Kuranishi map, $\alpha^+=\sF^{-1}(\alpha)\in W\subset S_{A,\epsilon}$  satisfies:
  $$
  \alpha=\alpha^+ +\frac{1}{2}\delta_T[\alpha^+,\alpha^+], 
  \qwithq \delta_T=D^\ast_T G_T.
  $$
  Since the derivative of the Kuranishi map at $\alpha=0$ is the identity [cf. Lemma \ref{lem:map F}], we can write   
\begin{equation}
    \label{eq:previo Moduli kahler 2}
  \alpha^+=\alpha-\frac{1}{2}\delta_T[\alpha,\alpha]+ O(\vert\alpha\vert^2),  
\end{equation}
  where $o(\vert\alpha\vert^2)$ has vanishing   $2$-jet at $\alpha=0$. Take $\lambda\in \cH^1_T$ a tangent vector at $A$, by the affine structure of the tangent space, that vector represent also a canonically extended vector field $\ol{\lambda}$ on $U_\epsilon$.   
  Let $c(t)$ be a integral curve of $\lambda$,   then $c(0)=\alpha$ and $\dtzero c(t)=\lambda$. So using \eqref{eq:previo Moduli kahler 2} we can compute  
  \begin{equation} 
  \label{eq:lambda +}
\begin{split}
  \lambda^+ =\sF^{-1}_{\ast\alpha}\left(\ol{\lambda}\right)& =\dtzero\sF^{-1}(c(t))=\dtzero c(t)^+ \\
   & = \dtzero\left (c(t)-\frac{1}{2}\delta_T[c(t),c(t)]+ o(\vert\alpha\vert^2) \right)\\
   & = \lambda-\delta_T[\lambda,\alpha]+ o(\vert\alpha\vert,\lambda),
\end{split}
\end{equation}
  where   $o(\vert\alpha\vert,\lambda)$ has vanishing   $1$--jet at $\alpha=0$, for fixed $\lambda$. Note that, since $\sF^{-1}$ is bijective from $S_{A,\epsilon}$ onto $U_\epsilon$, we get a  vector field $\lambda^+:= (\sF^{-1}_\ast)_{\alpha}(\ol{\lambda})$ along  $S_{A,\epsilon}$. More precisely,  $\lambda^+$  corresponds to a local vector field on $\cM^\ast$ under the natural projection $\pi_{\vert_{S_{A,\epsilon}}}\colon S_{A,\epsilon}\to \pi(S_{A,\epsilon})\subset \cB$ [cf. Lemma \ref{lem:slice lemma}], so  $\lambda^+$ represents a class in  $H^1_{B,\alpha^+}$. Note that $\lambda^+$ is not necessarily  harmonic, so we denote by $\gamma:=\gamma(\alpha,\lambda)$    the harmonic representative in $[\lambda^+]\in H^1_{B,\alpha^+}$.  Hence there exists $f(x,\lambda)\in\Omega^0_B(\fg_E)$ such that  
$$
  \gamma=\lambda^+ +D_{T,\alpha^+}(f).
$$
 Note that  $D^\ast_{T,\alpha^+}\gamma=0$,  because  $\gamma$ is $D_{T,\alpha^+}$--harmonic, so, applying $D^\ast_{T,\alpha^+}$ to the above equality, we see that $f$ indeed satisfies \eqref{eq:solution f 1}. 
\end{proof}
\begin{lemma}
\label{lem:previo Moduli kahler 2}
    Under the assumptions of Lemma \ref{lem:previo Moduli kahler 1}, the following assertions hold:
\begin{enumerate}[(i)]
    \item
    The harmonic representative of $\lambda^+$ can be expressed, to first order, as: 
    \begin{equation}
    \label{eq:repre value}   
        \gamma(\alpha,\lambda)= \lambda-\delta_T[\lambda,\alpha]  +D_Tf(\alpha,\lambda) +o(\vert\alpha\vert,\lambda).
    \end{equation} 
    \item 
    Let $f\in\Omega^0_B(\fg_E)$ be  a solution of \eqref{eq:solution f 1}, then $f$ is the unique solution of the elliptic equation 
    \begin{equation}
    \label{eq:solution f 2}
        \Delta_{T,\alpha^+}(f) + D^\ast_{T,\alpha^+}(\lambda^+) =0.
    \end{equation}
\end{enumerate}
\end{lemma}  
\begin{proof} Item $(i)$ follows by direct computation, using \eqref{eq:lambda +}: 
    \begin{align*}
        \gamma(\alpha,\lambda)&=\lambda^+ + D_{T,\alpha^+}f(\alpha,\lambda)\\
        &=\lambda-\delta_T[\lambda,\alpha] +o(\vert\alpha\vert,\lambda) 
        +D_Tf(\alpha,\lambda)
        +[\lambda^+,f(\alpha,\lambda)]\\
        &=\lambda-\delta_T[\lambda,\alpha] +D_Tf(\alpha,\lambda)+o(\vert\alpha\vert,\lambda).
    \end{align*}
 
    Now for $(ii)$, recall from Definition \ref{def:TransversalLaplacian} that  $\Delta_{T,\alpha^+}=D^\ast_{T,\alpha^+}\circ D_{T,\alpha^+}+D_{T,\alpha^+}\circ D^\ast_{T,\alpha^+}  -D_{V,\alpha^+}^2$. Since $f$ is basic with respect to $A$, and $\alpha^+$ is itself basic with respect to $A$, then $A$ and $A+\alpha^+$ define the same spaces of basic forms, so $f$ is basic with respect to $\alpha^+$, i.e., $D_{V,\alpha^+}f=0$. As a result,  \eqref{eq:solution f 1} can be rewritten as the elliptic equation \eqref{eq:solution f 2}. 
    Denote by $f(\alpha,\lambda)\in \Omega^0_B(\fg_E)$ a solution of  \eqref{eq:solution f 2}, and observe that $f$ is unique for fixed  $\alpha$ and $\lambda$, since any two such solutions differ by an element of $\cH^0_{B,\alpha^+}$, which is zero because $A+\alpha^+$ is irreducible. 
    By standard  elliptic theory,  $f(\alpha,\lambda)$ depends smoothly on $\alpha$. For $\alpha=0$, the derivative  of $\sF^{-1}$ is the identity, so $\lambda^+$ is harmonic and  by uniqueness $f(0,\lambda)=0$. Clearly, $f$  is linear in $\lambda$.
\end{proof}  

We're now in position to show the integrability of $\cJ$ [cf. Definition \ref{def:complex structure on M}], using  Lemma \ref{lem:previo Moduli kahler 1} and the formula \eqref{eq:repre value} for the harmonic representative of $\lambda^+$:
\begin{proposition}
\label{prop:J integrable}
    The almost complex structure $\cJ$ [cf. Definition \ref{def:complex structure on M}] on $\cM^\ast$ is integrable.
\end{proposition}
\begin{proof}
    We  show that the assignment $\alpha\mapsto \alpha^+= \sF^{-1}(\alpha)$ under the Kuranishi map  [cf. Definition \ref{def:kuranishi map}] is a system of normal coordinates for the metric \eqref{eq:inner product in TA} about $[A]$ [cf.  \cite[Theorem~11.6]{moroianu2007lectures}].
    Let  $\lambda_1,\lambda_2\in \cH^1_T$  be tangent vectors and  denote by  $\bar{\lambda}_{i}, i=1,2$ the  canonically extended vector fields on $U_\epsilon$.  
    We determine the $1$--jet of the inner product \eqref{eq:inner product in TA} at $\alpha=0$, in terms of the expression for harmonic representatives in \eqref{eq:repre value}: 
    \begin{align*}
        (\bar{\lambda}_{1},\bar{\lambda}_{2})_A &:=(\bH\lambda_1^+,\bH\lambda_2^+) =(\gamma(\alpha,\lambda_1),\gamma(\alpha,\lambda_2))\\ &=(\lambda_1-\delta_T[\lambda_1,\alpha] +D_Tf(\alpha,\lambda_1),
        \lambda_2-\delta_T[\lambda_2,\alpha] +D_Tf(\alpha,\lambda_2))+o(\vert\alpha\vert)\\
        &=(\lambda_1,\lambda_2) +(\lambda_1,-\delta_T[\lambda_2,\alpha] +D_Tf(\alpha,\lambda_2)) \\
        &+(-\delta_T[\lambda_1,\alpha]+D_Tf(\alpha,\lambda_1), \lambda_2)+o(\vert\alpha\vert)\\
        &=(\lambda_1,\lambda_2) +(D_T\lambda_1,-G_T[\lambda_2,\alpha])
        +(D_T^\ast\lambda_1,f(\alpha,\lambda_2)) \\
        &+(G_T[\lambda_1,\alpha],D_T \lambda_2) 
        +( f(\alpha, \lambda_1),D_T^\ast\lambda_2)+o(\vert\alpha\vert)\\
        &=(\lambda_1,\lambda_2) +o(\vert\alpha\vert).
    \end{align*}
    In the above we have used the fact that each $\lambda_i, i=1,2$ is $D_T$--harmonic, thus \eqref{eq:DT-harmonic} implies $D_T\lambda_i=D_T\lambda_i^\ast=0$. Finally,  for fixed $\lambda_i, i=1,2$, the $1$--jet $o(\vert\alpha\vert)$ vanishes at $\alpha=0$, hence $\sF^{-1}$ defines normal coordinates at $[A]$. 
\end{proof}

\subsection{K\"ahler structure on  \texorpdfstring{$\cM^\ast$}{Lg}}
\label{sec:khaler M}
\begin{definition}
\label{def:kahler form}
    Let $\cJ$ be the complex structure on $\cM^\ast$ from Definition \ref{def:complex structure on M}; using the inner product \eqref{eq:inner product in TA}, we define  a bilinear form on $T\cM^\ast$ by: 
    $$ 
    \Omega_{[A]}(\alpha,\beta):=(\alpha,\cJ\beta).
    $$ 
\end{definition}  
From  Lemma \ref{lem:Omega}, we obtain immediately:
\begin{lemma}
\label{lem:Omega}
  Let  $\alpha$ and $\beta$ be the harmonic representatives of its classes in  $H^1_B(\fg_E)$, then  $\Omega$ is the skew-symmetric  $2$-form associated to $\cJ$ via the inner product \eqref{eq:innerProduct}:
\begin{equation}
\label{eq:Omega}
    \Omega_{[A]}(\alpha,\beta)= 
    \frac{1}{2}\displaystyle\int_M\langle\alpha\wedge\beta\rangle\wedge\eta \wedge\omega^2.
\end{equation}
\end{lemma}

Now, we show that  $\Omega_{[A]}$ is closed using  the 
\begin{proposition}
\label{prop:Omega closed}
  The skew-symmetric form $\Omega$  associated to $\cJ$ [cf. Definitions \ref{def:complex structure on M} and \ref{def:kahler form}] is closed. 
\end{proposition}
\begin{proof}
    In normal  coordinates at $[A]\in\cM^\ast$, given by  the inverse of the Kuranishi map  [cf.  Proposition \ref{prop:J integrable}], we use \eqref{eq:repre value} and  \eqref{eq:Omega}: 
    \begin{align*}
        \Omega_{[A]} (\gamma(\alpha,\lambda_1),\gamma(\alpha,\lambda_2))
        =&\; \frac{1}{2}\displaystyle\int_M\langle\lambda_1\wedge\lambda_2\rangle \wedge\eta\wedge\omega^2
        - \int_M\langle\lambda_1\wedge\delta_T[\lambda_2,\alpha]\rangle \wedge\eta\wedge\omega^2\\  
        &- \displaystyle\int_M\langle\delta_T[\lambda_1,\alpha]\wedge\lambda_2\rangle \wedge\eta\wedge\omega^2  +o(\vert\alpha\vert)\\
        =& \displaystyle\int_M\langle\lambda_1\wedge\lambda_2\rangle \wedge\eta\wedge\omega^2 -(D_TJ\lambda_1,G_T[\lambda_2,\alpha])\\
        & +(G_T[\lambda_1,\alpha],D_TJ\lambda_2) +o(\vert\alpha\vert)\\
        =&\displaystyle\int_M \langle\lambda_1\wedge\lambda_2\rangle \wedge\eta\wedge\omega^2 +o(\vert\alpha\vert).
    \end{align*}
    In the last equality we used the fact that $J\lambda_i$, $i=1,2$ is $D_T$--harmonic. At the point $[A]$ we have $\alpha=0$, so the $1$-jet $o(\vert\alpha\vert)$ is constant and $d\Omega=0$. 
  \end{proof}
  It follows from Propositions \ref{prop:J integrable} and \ref{prop:Omega closed} that $(\cJ,\Omega)$ is a Kähler structure on $\cM^\ast$, which concludes the proof of Theorem \ref{thm:M Kahler}.
 \begin{remark}
\label{rem:hyperkahler}
  If $M$ is a $5$-dimensional contact Calabi-Yau manifold,  there exists $\{J_i\}_{i=1}^3$ a hypercomplex structure with associated Kähler forms $\{\omega_i\}_{i=1}^3$ on the distribution $\rH$ \eqref{eq:TangentDecomposition}, so    Proposition \ref{prop:J integrable} can be applied to each $J_i, i=1,2,3$ to  provide  a hypercomplex structures $\{\cJ_i\}_{i=1}^3$ on $\cM^\ast$. Analogously,  Proposition \ref{prop:Omega closed} can be applied to each $ \cJ_i,i=1,2,3$   to obtain  a Sp$(m)$--structure on $\cM^\ast$ ($\dim(\cM^\ast)=4m$). As a result  the moduli space $\cM^\ast$  would be hyper-Kähler in this case [cf.  \cites{Baraglia2016}]. For  $7$-dimensional Sasakian manifold $M$, even in the contact Calabi-Yau case,  we can not hope the existence of a transverse hyper-Kähler structure on the distribution $\rH$ \eqref{eq:TangentDecomposition}, just like in the $5$--dimensional case, because $\rH$ \eqref{eq:TangentDecomposition} should have congruent zero rank module $4$.  
\end{remark}  
\section*{Afterword: upcoming developments}
\label{sec: Afterword}
\subsection*{\texorpdfstring{$3$}{Lg}--Sasakian instantons}
\label{sec:3Sasakian}
Despite the difficulty to endow $\cM^\ast$ with a hyper-Kähler structure  in the general  $7$--dimensional Sasakian case [cf. Remark \ref{rem:hyperkahler}], we present an approach for an interesting special case of Sasaki-Einstein manifolds. 
A Riemannian  $(4n+ 3)$-manifold $(M,g)$  is said to be $3$--\textit{Sasakian} if $n\geq 1$ and the natural cone $(C(M), \bar{g})$ is hyper-Kähler. Equivalently,  Sasakian structures $\{(\xi_j,\eta_j,\Phi_j)\}_{j=1}^3$  on $M$ which   satisfy  `quaternionic relations' \cite[Proposition~1.2.2]{Boyer2008}: 
\begin{equation}
\label{eq:quaternionic relations}
\begin{array}{c}
    g(\xi_j,\xi_k)=\delta_{jk},\quad [\xi_j,\xi_k] = 2\epsilon_{jkl}\xi_l,\medskip\\
    \eta_j(\xi_k) = \delta^{jk},
    \quad 
    \Phi^j(\xi^k)=-\sigma_{jkl}\xi^l, 
    \quad 
    \Phi^j\circ\Phi^k-\xi_j\otimes\eta_k =-\sigma_{jkl}\Phi^l-\delta^{jk}\mathbbm{1}.
\end{array}
\end{equation}    
\begin{conjecture}
    On a $3$--Sasakian $7$-manifold, the `transverse quaternionic structure' \eqref{eq:quaternionic relations} induces a hyper-Kähler structure on $\cM^\ast$.
\end{conjecture}
Since a $3$--Sasakian manifold admits actually a $\n{S}^2$-family  of Sasakian structures, one should expect to obtain a family $\cM^\ast(\tau), \tau\in \n{S}^2$ of moduli spaces of irreducible contact instantons, with respect to the $3$-forms $\eta(\tau)\wedge d\eta(\tau)$, cf. \eqref{eq:contacIns1}. 
It would then be interesting to assess the relations between the moduli spaces in such a family; there is no a priori reason to expect them to be all isomorphic.
\subsection*{An orientation on the moduli space}
Defining an orientation on a moduli space $\cM^\ast$ of   connections on a principal bundle $P\to M$ cut out by some elliptic condition is an essential step in defining enumerative invariants. The recent article   \cite{joyce2020orientations} provides a package of general techniques to obtain a canonical orientation for a large class of such gauge-theoretic problems. This choice occurs as a global section of the principal $ \n{Z}_2$--bundle $\pi^\ast \colon i^\ast(\cO^{E_\bullet}_P)\to \cM^\ast$, pulled back from the orientation bundle $\pi\colon\cO^{E_\bullet}_P\to \cB_P$  under the inclusion $i\colon \cM^\ast\to \cB_P$. In view of that important progress, it should then be relatively straightforward to apply their techniques and prove \cite[Problem~1.3]{joyce2020orientations} for our moduli space of contact instantons.


%
\subsection*{Sufficient conditions for the obstruction vanishing \texorpdfstring{$H^2_B=0$}{Lg}}

In view of Proposition \ref{prop:vanishingH2}, one might apply Bochner-type methods to arrange the vanishing of $H^2_B$, see e.g. \cites{itoh1983moduli}. Then it is not very hard to establish a result along the following lines, which we state here as an announcement, while postponing the proof to a more detailed upcoming work:
\begin{conjecture}
\label{conj: t-sc=0 => H_B^2=0}
Let $\bE\to M$ be a Sasakian $G$-bundle [Definition \ref{def:SasakianoHol}] over a compact   Sasakian $7$-manifold with  positive transverse scalar curvature. Then, at each irreducible selfdual  contact instanton $A$ on $\bE$,  the second basic cohomology group   $H_B^2 $ vanishes.    
\end{conjecture}
In particular, since the transverse scalar curvature on a Sasaki-Einstein manifold is always positive [cf \cite[Proposition~1.1.9]{boyer19983}]: 
\begin{corollary}
The moduli space of irreducible SDCI over a Sasaki-Einstein manifold is smooth. 
\end{corollary}
\subsection*{Computations of transverse index in particular examples}

Determining the dimension of a moduli space of irreducible and unobstructed contact instantons amounts  to  computing the transverse index \eqref{eq:ind_T}. In    \cite[\S~5]{Baraglia2016},  this is performed by replacing the foliated complex \cite[(3.3)]{Baraglia2016} 
$$
\begin{tikzcd}
   0\arrow[r] 
 &\Omega^0_B( \fg_E) \arrow[r,"D_B"]
 & \Omega^1_B( \fg_E) \arrow[r, "D_B"]
 &  \Omega^+_B (\fg_E) \arrow[r] 
 &  0
\end{tikzcd} 
$$
with a complex involving a suitable lifted $\bT$-action on $\fg_E$ \cite[Propositon~2.8]{Baraglia2016}. Hence the dimension of $\cM^\ast $  is computed from the index of the  symbol complex 
$$
\begin{tikzcd}
\pi^\ast(H^\ast\otimes\fg_E) \arrow[rr,"\sigma(Q)(\varsigma)"]
 &
 & \pi^\ast\left(\left(\Lambda^+H^\ast\oplus\R\right)\otimes\fg_E\right)
\end{tikzcd},
$$
associated to the transverse elliptic operator \cite[(5.1)]{Baraglia2016} 
$$
\begin{tikzcd}
Q=(D_T^\ast,D_T) \colon \Omega^1_H(\fg_E) \arrow[r]
 & \Omega^0_H(\fg_E)\oplus \Omega^+_H(\fg_E)
\end{tikzcd}. 
$$
The transverse index can then be computed in several interesting cases (see  \cite[\S\S~5.2 \& 5.3]{Baraglia2016}). We expect the same approach to be applicable to the index of the  symbol complex
$$
\begin{tikzcd}
\pi^\ast(H^\ast\otimes\fg_E) \arrow[rr,"\sigma(\mathbbm{D}_A)(\varsigma)"]
 &
 & \pi^\ast\left(\left(\Lambda^2_{6\oplus 2}H^\ast\oplus\R\right)\otimes\fg_E\right)
\end{tikzcd},
$$
associated to the transverse elliptic operator $\mathbbm{D}_A:=d_7\oplus d^\ast_A$ defined in  \eqref{eq:Euler characteristic}:
$$
\begin{tikzcd}
 \Omega^1_H(\fg_E) \arrow[r,"\mathbbm{D}_A"]
 &\Omega^0_H(\fg_E) \oplus \left(\Omega_{6\oplus 1}^2 \oplus\Omega^2_H\right)(\fg_E).
\end{tikzcd}  
$$ 

\newpage
\appendix
\input{appendixSasakianBundles.tex}

\bibliographystyle{abbrv}  
\bibliography{Bibliografia-2018-06}

\end{document}

%% file: pacotes.tex
\usepackage{amsmath} 
\usepackage{amsthm,amssymb,mathrsfs} 
\usepackage[T1]{fontenc}
\usepackage[utf8]{inputenc}
\usepackage{ae,aecompl}
\usepackage{times}
\usepackage{cite}

\usepackage[colorlinks=true,
            linkcolor=blue,
            urlcolor=cyan,
            citecolor=green]{hyperref}

\usepackage{verbatim}
\usepackage{tikz}
\usepackage[numeric,backrefs]{amsrefs}
\usepackage{enumerate}
\usepackage[cmtip, all]{xy}
\usepackage{booktabs}
\usepackage{pb-diagram}
\usepackage{bbm}
\usepackage[top=2.54cm,bottom=2.54cm,left=2.54cm,right=2.54cm]{geometry}
\usepackage{tikz-cd} 
\numberwithin{equation}{section} 
\usepackage[textsize=footnotesize,textwidth=2.0cm,colorinlistoftodos]{todonotes}
\usepackage[affil-it]{authblk}
 
\usepackage{changepage}

%% file: init.tex
\newcommand{\rd}{{\rm d}}

\newcommand{\rC}{{\rm C}}

\newcommand{\rG}{{\rm G}}
\newcommand{\rH}{{\rm H}}
\newcommand{\rI}{{\rm I}}

\newcommand{\rL}{{\rm L}}

\newcommand{\rS}{{\rm S}}

\newcommand{\rW}{{\rm W}}

\newcommand{\rZ}{{\rm Z}} 

\newcommand{\bi}{{\bf i}}

\newcommand{\bE}{{\bf E}}

\newcommand{\bH}{{\bf H}}

\newcommand{\bT}{{\bf T}}

\newcommand{\cA}{\mathcal{A}}
\newcommand{\cB}{\mathcal{B}}

\newcommand{\cE}{\mathcal{E}}
\newcommand{\cF}{\mathcal{F}}
\newcommand{\cG}{\mathcal{G}}
\newcommand{\cH}{\mathcal{H}}

\newcommand{\cJ}{\mathcal{J}}

\newcommand{\cL}{\mathcal{L}}
\newcommand{\cM}{\mathcal{M}}

\newcommand{\cO}{\mathcal{O}}

\newcommand{\sF}{\mathscr{F}}

\newcommand{\fg}{{\mathfrak g}}

\newcommand{\R}{\mathbb{R}}
\newcommand{\C}{\mathbb{C}}

\newcommand{\SU}{{\rm SU}}

\newcommand{\Aut}{\mathrm{Aut}}

\renewcommand{\det}{\mathop\mathrm{det}\nolimits}

\newcommand{\End}{{\mathrm{End}}}

\renewcommand{\epsilon}{\varepsilon}

\newcommand{\Hol}{\mathrm{Hol}}

\newcommand{\dvol}{\mathop\mathrm{dvol}\nolimits}

\newcommand{\im}{\mathop{\mathrm{im}}}
\renewcommand{\Im}{\mathop{\mathrm{Im}}}

\newcommand{\indT}{\mathop{\mathrm{index}_T}}

\renewcommand{\Re}{\mathop{\mathrm{Re}}}

\newcommand{\tr}{\mathop{\mathrm{tr}}\nolimits}

\newcommand{\zbar}{\overline z}
\newcommand{\ol}[1]{\overline{#1}}
\newcommand{\wt}[1]{\widetilde{#1}}
\newcommand{\cl}[1]{\mathcal{#1}}

\newcommand{\qandq}{\quad\text{and}\quad}
\newcommand{\qforq}{\quad\text{for}\quad}
\newcommand{\qwithq}{\quad\text{with}\quad}
\newcommand{\qwhereq}{\quad\text{where}\quad}

\newcommand{\co}{\mskip0.5mu\colon\thinspace}
\def\<{\mathopen{}\left<}
\def\>{\right>\mathclose{}}
\def\({\mathopen{}\left(}
\def\){\right)\mathclose{}}

\usepackage{multicol, color}

\definecolor{gold}{rgb}{0.85,.66,0}
\definecolor{cherry}{rgb}{0.9,.1,.2}
\definecolor{burgundy}{rgb}{0.8,.2,.2}
\definecolor{orangered}{rgb}{0.85,.3,0}
\definecolor{orange}{rgb}{0.85,.4,0}
\definecolor{olive}{rgb}{.45,.4,0}
\definecolor{lime}{rgb}{.6,.9,0}
\definecolor{green}{rgb}{.2,.7,0}
\definecolor{grey}{rgb}{.4,.4,.2}
\definecolor{brown}{rgb}{.4,.3,.1}

\newtheorem{theorem}{Theorem}[section] 
\newtheorem{proposition}{Proposition}[section] 
\newtheorem{corollary}[proposition]{Corollary}  
\newtheorem{lemma}[proposition]{Lemma}  
 \newtheorem{conjecture}[proposition]{Conjecture} 
\theoremstyle{remark} 
\newtheorem{remark}[proposition]{Remark} 

\theoremstyle{definition}
\newtheorem{definition}[proposition]{Definition}

\newtheorem{notation}[proposition]{Notation}

\hypersetup{linkbordercolor=blue}
\DeclareMathOperator{\Ker}{ker}

\DeclareMathOperator{\Y}{YM}

\newcommand{\f}[1]{\text{#1}}
\newcommand{\dtzero}{\left.\frac d{dt}\right|_{t=0}}

\newcommand{\n}[1]{\mathbb{#1}}
\newcommand{\ii}{{ \textbf{i}}}  

%% file: appendixSasakianBundles.tex
\newpage
\section{Sasakian vector bundles}
\label{apendixA} 

We gather here some general results and definitions on  `holomorphic' vector bundles over Sasakian manifolds. This concept obviously does not make strict sense as in classical complex geometry, but it admits a straightforward adaptation in terms of the transverse complex structure. All results and definitions in this Appendix stem from the original insights in  \cites{Biswas2010}. We adopt the usual  notation for a $(2n+1)$-dimensional Sasakian manifold $(M,\eta,\xi,\Phi,g)$. Standard references for Sasakian geometry are \cites{Boyer2008,blair2010contact}.

\subsection{Differential forms on Sasakian manifolds}
    \label{sec: App differential forms}
 The contact structure induces  an  orthogonal  decomposition of the tangent bundle,
\begin{equation}
\label{eq:TangentDecomposition}
TM= H\oplus\R\cdot\xi=: H\oplus N_\xi,
\end{equation}
where   $\ker(\eta)=: H\subset TM$ is the   distribution of rank $2n$ transverse to the Reeb field $\xi$, and the restriction $J=\Phi\vert _H$ defines an almost complex structure on $H$. We write  indistinctly $\xi\lrcorner \alpha$ or  $i_\xi(\alpha)$ to denote the interior product by $\xi$.
We denote the complexification of the tangent bundle  by $TM_\C:=TM\otimes_{R}\C$ and also, respectively,  $N_\xi^\C$ and $H_\C$.

\begin{definition}
    \label{defFormTransverse}
A differential form $\alpha\in\Omega^{k}(M)$ is called  \emph{transverse} if $i_\xi\alpha=0.$ \ If moreover $i_\xi d\alpha=0$, then $\alpha$ is said to be \emph{basic} (i.e., $S^1$-invariant).
\end{definition} 

Let $\Omega^{k}_{H}(M)=\Gamma(M,\Lambda^{k}H^{\ast})$  denote the space of transverse $k$-forms.  Let  $\alpha\in \Omega^p_H(M)$ a locally defined, transverse    complex  form and   $x\in M$. Since $i_\xi\alpha_x=0$, the evaluation of  $\alpha_x$ on $\Lambda^{p}( (TM_\C)_x)$   is determined  by the values of $\alpha_x$ on  the subbundle  $\Lambda^{p}((H_\C)_x)\subset\Lambda^{p}( (TM_\C)_x)$.  We denote the complexification of the transverse complex structure $\Phi$ by 
\begin{equation}\label{Phi_C}
    \Phi_{x}^{\C}:=\Phi\vert_{ H_x}\otimes_{\R}\C\colon  (H_\C)_x\to  (H_\C)_x.
\end{equation}
Since $(\Phi\vert_{ H_x})^{2}=-\mathbbm{1}$, the eigenvalues of  $\Phi_{x}^{\C}$ are $\pm\bi:=\pm\sqrt{-1}$, and the complexified horizontal distribution splits accordingly:
\begin{equation}
    \label{eq:Hpq}
    (H_\C)_x= H_{x}^{1,0}\oplus  H_{x}^{0,1}.
\end{equation}
For $p,q\geq 0$ we set 
$$ 
H_x^{p,q}:= \Lambda^{p} (H_x^{1,0})^\ast\otimes \Lambda^{q} (H_x^{0,1})^\ast \subset\Lambda^{p+q} (H_{x})^\ast\otimes_{\R}\C,
$$ 
therefore
\begin{equation}
\label{eq:Quebra1}
    \Lambda^{d}(H_\C)_x^\ast
    =\bigoplus_{i=0}^{d} (H_{x}^{i,d-i})^\ast,
    \quad 0\leq d \leq 2n.
\end{equation}
A section $\alpha$ is said to be  of type $(a,d-a)$. If the evaluation of $\alpha$ on $ H_{x}^{i,d-i}$ is zero for all  $i\neq a$. The decomposition \eqref{eq:Quebra1} gives a decomposition into a direct sum of vector bundles  
$$
\Omega^{d}_{H_\C}(M)=\bigoplus_{i=0}^{d} \Omega^{i,d-i}_{H_\C}(M),
$$
where $\Omega^{pq}_{H_\C}(M)$ denotes $\Gamma(M,\Lambda^p(H_\C)^\ast\otimes\Lambda^q(H_\C)^\ast)$. Taking into account the dimension of   $N_\xi$ and $H$,
it follows that
\begin{align*}
    \Lambda^{d}(TM_\C)^\ast 
    &=\Lambda^{d}( (H_\C)^\ast\oplus (N_\xi^\C)^\ast)\\
    &=\Lambda^{d}( H_\C)^\ast\oplus\left( \Lambda^{d-1}( H_\C)^\ast\otimes\Lambda^{1}(N_\xi^\C)^\ast\right) \oplus \left( \Lambda^{d-2}( H_\C)^\ast\otimes\Lambda^{2}(N_\xi^\C)^\ast\right) \oplus\dots\\
    &\cdots\oplus\left( \Lambda^1( H_\C)^\ast\otimes\Lambda^1(N_\xi^\C)^\ast\right) \oplus\Lambda^d(N_\xi^\C)^\ast\\
    &\simeq\Lambda^d( H_\C)^\ast\oplus\left(\eta\otimes\Lambda^{d-1} (H_\C)^\ast\right)
\end{align*}
hence, combining with the decomposition \eqref{eq:Quebra1},
\begin{equation}
\label{eq:decompostion k forms}
    \Omega^{d}(M)
    =\left(
    \bigoplus_{i=0}^{d} \Omega_{H_\C}^{i,d-i}(M) 
    \right) 
    \oplus\left(
    \eta\otimes\left( \bigoplus_{j=0}^{d-1}\Omega_{H_\C}^{j,d-j-1}(M)\right)  
    \right). 
\end{equation}
\begin{proposition}
\label{prop: omega (1,1)}
The transverse  $(1,1)$-form  $\omega:=d\eta\vert_{H}$ is the fundamental symplectic form on $H$.
\end{proposition}
\subsection{Partial connections}
Consider an integrable subbundle  $S\subset TM_\C$, i.e. the sections of $S$ are closed under the Lie bracket. Of course we have in mind the particular case in which $S=N_\xi^\C\subset TM_\C$, but we state the first few definitions in general terms.
\begin{definition}
Consider a complex vector bundle $E\to M$, a \emph{partial connection} on  $E$ in the direction of   $S$ is a smooth  operator 
$D\colon \Gamma(E)\to \Gamma(\Lambda^1 S^{\ast})\otimes \Gamma(E)$, satisfying the  `Leibniz rule' 
$$
D(fs)=fD(s)+q_S(df)\otimes s,  
$$
where the projection $q_S\colon(TM_\C)^{\ast}\to S^{\ast}$ is the dual of the inclusion  $S\hookrightarrow TM_\C$.
\end{definition}

Since the distribution  $S$ is integrable, smooth sections of $\Ker(q_S)$ are closed under the exterior derivation \cite[Section 3.2]{Biswas2010}, this induces an exterior derivative on the smooth sections of $S^{\ast}$:
$$
\hat{d}\colon \Lambda^1 S^{\ast}\to \Lambda^{2}S^{\ast}.
$$
Consider a partial connection $D$ on  $S$ and the operator
$D_1\colon \Gamma(\Lambda^1S^{\ast})\otimes \Gamma(E)\to \Gamma(\Lambda^{2} S^{\ast})\otimes \Gamma(E)$ defined by 
$$
D_1(\theta\otimes s)= \hat{d}\theta\otimes s-\theta\wedge D(s).
$$
Their composition 
\[
  \Gamma(E)\xrightarrow[{}]{D}\Gamma(E)\otimes\Gamma(\Lambda^1S^{\ast})\xrightarrow[{}]{D_1}\Gamma(E)\otimes \Gamma(\Lambda^{2}S^{\ast})
\]
defines a torsion  $D_1\circ D:=K(D)\in \Gamma \left( \Lambda^{2}S^{\ast}\otimes \End(E)\right) $.
This section is called the  \emph{curvature} of  $D$, we will say that $D$ is a  \emph{flat connection} if $K(D)\equiv 0$.
We denote the  \emph{extended anti-holomorphic} $(n + 1)$–dimensional foliation 
\begin{equation}
\label{eq:ExtendedFoliationH}
    \wt{H}^{0,1}:=H^{0,1}\oplus(N_\xi^\C)\subset TM\otimes_{R}\C,
\end{equation}
where $H^{0,1}$ is defined in   \eqref{eq:Hpq}, it is shown in  \cite[Lemma 3.2]{Biswas2010}  that the distribution in   \eqref{eq:ExtendedFoliationH} is integrable. 

\begin{definition}
\label{def:fibradoSasakiano}
A (complex) \emph{Sasakian vector bundle} on a Sasakian manifold  $(M,g,\xi)$ is a pair $(E,D_0)$. Where  $D_0$ is a partial connection in the direction of     subbundle $N_\xi\subset TM$ and  $E\to M$ is a complex vector bundle.
\end{definition}
Notice that  $N_\xi^\C\subset \wt{H}^{0,1},$ therefore,   we can consider    partial connections along $\wt{H}^{0,1}$, which define by restriction a partial connection along $N_\xi^\C$. Furthermore,  $N_\xi$ is a $1-$dimensional foliation on M, so any partial connection along $N_\xi$ is flat. When the context is clear, we denote by  $\bar{\partial}$  a  flat partial connection $D$ along $\wt{H}^{0,1}$ such that $\bar{\partial}\vert_{N_\xi}=D_0$,  and  we abbreviate the notation by $\textbf{E}:=(E,D_0)$. 
\begin{definition} 
    \label{def:SasakianoHol}
A \emph{holomorphic Sasakian vector bundle} on a Sasakian manifold  $(M,g,\xi)$ is a pair  $\cE:=(\bE,\bar{\partial})$, where  $\bE=(E,D_0)$ is a Sasakian vector bundle  and  $\bar{\partial}$ is a flat connection on $E$ along $\wt{H}^{0,1}$ \eqref{eq:ExtendedFoliationH}. 
\end{definition}

\subsection{Hermitian and holomorphic structures}
\label{HolHermitianVectorBundles}

We define a Hermitian structure on $\textbf{E}$ as a smooth Hermitian structure $h$  on $E$ which is compatible with $D_0$:
$$
d(h(s_1, s_2))\vert_{N_\xi}= h(D_0(s_1), s_2) + h(s_1, D_0(s_2))
\qforq
s_1,s_2\in\Gamma(E).
$$
A unitary connection on $(\textbf{E}, h)$ is a connection $A$ on $E$ such that $d_A$ preserves $h$ in the usual sense.

A connection $A\in \cl{A}(E)$ induces a partial connection along $\widetilde{H}^{0,1}$ \eqref{eq:ExtendedFoliationH} given by $D_{\widetilde{H}^{0,1}}:=d_A\vert_{\widetilde{H}^{0,1}}$. If it  coincides with $\bar{\partial}$, then $A$ is called a \emph{integrable connection} on $\cE$.  Let $\cA(\cE)\subset\cA(E)$ denote the subset of integrable connections on $\textbf{E}$. The class of connections mutually compatible with the holomorphic structure and the Hermitian metric is the natural analogue of the concept of \emph{Chern connection}.

\begin{proposition}
    \label{prop:ChernConnection}
Let $(\cl{E},\bar{\partial})$ be a holomorphic Sasakian bundle with Hermitian structure, then
there exists a unique unitary and integrable Chern connection $A_h$ on $\cl{E}$ and $F_{A_h}\in \Omega^{1,1}$, moreover the    expression
$$
\det\left(\mathbbm{1}_E+\frac{i}{2\pi}H_{A_h}\right) =\sum_{j=0}^{n}c_j(\cl{E},h) 
$$ 
 defines closed Chern forms $c_j(\cl{E},h)\in\Omega^{j,j}(M).$
\end{proposition}

%% file: SasakiInstantons.bbl
\begin{bibdiv}
\begin{biblist}

\bib{Baraglia2016}{article}{
      author={Baraglia, D.},
      author={Hekmati, P.},
       title={Moduli spaces of contact instantons},
        date={2016},
     journal={Advances in Mathematics},
      volume={294},
       pages={562\ndash 595},
}

\bib{bedulli2007ricci}{article}{
      author={Bedulli, L.},
      author={Vezzoni, L.},
       title={{The Ricci tensor of SU (3)-manifolds}},
        date={2007},
     journal={Journal of Geometry and Physics},
      volume={57},
      number={4},
       pages={1125\ndash 1146},
}

\bib{Biswas2010}{article}{
      author={Biswas, I.},
      author={Schumacher, G.},
       title={{Vector bundles on {S}asakian manifolds}},
        date={2010},
     journal={Advances in Theoretical and Mathematical Physics.},
      volume={14},
      number={2},
       pages={541\ndash 562},
}

\bib{blair2010contact}{book}{
      author={Blair, D.~E.},
       title={{Riemannian geometry of contact and symplectic manifolds}},
   publisher={Springer Science and Business Media},
        date={2010},
}

\bib{boyer19983}{article}{
      author={Boyer, C.~P.},
      author={Galicki, K.},
       title={{3-Sasakian manifolds}},
        date={1998},
     journal={arXiv preprint hep-th},
        ISSN={9810250/},
}

\bib{boyer2001einstein}{article}{
      author={Boyer, C.~P.},
      author={Galicki, K.},
       title={{Einstein manifolds and contact geometry}},
        date={2001},
     journal={Proceedings of the American Mathematical Society},
       pages={2419\ndash 2430},
}

\bib{Boyer2008}{book}{
      author={Boyer, C.~P.},
      author={Galicki, K.},
       title={Sasakian geometry},
      series={Oxford Mathematical Monographs},
   publisher={Oxford University Press, Oxford},
        date={2008},
        ISBN={978-0-19-856495-9},
      review={\MR{2382957 (2009c:53058)}},
}

\bib{Calvo-Andrade2016}{article}{
      author={Calvo-Andrade, Omegar},
      author={D{\'\i}az, L{\'a}zaro O.~Rodr{\'\i}guez},
      author={Earp, Henrique N.~S{\'a}},
       title={{Gauge theory and $\rm{G}_2$-geometry on Calabi-Yau links}},
        date={2020},
     journal={Rev. Mat. Iberoam.},
      volume={36},
      number={6},
       pages={1753\ndash 1778},
}

\bib{Donaldson1990}{book}{
      author={Donaldson, S.~K.},
      author={Kronheimer, P.~B.},
       title={The geometry of four-manifolds},
      series={Oxford Mathematical Monographs},
   publisher={The Clarendon Press Oxford University Press},
     address={New York},
        date={1990},
        ISBN={0-19-853553-8},
        note={Oxford Science Publications},
      review={\MR{MR1079726 (92a:57036)}},
}

\bib{Donaldson1998}{incollection}{
      author={Donaldson, S.~K.},
      author={Thomas, R.~P.},
       title={Gauge theory in higher dimensions},
        date={1998},
   booktitle={The geometric universe ({O}xford, 1996)},
   publisher={Oxford Univ. Press},
     address={Oxford},
       pages={31\ndash 47},
         url={http://www.ma.ic.ac.uk/~rpwt/skd.pdf},
      review={\MR{MR1634503 (2000a:57085)}},
}

\bib{kacimi1990operateurs}{article}{
      author={El~Kacimi-Alaoui, A.},
       title={{Op{\'e}rateurs transversalement elliptiques sur un feuilletage
  riemannien et applications}},
        date={1990},
     journal={Compositio Mathematica},
      volume={73},
      number={1},
       pages={57\ndash 106},
}

\bib{Freed1991}{book}{
      author={Freed, D.~S.},
      author={Uhlenbeck, K.~K.},
       title={Instantons and four-manifolds},
     edition={Second},
      series={Mathematical Sciences Research Institute Publications},
   publisher={Springer-Verlag},
     address={New York},
        date={1991},
      volume={1},
        ISBN={0-387-97377-X},
      review={\MR{MR1081321 (91i:57019)}},
}

\bib{habib2015some}{article}{
      author={Habib, G.},
      author={Vezzoni, L.},
       title={{Some remarks on Calabi--Yau and hyper-K{\"a}hler foliations}},
        date={2015},
     journal={Differential Geometry and its Applications},
      volume={41},
       pages={12\ndash 32},
}

\bib{itoh1983moduli}{article}{
      author={Itoh, M.},
       title={{On the moduli space of anti-self-dual Yang-Mills connections on
  K{\"a}hler surfaces}},
        date={1983},
     journal={Publications of the Research Institute for Mathematical
  Sciences},
      volume={19},
      number={1},
       pages={15\ndash 32},
}

\bib{itoh1988geometry}{article}{
      author={Itoh, M.},
       title={{Geometry of anti-self-dual connections and Kuranishi map}},
        date={1988},
     journal={Journal of the Mathematical Society of Japan},
      volume={40},
      number={1},
       pages={9\ndash 33},
}

\bib{joyce2020orientations}{article}{
      author={Joyce, D.},
      author={Tanaka, Y.},
      author={Upmeier, M.},
       title={{On orientations for gauge-theoretic moduli spaces}},
        date={2020},
     journal={Advances in Mathematics},
      volume={362},
       pages={106957},
}

\bib{moroianu2007lectures}{book}{
      author={Moroianu, A.},
       title={{Lectures on K{\"a}hler geometry}},
   publisher={Cambridge University Press},
        date={2007},
      volume={69},
}

\bib{SaEarp2009}{thesis}{
      author={S{\'a}~Earp, H.~N.},
       title={{Instantons on $\rG_2$--manifolds}},
        type={Ph.D. Thesis},
   publisher={PhD thesis},
        date={2009},
}

\bib{sa2014generalised}{article}{
      author={S{\'a}~Earp, H.~N.},
       title={{Generalised Chern--Simons Theory and $\rG_2$-Instantons over
  Associative Fibrations}},
        date={2014},
     journal={Symmetry, Integrability and Geometry: Methods and Applications},
      volume={10},
      number={0},
       pages={83\ndash 11},
}

\bib{Tian2000}{article}{
      author={Tian, G.},
       title={Gauge theory and calibrated geometry. {I}},
        date={2000},
        ISSN={0003-486X},
     journal={Annals of Mathematics. Second Series.},
      volume={151},
      number={1},
       pages={193\ndash 268},
         url={http://dx.doi.org/10.2307/121116},
      review={\MR{MR1745014 (2000m:53074)}},
}

\bib{tondeur2012foliations}{book}{
      author={Tondeur, P.},
       title={{Foliations on Riemannian manifolds}},
   publisher={Springer Science and Business Media},
        date={2012},
}

\end{biblist}
\end{bibdiv}
